\documentclass[sort&compress,3p]{elsarticle}

\usepackage{amssymb}
\usepackage{amsmath}
\usepackage{bm}
\usepackage{amsthm}
\usepackage{caption}
\usepackage{subcaption}
\usepackage{comment}
\usepackage{xcolor}
\usepackage{soul}
\usepackage[ruled,vlined]{algorithm2e}
\usepackage{setspace}

\def\NN{\mathbb{N}}
\def\RR{\mathbb{R}}

\newcommand{\bfm}[1]{\bm{#1}}
\newcommand{\mc}[1]{\mathcal{#1}}
\newcommand{\wt}[1]{\widetilde{#1}}
\newcommand{\defset}[2]{\left\{ #1: \ #2 \right\}}
\newcommand{\sprod}[2]{\left\langle#1, #2\right\rangle}
\newcommand{\norm}[1]{\| #1 \|}
\newcommand{\C}[1]{{C}^{#1}}
\newcommand{\G}[1]{{G}^{#1}}

\newcommand{\bE}{\boldsymbol{\mc{E}}}

\newcommand{\pd}{d}
\newcommand{\pdt}{\pd_{\td}}
\newcommand{\pdw}{\pd_{\w}}
\newcommand{\cF}{\bfm{C}}

\newcommand{\poly}{\mathbb{P}}

\newcommand{\bern}[2]{B^{#1}_{#2}}
\newcommand{\bernT}[2]{B^{\mbox{{\tiny{$\triangle$}}}, #1}_{#2}}
\newcommand{\bernB}[2]{{B}^{#1}_{#2}}

\newcommand{\w}{\omega}
\newcommand{\td}{\tau}
\newcommand{\tr}{\theta}
\newcommand{\dr}[1]{#1'}
\newcommand{\q}{\gamma}

\newcommand{\set}[1]{\left\{ #1 \right\}}

\newcommand{\parDomT}{\triangle_0}
\newcommand{\parDomQ}{\Box_0}
\newcommand{\parDer}[2]{D_{#1}^{#2}}
\newcommand{\jac}[1]{\mbox{J} {#1}}
\newcommand{\orth}[1]{{#1}^{\perp}}

\newcommand{\cf}{{\rm cf}}

\newtheorem{theorem}{Theorem}
\newtheorem{proposition}[theorem]{Proposition}
\newtheorem{lemma}[theorem]{Lemma}
\newtheorem{corollary}[theorem]{Corollary}
\newtheorem{remark}{Remark}
\newtheorem{example}{Example}

\usepackage{lineno}

\journal{}

\begin{document}

\begin{frontmatter}

\title{$\C{1}$-smooth isogeometric spline functions of general degree over planar mixed meshes: The case of two quadratic mesh elements} 

\author[UL,UR]{Jan Gro\v{s}elj}
\ead{jan.groselj@fmf.uni-lj.si}
\author[CUAS]{Mario Kapl}
\ead{m.kapl@fh-kaernten.at}
\author[UL,UR]{Marjeta Knez} 
\ead{marjetka.knez@fmf.uni-lj.si}
\author[RICAM]{\corref{cor}Thomas Takacs}
\ead{thomas.takacs@ricam.oeaw.ac.at}
\author[UP]{Vito Vitrih}
\ead{vito.vitrih@upr.si}
\address[UL]{FMF, University of Ljubljana, Jadranska 19, 1000 Ljubljana, Slovenia}
\address[UR]{IMFM, Jadranska 19, 1000 Ljubljana, Slovenia} 
\address[CUAS]{ADMiRE Research Center, Carinthia University of Applied Sciences, Europastra\ss{}e 4, 9524 Villach, Austria}
\address[RICAM]{Johann Radon Institute for Computational and Applied Mathematics, Austrian Academy of Sciences, Altenberger Str. 69, 4040 Linz, Austria}
\address[UP]{UP FAMNIT and UP IAM, University of Primorska, Glagolja\v ska 8, 6000 Koper, Slovenia}

\cortext[cor]{Corresponding author.}

\begin{abstract}
Splines over triangulations and splines over quadrangulations (tensor product splines) are two common ways to extend bivariate polynomials to splines. However, combination of both approaches leads to splines defined over mixed triangle and quadrilateral meshes using the isogeometric approach. Mixed meshes are especially useful for representing complicated geometries obtained e.g. from trimming. As (bi-)linearly parameterized mesh elements are not flexible enough to cover smooth domains, we 
focus in this work on the case of planar mixed meshes parameterized by (bi-)quadratic geometry mappings. In particular we study in detail the space of $\C{1}$-smooth isogeometric spline functions of general polynomial degree over two such mixed mesh elements. We present the theoretical framework to analyze the smoothness conditions over the common interface for all possible configurations of mesh elements. This comprises the investigation of the dimension as well as the construction of a basis of the corresponding $\C{1}$-smooth isogeometric spline space over the domain described by two elements. Several examples of interest are presented in detail.
\end{abstract}


\end{frontmatter}

\section{Introduction}
\label{sec:intro}

Planar triangle or quadrilateral meshes are two common concepts for the modeling of the geometry of complicated planar domains. To solve fourth order partial differential equations (PDEs) such as the biharmonic equation, e.g.~\cite{TaDeQu14}, the Kirchhoff--Love shell problem, e.g.~\cite{KiBlLi09,KiBaHs10}, problems of strain gradient elasticity, e.g.~\cite{gradientElast2011,KhakaloNiiranenC1}, or the Cahn--Hilliard equation, e.g.~\cite{GoCaBa08}, via their weak form and Galerkin discretization over these meshes, globally $\C{1}$-smooth functions are needed. The construction of globally $\C{1}$-smooth spaces over triangle or quadrilateral meshes has been of interest since the origin of the finite element method (FEM) and has gained even more importance since the introduction of isogeometric analysis (IGA)~\cite{HuCoBa05,CoHuBa09}. Many of the developed methods, in particular in the framework of IGA, employ the fact that a function is $\C{1}$-smooth over a given mesh if and only if the associated graph surface is $\G{1}$-smooth~\cite{GrPe15}, i.e. possessing a uniquely defined tangent plane at each point~\cite{Pe02}. To simplify the construction and to make it independent of the geometry of the mesh, some existing approaches require that the constructed $\C{1}$-smooth functions are additionally $\C{2}$-smooth at the vertices.     
 
Two first $\C{1}$-smooth triangular finite elements have been the Argyris element~\cite{ArFrSc68} and the Bell element~\cite{Be69}, see also~\cite{Ci02,BrSc07}, where in both cases, $\C{1}$-smooth splines over linearly parameterized triangles are constructed, which are polynomial functions of degree~$\pd \geq 5$ on the individual triangles. The construction of $\C{1}$-smooth triangular spline spaces of lower polynomial degree often relies on the use of triangle meshes with specific configurations or splitting of the triangles, cf. the book~\cite{LaSc07}. Examples of recently developed $\C{1}$-smooth triangular splines are~\cite{GrMa22,GrSp21,SpMaPeSa12,Sp13,JaQi14}.
   
A first quadrilateral $\C{1}$-smooth finite element construction over bilinear meshes has been the Bogner--Fox--Schmit element~\cite{BoFoSc65}, which works for polynomial degree~$\pd \geq 3$, but which is limited to tensor-product meshes. Examples of $\C{1}$-smooth finite elements over general bilinear quadrilateral meshes are the Brenner--Sung element~\cite{BrSu05} for $\pd \geq 6$ and the constructions~\cite{Ma01,BeMa14,KaSaTa20} for $\pd \geq 5$. While the methods~\cite{BrSu05,KaSaTa20} generate $\C{1}$-smooth spline functions which are additionally $\C{2}$-smooth at the vertices, the obtained spline functions in~\cite{Ma01,BeMa14} are in general just $\C{1}$-smooth everywhere.

In the framework of IGA, $\C{1}$-smooth spline spaces over quadrilateral meshes are generated, where the individual quadrilateral patches need not be bilinearly parameterized. Depending on the employed multi-patch parameterization of the considered planar quadrilateral mesh, different strategies for the construction of $\C{1}$-smooth spline spaces have been developed, cf. the survey articles~\cite{KaSaTa18,HuSaTaTo21}. Examples of proposed parameterizations are $\C{1}$-smooth parameterizations with singularities~\cite{NgPe16,ToSpHu17} or $\G{1}$-caps~\cite{KaNgPe17,KaPe17,KaPe18} at the extraordinary vertices, analysis-suitable $\G{1}$ multi-patch parameterizations~\cite{KaViJu15,KaBuBeJu16,KaSaTa17,KaSaTa19}, which form a particular class of regular $\C{0}$ multi-patch geometries~\cite{CoSaTa16,KaSaTa17b}, or other general multi-patch configurations~\cite{ChAnRa18}.

The recent paper~\cite{GrKaKnTaVi2020} deals with the construction of $\C{1}$-smooth spline spaces over planar mixed triangle and quadrilateral meshes. The use of mixed triangle and quadrilateral meshes is of high practical relevance, since they appear in and can be advantageous for many applications. One important example is the untrimming of trimmed tensor-product splines. There, mixed meshes are beneficial in representing the geometry, cf.~\cite{To22}. The technique~\cite{GrKaKnTaVi2020} is based on a mixed mesh, where the individual triangles and quadrilaterals are linearly and bilinearly parameterized, respectively. It generates $\C{1}$-smooth splines, which are polynomial functions of degree~$\pd \geq 5$ on the single element. A further recent construction of smooth spline spaces over mixed triangle and quadrilateral meshes is the work~\cite{To22}. However, there the spline spaces of polynomial degree~$\pd=2$ are just $\C{0}$-smooth in the vicinity of extraordinary vertices. For the case of purely quadrilateral meshes, the construction from~\cite{To22} has recently been extended in~\cite{TaTo23} to splines that are $\C{1}$ at all extraordinary vertices but still remain $C^0$ in the neighborhood of extraordinary vertices.

The goal of this paper is to study the $\C{1}$-smoothness conditions over mixed planar partitions composed of B\'ezier triangles and tensor-product B\'ezier quadrilaterals of (bi-)degree~$\delta \geq 1$ in an isogeometric setting. This extends the work~\cite{GrKaKnTaVi2020}, where (bi-)linearly parameterized elements have been considered. Over the considered elements of (bi-)degree~$\delta$, mapped polynomial function spaces of some degree~$\pd \geq \delta$ can be defined. A similar study has been performed in~\cite{MoViVi16}, where a general dimension formula and basis construction is presented for $\C{1}$-smooth splines over mixed triangle and quadrilateral meshes. While one can easily generate $\C{0}$-smooth isogeometric spline functions over such partitions, the dimension count and basis constructions for $\C{1}$-smooth spaces become highly nontrivial and, as developed in~\cite{MoViVi16}, requires the computation of generators of syzygy modules for each edge. The respective syzygies are defined through the gluing data of the edge. Being an algebraic approach, it is very general and covers, in principle, any element segmentation and any combination of degrees $\delta$ and $d$. However, in such a setting it is difficult to analyze how the geometry, i.e., the element parameterizations, influences the dimension of the $C^1$-smooth space and its basis structure. In our work we want to provide more geometric insight and especially want to focus on the local polynomial reproduction properties of the space.

In this paper, we focus on a single interface between two elements, which are allowed to be triangular or quadrilateral. Firstly, the space of $\C{1}$-smooth isogometric spline functions defined on two elements of general (bi-)degree~$\delta$ is considered. We investigate the $\C{1}$-smoothness conditions of the functions across the interface of the two elements and analyze their representation in the vicinity of the interface, where we focus on conditions related to the trace and normal derivative along the interface. We then restrict ourselves to the case of quadratic triangles and biquadratic quadrilaterals, i.e., to the case of $\delta=2$, where we further study the structure of the corresponding $\C{1}$-smooth isogeometric spline space, determining its dimension and providing a basis construction for it. We aim at an exhaustive representation covering all cases. While this is of theoretical interest in itself, the study of polynomial reproduction properties of traces and normal derivatives has several practical implications for meshing and refinement, and will, in future research, serve as the basis for the construction and numerical analysis of $\C{1}$-smooth isogeometric spline spaces over partitions composed of multiple elements. A possible application is then to generate $\C{1}$-smooth isogeometric spline spaces over mixed triangle and quadrilateral meshes, that are obtained by untrimming, and to use the resulting spaces to solve fourth order PDEs.

The remainder of the paper is organized as follows. Section~\ref{sec:C1continuity} introduces the class of planar mixed meshes composed of two mesh elements which can be B\'{e}zier triangles and {B\'{e}zier} quadrilaterals of (bi-)degree~$\delta$. We define the associated $\C{1}$-smooth isogeometric spline space and study the $\C{1}$-smoothness condition of an isogeometric spline function across the interface of the two mesh elements. In Section~\ref{sec:quadratic-elements} we restrict ourselves to element mappings of degree two, that is, to quadratic triangles and biquadratic quadrilaterals, and study the specific smoothness conditions over the two mesh elements for all possible cases. The obtained results are summarized in Section~\ref{sec:basis}, where the dimension of the $\C{1}$-smooth isogeometric spline space is presented and a basis is constructed. Section~\ref{sec:examples} further presents several examples of different configurations of the two mesh elements and illustrates the corresponding $\C{1}$-smooth isogeometric basis functions over them. Finally, we conclude the paper in Section~\ref{sec:Conclusion}. Concerning the notation and basic concepts, we mostly follow the recent paper~\cite{GrKaKnTaVi2020}.

\section{$\C{1}$-smooth isogeometric functions over two mixed (triangular and quadrilateral) elements} \label{sec:C1continuity}

In the following we introduce the isogeometric spaces that we consider in this paper and derive the $\C{1}$-smoothness conditions for those spaces. We then describe the $\C{1}$-smoothness in terms of conditions on the traces and normal derivatives for a fixed element interface.

\subsection{Isogeometric space over two mixed elements} 

Let $\Omega \subset \RR^2$ be  a given  open 
domain with a $\C{0}$ boundary, such that its  closure~$\overline{\Omega}$ is the 
union of the closures of triangular or quadrilateral elements~$\Omega^{(\ell)}$,  $\ell \in \{1,2\}$.
The elements are assumed to be open sets having the bijective and regular parameterizations
\[
\bfm{F}^{(\ell)}: \mathcal{D}^{(\ell)} \to \overline{\Omega^{(\ell)}},
\]
where $\mathcal{D}^{(\ell)}$ equals
\[
\parDomT :=\defset{(u,v)\in \RR^2}{u\in \left[0,1\right], \; 0 \leq v \leq 1-u} 
\quad \mbox{ or } \quad \parDomQ:=\left[0,1\right]^2
\]
for the triangular and quadrilateral elements, respectively.
Let us denote by $\poly^1_\delta$, $\poly^2_\delta$ and $\poly^2_{\delta,\delta}$ the univariate, triangle and tensor-product polynomial spaces of (bi-)degree $\delta$, respectively. We assume
\begin{equation} \label{eq:geometry_mapping}
\bfm{F}^{(\ell)} \in
 \begin{cases}
    (\poly^2_{\delta,\delta})^2 & \mbox{ if }    
    \mathcal{D}^{(\ell)}= \parDomQ \\
    (\poly^2_\delta)^2 & \mbox{ if }
    \mathcal{D}^{(\ell)}= \parDomT .
 \end{cases}
\end{equation}
Let $\bE = (\overline{\Omega^{(1)}}\cap \overline{\Omega^{(2)}})^\circ$ be the common interface, and, without loss of generality, we assume it equals
$
\bE = \left\{\bfm{F}^{(1)}(0,t) = \bfm{F}^{(2)}(0,t): t\in \left(0,1\right) \right\}.
$
Here the notation $(\cdot)^\circ$ denotes the edge without endpoints, i.e., the open curve segment interpreted as a one dimensional manifold.

We define a function $\varphi$ over the union of the two elements $\overline\Omega = \overline{\Omega^{(1)}} \cup \overline{\Omega^{(2)}}$ as 
\[ 
\varphi:\overline\Omega \to \RR, \quad
\varphi(x,y) = \begin{cases}
\varphi^{(1)}(x,y), & (x,y)\in \overline{\Omega^{(1)}}\\
\varphi^{(2)}(x,y), & (x,y)\in \overline{\Omega^{(2)}}\setminus \overline{\bE}
\end{cases},
\]
and consider its graph $\Phi \subset \overline\Omega \times \RR$ 
as the union of two patches given by parameterizations
\[
\Phi^{(1)} := \begin{bmatrix}
\bfm{F}^{(1)}\\
f^{(1)}
\end{bmatrix}: \mathcal{D}^{(1)} \to \RR^3,\quad 
\Phi^{(2)} := \begin{bmatrix}
\bfm{F}^{(2)}\\
f^{(2)}
\end{bmatrix}: \mathcal{D}^{(2)} \to \RR^3,
\]
where $f^{(\ell)} = \varphi^{(\ell)} \circ \bfm{F}^{(\ell)}$. 
The isogeometric space $\mathcal{V}_\pd(\Omega)$ of degree $\pd \geq \delta$ over the domain $\Omega$ is defined as 
\[
 \mathcal{V}_\pd(\Omega) := \left\{ \varphi : \overline{\Omega}\rightarrow \RR , \quad  \varphi \circ \bfm{F}^{(\ell)} = f^{(\ell)} \in
 \begin{cases}
\poly^2_{\pd,\pd} & \mbox{ if } 
\mathcal{D}^{(\ell)}= \parDomQ\\
\poly^2_{\pd} & \mbox{ if } 
\mathcal{D}^{(\ell)}= \parDomT
 \end{cases}, \mbox{ for }\ell \in \{1,2\}
 \right\}.
\]
The space $\mathcal{V}_\pd(\Omega)$ can, in principle, be defined for any (bi-)degree $\pd$. However, it is considered \emph{isogeometric} and reproduces, in general, linear functions only if $\pd\geq \delta$. We define the $\C{1}$-smooth isogeometric space to be $\mathcal{V}^1_d(\Omega) := \mathcal{V}_d(\Omega) \cap \C{1}(\overline{\Omega})$. In the following we study the continuity conditions that describe this subspace.

\subsection{Continuity conditions} 

It is well known (\cite{KaViJu15,GrPe15,CoSaTa16}) that along the common interface the function $\varphi$ is $\C{1}$ continuous if and only if its graph $\Phi$ is $\G{1}$ continuous. The later is true if and only if
\begin{equation} \label{eq:G1cond-1}
\Phi^{(1)}(0,v) = \Phi^{(2)}(0,v), \quad \det
\left[
\parDer{u}{}{\Phi^{(2)}}(0,v),\,  \parDer{u}{}{\Phi^{(1)}}(0,v),\,  \parDer{v}{}{\Phi^{(1)}}(0,v)
\right]= 0.
\end{equation}
We introduce the {\it gluing functions} for the interface $\bE$,
\begin{equation} \label{def-alpha}
\wt{\alpha}_1(v):=  \det \jac{\bfm{F}^{(1)}}(0,v), \quad 
\wt{\alpha}_2(v):=  \det \jac{\bfm{F}^{(2)}}(0,v), \quad
\alpha(v):= \det\left[\parDer{u}{}{\bfm{F}^{(2)}}(0,v), \, \parDer{u}{}{\bfm{F}^{(1)}(0,v)}
\right].
\end{equation}
Here $\jac{\bfm{F}} = [\parDer{u}{}{\bfm{F}},\parDer{v}{}{\bfm{F}}]$ denotes the Jacobian of the mapping $\bfm{F}$.
Let $\q = {\rm gcd}(\wt{\alpha}_1,\wt{\alpha}_2)$ be the (polynomial) greatest common divisor of polynomials $\wt{\alpha}_1$ and $\wt{\alpha}_2$, and let
\[
\alpha_\ell := \frac{1}{\q} \wt{\alpha}_\ell, \quad \ell \in \{1,2\}.
\]
Since the polynomial gcd is not unique, we assume without loss of generality $\q(0)=1$. Then, condition \eqref{eq:G1cond-1} is equivalent to
\begin{linenomath}
\begin{align} 
& f^{(1)}(0,v) = f^{(2)}(0,v), \label{eq:geom-cont-0}\\
& \q(v){\alpha}_1(v) \parDer{u}{}{f^{(2)}}(0,v) - \q(v){\alpha}_2(v) \parDer{u}{}{f^{(1)}}(0,v) + \alpha(v) \parDer{v}{}{f^{(1)}}(0,v) = 0,\label{eq:geom-cont}
\end{align}
\end{linenomath}
where~\eqref{eq:geom-cont} follows from performing the Laplace expansion of the determinant in~\eqref{eq:G1cond-1} along the last row. 
Note that $\parDer{v}{}{f^{(1)}}(0,v)=\parDer{v}{}{f^{(2)}}(0,v)$ and that $\q(v)\neq 0$, ${\alpha}_\ell(v) \neq 0$, for $v\in [0,1]$, $\ell \in \{1,2\}$. 

Along the interface $\bE$ we define the vector-valued function 
\[
\bfm{n}:[0,1]\to \RR^2, \quad \bfm{n}(v)= \orth{\left(\parDer{v}{}{\bfm{F}^{(1)}} (0,v)\right)},
\]
where $\orth{(x,y)} := (y,-x)$, 
which prescribes the direction vector at every point of the interface (in the direction of the normal). By defining 
\begin{equation}\label{def-beta}
\beta(v) = \norm{\bfm{n}(v)}^2, \quad 
{\beta_\ell(v) = \sprod{\orth{(\parDer{u}{}{\bfm{F}^{(\ell)}} (0,v))}}{\bfm{n}(v)}, \quad \ell \in \{1,2\},}
\end{equation}
and using \cite[Lemma 1]{GrKaKnTaVi2020}, we get that
\begin{equation}\label{rel-alpha-beta}
\beta(v) \alpha(v) = \q(v) \alpha_2(v) \beta_1(v) - \q(v) \alpha_1(v) \beta_2(v).
\end{equation}
This equality can also be directly checked by noting that 
\begin{equation}\label{def-alpha-1}
\q(v)\alpha_\ell(v) = \det \jac{\bfm{F}^{(\ell)}}(0,v) = \sprod{{\parDer{u}{}{\bfm{F}^{(\ell)}} (0,v)}}{\orth{\left(\parDer{v}{}{\bfm{F}^{(1)}} (0,v)\right)}} = 
\sprod{{\parDer{u}{}{\bfm{F}^{(\ell)}} (0,v)}}{\bfm{n}(v)},
\end{equation}
for $\ell \in \{1,2\}$. With this equality equation \eqref{eq:geom-cont} can be rewritten to obtain
\begin{equation}\label{eq:geom-cont-1}
\alpha_2(v)\left(\beta(v) \parDer{u}{}{f^{(1)}} (0,v)-\beta_1(v) \parDer{v}{}{f^{(1)}} (0,v)\right)=
\alpha_1(v)\left(\beta(v) \parDer{u}{}{f^{(2)}} (0,v)-\beta_2(v) \parDer{v}{}{f^{(2)}} (0,v)\right).
\end{equation}
Moreover, in \cite[Lemma 2]{GrKaKnTaVi2020} it is proven that for a fixed vector $\bfm{d}$ the directional derivative 
$\parDer{\bfm{d}}{}{\varphi^{(\ell)}}(x,y)$ at a point $(x,y) = \bfm{F}^{(\ell)}(u,v)$ is in local coordinates equal to
\[
\begin{split}
{\w}_{\bfm{d}}^{(\ell)}(u,v) & :=  \sprod{\bfm{d}}{ 
\bfm{G}^{(\ell)}(u,v)
},\\
 \bfm{G}^{(\ell)}(u,v)  &:=  \frac{1}{\det \jac{\bfm{F}^{(\ell)}}(u,v)} \left(
\parDer{u}{}{f^{(\ell)}(u,v)} \orth{\left(\parDer{v}{}{\bfm{F}^{(\ell)}}(u,v) \right)}
- \parDer{v}{}{f^{(\ell)}(u,v)} \orth{\left(\parDer{u}{}{\bfm{F}^{(\ell)}}(u,v) \right)}
\right).
\end{split}
\]
We choose $\bfm{d}$ to be the normal vector~$\bfm{n}(v)$, which is orthogonal to the interface at every point and therefore depends on $v$. 
From \eqref{def-alpha} and \eqref{def-beta} it then follows that along the interface, the normal derivative rewrites to  
\begin{equation}\label{equal-wn-0}
{\w}_{\bfm{n}(v)}^{(\ell)}(0,v)  = \frac{1}{\q(v) \alpha_{\ell}(v)}\left(\beta(v) \parDer{u}{}{f^{(\ell)}} (0,v)-\beta_\ell(v) \parDer{v}{}{f^{(\ell)}} (0,v)\right),
\end{equation}
so the $\G{1}$ continuity condition \eqref{eq:geom-cont-1} equals 
\begin{equation} \label{equal-wn}
{\w}_{\bfm{n}(v)}^{(1)}(0,v) = {\w}_{\bfm{n}(v)}^{(2)}(0,v).
\end{equation}
From~\eqref{eq:geom-cont-1} as well as from~\eqref{equal-wn} we see that the polynomial
$ \beta \parDer{u}{}{f^{(\ell)}} (0,\cdot)-\beta_\ell \parDer{v}{}{f^{(\ell)}} (0,\cdot) $ must be divisible by $\alpha_\ell$, and the common normal derivative \eqref{equal-wn} must be a (rational) function of the form 
\begin{equation}\label{eq-w-tilde}
  {\w}_{\bfm{n}} := \frac{\wt{\w}_{\bfm{n}}}{\q}, \qquad \wt{\w}_{\bfm{n}}(v):=\frac{1}{\alpha_{\ell}(v)}\left(\beta(v) \parDer{u}{}{f^{(\ell)}} (0,v)-\beta_\ell(v) \parDer{v}{}{f^{(\ell)}} (0,v)\right), \quad \ell=1,2.
\end{equation}
In what follows, to simplify the terminology, we use the term \emph{normal derivative} also for the $\q$-scaled (polynomial) normal derivative, which we denote, for simplicity, by
\[
 \w(v):={\wt{\w}}_{\bfm{n}}(v) = \q(v)\;{\w}_{\bfm{n}}(v).
\]
Moreover, we denote by 
\[ 
 \tr(v):= f^{(1)} (0,v) = f^{(2)} (0,v)
\]
the trace and by
\begin{equation}\label{eq:def-tau}
 \td(v):= \dr{\tr}(v) = \parDer{v}{}{f^{(1)}} (0,v) = \parDer{v}{}{f^{(2)}} (0,v)
\end{equation}
the tangential derivative of the isogeometric function $\varphi$ along the trace.

To distinguish between the triangular and quadrilateral elements and to simplify further notation, we define
\[
 \sigma_\ell := 
 \begin{cases}
    1 & \mbox{ if } 
    \mathcal{D}^{(\ell)}= \parDomQ\\
    0 & \mbox{ if }
    \mathcal{D}^{(\ell)}= \parDomT.
 \end{cases}
\]
Note that, by definition, the degree of $\tr$ is bounded by $\pd$ and the degree of $\w$ cannot exceed $d+2\delta-2$. This is due to
\[
 \deg(\alpha_\ell) \leq 2\delta - 2+\sigma_\ell, \quad \deg(\beta_\ell) \leq 2\delta - 2+\sigma_\ell, \; \mbox{ and } \; \deg(\beta) \leq 2\delta - 2,
\]
and the fact that the degree of $\w$ is bounded by the degree of the numerator in~\eqref{equal-wn-0}. In the following theorem we characterize the $\C{1}$ conditions in terms of the gluing functions.
\begin{theorem} \label{main-thm-1}
Let $\varphi \in \mathcal{V}_\pd(\Omega)$ be an isogeometric function. Let $\tr \in \poly^1_\pd$ and $\w \in \poly^1_{\pd+2\delta-2}$ be given polynomials that determine the trace function and the normal derivative. We define the polynomial functions
\[
r_\ell := \alpha_{\ell} \,\w + \beta_\ell \, \td,\quad \ell \in \{1,2\},
\]
where $\td=\dr{\tr}$ as defined in~\eqref{eq:def-tau}. Then $\varphi$ is $\C{1}$-smooth if and only if 
\begin{description}
\item[A1:]
$\beta$ divides the polynomials $r_1$ and $r_2$, i.e., there exist polynomials $\eta_\ell$, such that
\[
 r_\ell = \beta \, \eta_\ell, \quad \mbox{ for }\;\ell\in\{1,2\};
\]
\item[A2:]
the degrees of the polynomials $r_1$ and $r_2$ are bounded by 
\[
\deg(r_\ell) \leq d-1+\sigma_\ell + \deg(\beta), \quad \mbox{ for }\; \ell\in\{1,2\};
\]
\end{description}
and that under these conditions the functions $f^{(\ell)} = \varphi \circ \bfm{F}^{(\ell)}$, for $\ell \in \{1,2\}$, satisfy
\begin{equation}\label{eq:expansion-f-ell}
 f^{(\ell)} (u,v) = \tr(v) + u \cdot \eta_\ell(v) + u^2 \cdot R_\ell(u,v),
\end{equation}
where $R_\ell(u,v) \in \poly^2_{\pd-2,\pd}$, if $\mathcal{D}^{(\ell)}= \parDomQ$, or $R_\ell(u,v) \in \poly^2_{\pd-2}$, if $\mathcal{D}^{(\ell)}= \parDomT$. Thus, we have
\begin{equation}\label{eq-D0}
 D_0 :=  \binom{\pd}{2} (2-\sigma_1-\sigma_2) + (\pd-1)(\pd+1)(\sigma_1+\sigma_2)
\end{equation}
degrees of freedom that have no influence on the $\C{1}$ condition at the interface.
\end{theorem}
\begin{proof}
Since $\tr$ is the trace and $\w$ represents the (scaled) normal derivative as in~\eqref{eq-w-tilde}, the continuity conditions~\eqref{eq:geom-cont-0}
and~\eqref{equal-wn} are equivalent to 
\begin{linenomath}\begin{align}
& f^{(\ell)}(0,v)=  \tr(v), \quad \ell \in \{1,2\},\label{cond-Tr}\\
& \beta(v) \parDer{u}{}{f^{(\ell)}} (0,v)= r_\ell(v), 
\quad \ell \in \{1,2\}.  \label{cond-w}
\end{align}\end{linenomath}
Equations \eqref{cond-w} follow from \eqref{equal-wn-0} and equations ${\w}_{\bfm{n}(v)}^{(\ell)}(0,v)=\w(v)$, $\ell \in \{1,2\}$. Then by \eqref{cond-w} $\varphi$ is $\C{1}$-smooth if and only if $\tr$ and $\w$ are chosen such that the two polynomials $r_\ell$, $\ell \in \{1,2\}$, are divisible by the polynomial $\beta$ and the degree of $r_\ell$ is not greater than $d-1+\sigma_\ell + \deg(\beta)$. 
So the quotient  $\eta_\ell$ between $r_\ell$ and $\beta$ is a polynomial of degree 
$d-1+\sigma_\ell$. Equations~\eqref{cond-Tr} and~\eqref{cond-w} directly imply~\eqref{eq:expansion-f-ell}, and~\eqref{eq-D0} collects the dimensions of the spaces of $\poly^2_{\pd-2}$ and $\poly^2_{\pd-2,\pd}$, which concludes the proof.
\end{proof}

We are interested in the dimension of the space $\mathcal{V}_\pd^1(\Omega)$, which is obviously bounded from below by $D_0$. It is however not directly clear how the dimension depends on the underlying geometry.

\subsection{Construction of traces and normal derivatives} 

Let us now examine how one can choose $\tr$ and $\w$, to achieve that polynomials $r_1$ and $r_2$ are both divisible by $\beta$, satisfying condition {\bf A1}, and are of a maximal degree as specified in condition {\bf A2}. To simplify the analysis of the divisibility of certain polynomials by $\beta$, in what follows, we use for a polynomial function~$g$ the simplifying notation
\[
 g = g^* \beta + \hat{g},
\]
where ${g}^* = {\rm quot}(g,\beta)$ is the quotient  and $\hat{g} = {\rm rem}(g,\beta)$ the remainder of $g$ after division by $\beta$, hence, $\deg(g^*) = \deg (g) - \deg (\beta)$ and $\deg(\hat{g})<\deg (\beta)$. Keep in mind that this representation is unique, i.e., for each $g$ there exists exactly one pair $(g^*,\hat{g})$, with $\deg(\hat{g})<\deg (\beta)$, and vice versa.

Here and in what follows we use the notation 
\begin{equation*} 
 \beta \; | \; \sigma,
\end{equation*}
to state that the polynomial $\beta$ divides the polynomial $\sigma$, which means that there exists a polynomial $q = {\rm quot}(\sigma,\beta)$, such that $\beta \cdot q = \sigma$. This is equivalent to $\hat{\sigma}\equiv0$.
\begin{theorem}\label{main-thm-2}
The assumptions of Theorem~\ref{main-thm-1} are fulfilled if and only if 
the trace function $\tr \in \poly^1_\pd$ and the normal derivative $\w \in \poly^1_{\pd+2\delta-2}$ are chosen as
\begin{equation}\label{def-trAndw-Theorem1}
\tr(v) = \tr_0 + \int_{0}^v \left(\td^*(\xi) \beta(\xi) + \hat{\td}(\xi)\right) d\xi, \quad 
\w(v) = \w^*(v) \beta(v) + \hat{\w}(v),
\end{equation}
where the functions $\td^*$ and $\w^*$ are split in low degree contributions $\td^*_{low}$ and $\w^*_{low}$ and high degree contributions $\td^*_{high}$ and $\w^*_{high}$, respectively, such that
\begin{description}
 \item[(1)] the low degree contributions $\td^*_{low}$ and $\w^*_{low}$ satisfy
\begin{equation}\label{eq:deg-t-low}
\deg(\td^*_{low}) \leq d_\td, \mbox{ with } \; d_\td:= \min \left\{ \pd-1-\deg(\beta), \pd-1+ \sigma_1 - \deg(\beta_1),\pd-1+ \sigma_2 - \deg(\beta_2)  \right\},
\end{equation}
and
\begin{equation}\label{eq:deg-w-low}
\deg(\w^*_{low}) \leq \pdw, \mbox{ with } \; \pdw :=\min_{\ell \in \{1,2\}}\{ \pd-1+ \sigma_\ell - \deg(\alpha_\ell) \},
\end{equation}
 \item[(2)] the high degree contributions $\td^*_{high}$ and $\w^*_{high}$ satisfy $\td^*_{high}(v) = v^{d_\td+1} q_\td(v)$ and $\w^*_{high}(v) = v^{\pdw+1} q_\w(v)$, for some polynomials $q_\td$ and $q_\w$, respectively, and together with $\hat{\td}$, $\hat{\w}$ satisfy
\begin{linenomath}\begin{align}
 \deg(\beta \, \td^*_{high} + \hat{\td}) &\leq  \pd-1,  \label{eq:high-and-hat-0} \\
 \deg(\alpha_\ell (\beta \, \w^*_{high} + \hat{\w}) + \beta_\ell (\beta \, \td^*_{high} + \hat{\td})) &\leq \pd-1+ \sigma_\ell + \deg(\beta), \label{eq:high-and-hat-1}
\end{align}\end{linenomath}
 and 
 \item[(3)] the polynomials $\hat{\td}$ and $\hat{\w}$ moreover satisfy 
\begin{equation} \label{eq:beta-div-hat-functions}
 \beta \; | \; \alpha_\ell \hat{\w} + \beta_\ell \hat{\td}.
\end{equation}
\end{description}
\end{theorem}
\begin{proof}
In this theorem we split the tangential derivative $\td$ and the normal derivative $\w$ into 
\begin{itemize}
 \item parts that are divisible by $\beta$ and are of low degree ($\td^*_{low}$ and $\w^*_{low}$), 
 \item parts that are divisible by $\beta$ and are of high degree ($\td^*_{high}$ and $\w^*_{high}$), and 
 \item parts that are not divisible by $\beta$ ($\hat{\td}$ and $\hat{\w}$).
\end{itemize}
This split is unique for each pair $\td$ and $\w$. Hence, we need to check if the conditions stated here are equivalent to the conditions {\bf A1} and {\bf A2} stated in Theorem~\ref{main-thm-1}. 
By definition we have
\begin{linenomath}\begin{align*}
r_\ell = \alpha_\ell \w + \beta_\ell \td & = \alpha_\ell  (\w^*_{low} \beta + \w^*_{high} \beta +\hat{\w}) + \beta_\ell (\td^*_{low} \beta + \td^*_{high} \beta +\hat{\td}) \\
& = \beta (\alpha_\ell (\w^*_{low} + \w^*_{high}) + \beta_\ell  (\td^*_{low} + \td^*_{high})) +
\alpha_\ell \hat{\w} + \beta_\ell \hat{\td}.
\end{align*}\end{linenomath}
Thus, {\bf A1} is equivalent to~\eqref{eq:beta-div-hat-functions}. What remains is to analyze the degree of $r_\ell$. From~\eqref{eq:deg-t-low} and~\eqref{eq:deg-w-low} we obtain the following bounds 
\[
 0 \leq \deg(\beta \, \alpha_\ell \, \w^*_{low}) \leq \deg(\beta) + \deg(\alpha_\ell) + \min_{\ell' \in \{1,2\}}\{ d-1+ \sigma_{\ell'} - \deg(\alpha_{\ell'}) \} \leq d-1+ \sigma_\ell + \deg(\beta)
\]
and
\[
 0 \leq \deg(\beta \, \beta_\ell \, \td^*_{low}) \leq \deg(\beta) + \deg(\beta_\ell) + \min_{\ell' \in \{1,2\}}\{ d-1+ \sigma_{\ell'} - \deg(\beta_{\ell'}) \} \leq d-1+ \sigma_\ell + \deg(\beta).
\]
Therefore, {\bf A2} is equivalent to~\eqref{eq:high-and-hat-1}.

Finally, we need to check that $\tr \in \poly^1_\pd$, which follows from~\eqref{eq:deg-t-low}, i.e.,  $\deg(\td^*_{low}) \leq  d-1-\deg(\beta)$, and from~\eqref{eq:high-and-hat-0}. This concludes the proof.
\end{proof}
\begin{remark}\label{rem:condition-delta-d}
 Note that this theorem requires no relation between the degree $\delta$ of the geometry mappings $\bfm{F}^{(\ell)}$ and the degree $\pd$ of the functions $f^{(\ell)}$. If we assume $\pd \geq 2\delta-2$, then conditions~\eqref{eq:high-and-hat-0}--\eqref{eq:high-and-hat-1} simplify to
\begin{linenomath}\begin{align*}
 \deg(\td^*_{high}) &\leq  d-1-\deg(\beta), \\
 \deg( \alpha_\ell \, \w^*_{high} + \beta_\ell \, \td^*_{high}) &\leq d-1+ \sigma_\ell.
\end{align*}\end{linenomath}
\end{remark}
From now on, we restrict ourselves to the case $\delta=2$. While Theorem~\ref{main-thm-2}~(1) is easy to analyze for arbitrary degree $\delta$, the contributions from parts~(2) and~(3) are significantly more complicated to study.

\section{The case of (bi-)quadratic elements}\label{sec:quadratic-elements}

In the remainder of the paper we analyze the case of (bi-)quadratic element mappings, i.e., $\delta=2$. Thus, we have $\pd\geq 2$ and, as a consequence, the conditions of Theorem~\ref{main-thm-2} simplify as specified in Remark~\ref{rem:condition-delta-d}. Before we analyze the $\C{1}$-smoothness conditions in detail, we introduce the B\'ezier representations of polynomial functions.

\subsection{Control point representation}
Let
\[
\bern{\pd}{i}(u) = \binom{\pd}{i} u^i(1-u)^{\pd-i}, \quad i=0,1,\ldots,\pd,
\]
\[
 \bernT{\pd}{i,j}(u,v) = \frac{\pd!}{i! j! (\pd-i-j)!} u^i v^j(1-u-v)^{\pd-i-j}, \quad i, j = 0,1,\ldots, \pd, \; i+j\leq \pd,
\]
and
\[
\bernB{ \scalebox{0.6}{$\Box$}, \pd}{i,j}(u,v) = \bern{\pd}{i}(u)\bern{\pd}{j}(v), \quad i,j=0,1,\ldots,\pd,
\]
denote the univariate, triangle and tensor-product Bernstein bases for the spaces $\poly^1_\pd$, $\poly^2_\pd$ and $\poly^2_{\pd,\pd}$, respectively. 

For $\mathcal{D}^{(\ell)}= \parDomT$, i.e., for a triangular element, the geometry mapping is given as
\[
\bfm{F}^{(\ell)}(u,v) = \sum_{0\leq i+j\leq 2} \cF_{i,j}^{(\ell)} \bernT{2}{i,j}(u,v),
\]
and for $\mathcal{D}^{(\ell)}= \parDomQ$, i.e., for a quadrilateral element, it is defined as
\[
\bfm{F}^{(\ell)}(u,v) = \sum_{i,j=0}^2 \cF_{i,j}^{(\ell)} \bernB{ \scalebox{0.6}{$\Box$},2}{i,j}(u,v).
\]
The interface $\bE$ is a curve, which is parameterized by a quadratic polynomial,
\begin{equation} \label{eq:edgePar1}
\bE = \left\{ (1-t)^2\cF_{0} + 2t(1-t) \cF_{1}+ t^2 \cF_{2}: t\in \left(0,1\right) \right\},
\end{equation}
with control points $\cF_{i} \in\RR^2$, $i\in\{0,1,2\}$. Thus, given control points
\begin{equation}\label{eq:control-pts-elements}
\cF^{(\ell)} = \begin{bmatrix}
\cF_{0,0}^{(\ell)} & \cF_{0,1}^{(\ell)} & \cF_{0,2}^{(\ell)}\\[1mm]
\cF_{1,0}^{(\ell)} & \cF_{1,1}^{(\ell)} &\\[1mm]
\cF_{2,0}^{(\ell)} &  & 
\end{bmatrix} \quad \mbox{ or } \quad
\cF^{(\ell)} = \begin{bmatrix}
\cF_{0,0}^{(\ell)} & \cF_{0,1}^{(\ell)} & \cF_{0,2}^{(\ell)}\\[1mm]
\cF_{1,0}^{(\ell)} & \cF_{1,1}^{(\ell)} & \cF_{1,2}^{(\ell)}\\[1mm]
\cF_{2,0}^{(\ell)} & \cF_{2,1}^{(\ell)} & \cF_{2,2}^{(\ell)}
\end{bmatrix}
\end{equation} 
for $\ell\in\{1,2\}$, depending if $\mathcal{D}^{(\ell)}= \parDomT$ or $\mathcal{D}^{(\ell)}= \parDomQ$, respectively, the continuity condition $\bfm{F}^{(1)}(0,v) = \bfm{F}^{(2)}(0,v)$ implies that the control points of $\bfm{F}^{(1)}$ and $\bfm{F}^{(2)}$ corresponding to the edge are the same, i.e., 
\[
 \cF_{0} := \cF_{0,0}^{(1)} = \cF_{0,0}^{(2)}, \quad
 \cF_{1} := \cF_{0,1}^{(1)} = \cF_{0,1}^{(2)}, \quad \mbox{and} \quad
 \cF_{2} := \cF_{0,2}^{(1)} = \cF_{0,2}^{(2)}.
\]

\subsection{Smoothness conditions}

The question of divisibility by $\beta$ depends on the degree of $\beta$, which in turn depends on the interface $\bE$, parameterized as in \eqref{eq:edgePar1}. Three different cases can happen:
\begin{description}
\item{Case (a): uniformly parameterized linear interface}. In this case $\cF_{1}  =\frac{1}{2}\cF_{0} + \frac{1}{2}\cF_{2} $,  and $\beta$ is a constant;
\item{Case (b): non-uniformly parameterized linear interface}. In this case $\cF_{1}  = (1-\lambda)\cF_{0} + \lambda \cF_{2} $ for some $\lambda\in (0,1)$, $\lambda\not = \frac{1}{2}$, 
and $\beta \in \poly^1_2$ is the square of a linear polynomial;  
\item{Case (c): parabolic interface}. In this case control points  $\cF_{0}$, $\cF_{1}$ and $\cF_{2}$ are not collinear, and $\beta$ is an irreducible, quadratic polynomial (with a non-zero leading coefficient).
\end{description}
In the next subsection we consider Case (a), which is the simplest case to be analyzed, since in that case $\beta$ is a constant. Cases (b) and (c) are considered in the Subsections~\ref{sec:non-uniform} and~\ref{sec:parabolic}, respectively. The special case of (bi-)linear elements handled in~\cite{GrKaKnTaVi2020} is completely covered by Case (a).

\subsubsection{Uniformly parameterized linear interface} 

When considering a uniformly parameterized linear interface, the gluing data simplifies significantly. We then have that
\[
 \bfm{n}(v)= \orth{\left(\parDer{v}{}{\bfm{F}^{(1)}} (0,v)\right)} =  \orth{\left(\cF_{2}-\cF_{0}\right)}
\]
is a constant vector and consequently $\beta = \norm{\bfm{n}}^2 \in \RR^+$. Moreover, we obtain 
\[
\wt{\alpha}_\ell(v)=  \sprod{\parDer{u}{}{\bfm{F}^{(\ell)}}(0,v)}{\bfm{n}}, \quad 
\beta_\ell(v)=  \sprod{\orth{(\parDer{u}{}{\bfm{F}^{(\ell)}} (0,v))}}{\bfm{n}},
\]
for $\ell \in \{1,2\}$, and it follows directly from the above equations that
\[
\deg(\alpha_\ell) \leq 1 + \sigma_\ell, \quad \deg(\beta_\ell) \leq 1 + \sigma_\ell, \quad \ell \in \{1,2\}.
\]
The following proposition gives sufficient and necessary conditions on the trace $\tr$ and normal derivative $\w$ to fulfill the assumptions {\bf A1}-{\bf A2} of Theorem~\ref{main-thm-1} in the case of uniformly parameterized linear interface. To shorten the
notation, we denote the coefficient of a univariate polynomial $p$ at the power $j$ by $\cf(p;j)$.
\begin{proposition} \label{proposition-linear-edge}
Suppose that we are in Case (a), that is, the interface $\bE$ is a line, parameterized uniformly, and let $\varphi$ be an isogeometric function. We can distinguish two cases:
\begin{description}
 \item[(1)] If
\begin{equation}\label{eq:Proposition3-case1}
 \deg ({\alpha}_{1}\beta_2-\alpha_2\beta_1) = \max(\deg ({\alpha}_{1}) +\sigma_2, \deg(\alpha_2) + \sigma_1) +1,
\end{equation}
 then $\varphi\in \mathcal{V}^1_\pd(\Omega)$ is equivalent to the trace and normal derivative satisfying $\tr \in \poly^1_{\pd-1}$ and $\w\in\poly^1_{\pdw}$,
 \item[(2)] else $\varphi\in \mathcal{V}^1_\pd(\Omega)$ if and only if the trace and normal derivative satisfy $\tr \in \poly^1_{\pd}$, $\w\in\poly^1_{\pdw+1}$, where the leading coefficients of $\tr$ and $\w$ satisfy a linear constraint. 
\end{description}
Here $\pdw$ is defined as in Theorem~\ref{main-thm-2}.
\end{proposition}
\begin{proof}
We follow the structure of Theorem~\ref{main-thm-2}. Since $\beta$ is a constant, it is clear that $\td^* = \td/\beta$, $\w^* = \w/\beta$, $\hat{\td}=0$ and $\hat{\w}=0$. So~\eqref{eq:beta-div-hat-functions} is always satisfied. We then have that $\varphi\in \mathcal{V}^1_\pd(\Omega)$ if and only if $\td^* = \td^*_{low} + \td^*_{high}$ and $\w^* = \w^*_{low} + \w^*_{high}$, where $\td^*_{low} \in \poly^1_{\pd}$, $\w^*_{low} \in \poly^1_{\pdw}$, $\td^*_{high}(v) = v^{d_\td+1} q_\td(v)$ and $\w^*_{high}(v) = v^{\pdw+1} q_\w(v)$, which must satisfy~\eqref{eq:high-and-hat-0}, which reduces to
\[
 \deg(\td^*_{high}) \leq \pd-1,
\]
as well as~\eqref{eq:high-and-hat-1}, which reduces to
\[
 \deg(\alpha_\ell \, \w^*_{high} + \beta_\ell \, \td^*_{high}) \leq d-1+ \sigma_\ell.
\]

Let $\pdt$ be given as in Theorem~\ref{main-thm-2} and assume first that $\pdt=\pd-1$, which is equivalent to $\deg(\beta_1) \leq \sigma_1$ and $\deg(\beta_2) \leq \sigma_2$. Then~\eqref{eq:high-and-hat-0} implies $\td^*_{high}=0$ and, consequently,~\eqref{eq:high-and-hat-1} reduces to
\[
 \deg(\alpha_\ell \, \w^*_{high}) \leq \pd-1+ \sigma_\ell
\]
which implies $\w^*_{high}=0$. Thus, $\td^* = \td^*_{low} \in \poly^1_{\pd-1}$ and $\w^* = \w^*_{low} \in \poly^1_{\pdw}$. This is covered by case (2), where the linear constraint is given as
\[
 \cf(\w;  \pdw+1) = 0. 
\]

Thus, from now on we assume $\pdt<\pd-1$, with the only possible option $\pdt=\pd-2$. Thus $\td^*_{low} \in \poly^1_{\pd-2}$ and~\eqref{eq:high-and-hat-0} implies $\deg(\td^*_{high}) \leq  d-1$. Consequently $\td^*_{high}(v) = v^{d-1} q_\td$, where $q_\td$ is a constant. Moreover,~\eqref{eq:high-and-hat-1} simplifies to
\[
 \deg(\alpha_\ell(v) \, v^{\pdw+1} q_\w(v) + \beta_\ell(v) \, v^{d-1} q_\td) \leq d-1+ \sigma_\ell.
\]
Since $\pdw \geq \pd - 2$, this reduces to
\begin{equation}\label{eq:degree-bound-sigma-ell}
 \deg(\alpha_\ell(v) \, v^{\pdw-d+2} q_\w(v) + \beta_\ell(v) \, q_\td) \leq \sigma_\ell.
\end{equation}
Assume, w.l.o.g., $\deg ({\alpha}_{1}) +\sigma_2 \geq \deg(\alpha_2) + \sigma_1$. We then have $\pdw-d+2 = 1+\sigma_1 - \deg(\alpha_1)$ and~\eqref{eq:degree-bound-sigma-ell} yields for $\ell=1$ that
\[
 \cf\left(\alpha_1(v) \, v^{1+\sigma_1 - \deg(\alpha_1)} q_\w(v) + \beta_1(v) \, q_\td \, ; \, \sigma_1 + k\right) = 0, \quad \forall \; k\geq 1.
\]
Since $\deg(\beta_1) \leq \sigma_1+1$, we have $\cf\left(\beta_1(v) \, q_\td \, ; \, \sigma_1 + k\right) = 0$ for all $k\geq 2$. 

Assuming $q_\w$ is a polynomial with $\deg(q_w) \geq 1$ and checking the maximum coefficient $K=1+\deg(q_w) \geq 2$ we obtain
\[
 \cf\left(\alpha_1(v) \, v^{1+\sigma_1 - \deg(\alpha_1)} q_\w(v) + \beta_1(v) \, q_\td \, ; \, \sigma_1 + K\right) = \cf(\alpha_1 \,;\, \deg(\alpha_1)) \cf(q_\w \,;\, \deg(q_\w)) = 0,
\]
which implies $\cf(q_\w \,;\, \deg(q_\w)) =0$, which is in contrast to our assumption. 
Hence, $q_\w (v) = q_\w$ must be a constant. For $k=1$ we then obtain the condition
\[
 \cf\left(\alpha_1(v) \, v^{1+\sigma_1 - \deg(\alpha_1)} q_\w + \beta_1(v) \, q_\td \, ; \, \sigma_1 + 1\right) = \cf(\alpha_1 \, ; \, \deg(\alpha_1)) \; q_\w + \cf(\beta_1 \, ; \, \sigma_1 + 1) \, q_\td = 0.
\]
For $\ell=2$ we obtain with the same reasoning that~\eqref{eq:degree-bound-sigma-ell} is equivalent to 
\[
 \cf\left(\alpha_2(v) \, v^{1+\sigma_1 - \deg(\alpha_1)} q_\w + \beta_2(v) \, q_\td \, ; \, \sigma_2 + 1\right) = \cf({\alpha}_{2}; \deg ({\alpha}_{1})-\sigma_1+\sigma_2) \, q_\w + \cf(\beta_2 \, ; \, \sigma_2 + 1) \, q_\td = 0.
\]
To summarize, we obtain that~\eqref{eq:high-and-hat-0}-\eqref{eq:high-and-hat-1} is equivalent to $\td^*_{high}(v) = v^{d-1} q_\td$ and $\w^*_{high}(v) = v^{\pdw+1} q_\w$ satisfying
\begin{eqnarray}
 \cf(\alpha_1 \, ; \, \deg(\alpha_1)) q_\w + \cf(\beta_1 \, ; \, \sigma_1 + 1) q_\td &=&0, \label{eq:condition-qw-qt-1} \\
 \cf(\alpha_2 \, ; \, \deg(\alpha_1)-\sigma_1+\sigma_2) q_\w + \cf(\beta_2 \, ; \, \sigma_2 + 1) q_\td&=&0. \label{eq:condition-qw-qt-2}
\end{eqnarray}
We trivially have $\cf(\alpha_1 \, ; \, \deg(\alpha_1)) \neq 0$, so the first equation~\eqref{eq:condition-qw-qt-1} is equivalent to
\begin{equation}\label{eq:linear-constraint-qw-qt}
 q_\w = - \frac{\cf(\beta_1 \, ; \, \sigma_1 + 1)}{\cf(\alpha_1 \, ; \, \deg(\alpha_1))} q_\td.
\end{equation}
If moreover the determinant $D$ of the linear system \eqref{eq:condition-qw-qt-1}--\eqref{eq:condition-qw-qt-2} satisfies
\begin{eqnarray*}
  0 \neq D &=& \cf(\alpha_1 \, ; \, \deg(\alpha_1)) \cf(\beta_2 \, ; \, \sigma_2 + 1) - \cf(\alpha_2 \, ; \, \deg(\alpha_1)-\sigma_1+\sigma_2) \cf(\beta_1 \, ; \, \sigma_1 + 1) \\ &=& \cf(\alpha_1 \beta_2 - \alpha_2 \beta_1 \, ; \, \deg(\alpha_1) + \sigma_2 + 1),
\end{eqnarray*}
which is equivalent to~\eqref{eq:Proposition3-case1}, then $q_\w = q_\td = 0$ and we are in case~(1). Else, we have $D=0$, which implies that any solution of~\eqref{eq:linear-constraint-qw-qt} also solves~\eqref{eq:condition-qw-qt-2}. Thus we are in case~(2), where the leading coefficients of $\w$ and $\tr$ satisfy the linear constraint 
\begin{equation}\label{eq:linear-constraint-w-tr}
 \cf(\w;  \pdw+1) = -d \frac{\cf(\beta_1 \, ; \, \sigma_1 + 1)}{\cf(\alpha_1 \, ; \, \deg(\alpha_1))} \,\cf(\tr; d),
\end{equation}
which is equivalent to~\eqref{eq:linear-constraint-qw-qt}. This completes the proof.
\end{proof}
\color{black}
Note that, in case (2), if and only if $\cf(\beta_1 \, ; \, \sigma_1 + 1)=\cf(\beta_2 \, ; \, \sigma_2 + 1)=0$, the trace and normal derivative are polynomials of degree $\pd$ and $\pdw$, respectively, which are completely independent of each other. Otherwise, they are of degree $\pd$ and $\pdw+1$, respectively, and are coupled through the extra condition on the leading coefficients.

\begin{remark}
Under assumptions of Proposition~\ref{proposition-linear-edge} we have at least
$\pd$ degrees of freedom for the construction of the trace and $\pdw + 1$ degrees of freedom for the construction of the normal derivative. 
In the generic case, that is, when $\pdw =d-2$ ($\alpha_1$ and $\alpha_2$ have no common factors and at least one is of a maximal possible degree $1+\sigma_\ell$) and~\eqref{eq:Proposition3-case1} is satisfied, 
the number of degrees of freedom is $2\pd-1$.   
\end{remark}

\subsubsection{Non-uniformly parameterized linear interface} \label{sec:non-uniform}

We now consider Case (b). So, we have $\cF_{1}  = (1-\lambda)\cF_{0}+ \lambda \cF_{2}$, with $\lambda \not = \frac{1}{2}$. Let 
\[
\rho(v) := 2 \lambda + 2(1-2\lambda) v. 
\]
Then it is straightforward to compute that
\[
\parDer{v}{}{\bfm{F}^{(\ell)}} (0,v) = \rho(v)\left(\cF_{2}-\cF_{0}\right), \quad \ell \in \{1,2\}.
\]
Let $\bfm{n}_0 = \orth{(\cF_{2}-\cF_{0})}$, which implies $\bfm{n}(v)= \rho(v)\bfm{n}_0$ and further $\beta(v) = \rho^2(v) \norm{\bfm{n}_0}^2$ and 
\begin{equation}\label{eq:gluing-data-case-b}
\wt{\alpha}_\ell(v) = \rho(v) \sprod{\parDer{u}{}{\bfm{F}^{(\ell)}} (0,v)}{\bfm{n}_0},  \quad
\beta_\ell(v) = \rho(v) \sprod{\orth{(\parDer{u}{}{\bfm{F}^{(\ell)}} (0,v))}}{\bfm{n}_0},
\end{equation}
for $\ell \in \{1,2\}$. Hence, we have that $\deg(\beta) = 2$, $\deg(\alpha_\ell) \leq 1+ \sigma_\ell$, $\deg(\beta_\ell) \leq 2+ \sigma_\ell$, and $\pdw\geq \pd-2$.

\begin{lemma}\label{lemma-case-b-high-low}
Suppose that we are in Case (b). Then the high degree contributions $\td^*_{high}$ and $\w^*_{high}$ of any $\C{1}$-smooth isogeometric function $\varphi$ must vanish. Furthermore, the low degree contributions $\td^*_{low}$ and $\w^*_{low}$ must satisfy $\td^*_{low} \in \poly^1_{\pd-3}$ and $\w^*_{low} \in \poly^1_{\pdw}$. Note that $\poly^1_{k}=\{0\}$ for negative $k$.
\end{lemma}
\begin{proof}
Let $\td$ and $\w$ be the tangential and normal derivative of the isogeometric function $\varphi$. According to Theorem~\ref{main-thm-2} and Remark~\ref{rem:condition-delta-d}, the high degree contributions $\td^*_{high}(v) = v^{d_\td+1} q_\td(v)$ and $\w^*_{high}(v) = v^{\pdw+1} q_\w(v)$ must satisfy
\[
 \deg(\td^*_{high}) \leq  d-3
\]
and
\[
 \deg(\alpha_\ell \, \w^*_{high} + \beta_\ell \, \td^*_{high}) \leq d -  1+ \sigma_\ell.
\]
However, the degree bounds of the gluing functions imply $\pdt = \pd-3$ and therefore $\td^*_{high}=0$. Consequenty, $\w^*_{high}$ must also vanish, since any non-zero $\w^*_{high}$ would yield
\[
 \deg(\alpha_\ell \, \w^*_{high}) \geq  \deg(\alpha_\ell) + \min_{\ell' \in \{1,2\}}\{ \pd-1+ \sigma_\ell' - \deg(\alpha_\ell') \} +1,
\]
which contradicts the degree bound. The given degrees for $\td^*_{low}$ and $\w^*_{low}$ directly follow from the degrees of the gluing functions, which completes the proof.
\end{proof}
We can now analyze the $\C{1}$-smooth space $\mathcal{V}^1_\pd(\Omega)$ for Case (b).
\begin{proposition} \label{proposition-nonuniform-edge}
Suppose that we are in Case (b), that is, the interface $\bE$ is a line, parameterized non-uniformly, and let $\varphi$ be an isogeometric function. Then $\varphi\in \mathcal{V}^1_\pd(\Omega)$, if and only if $\td^*\in \poly^1_{\pd-3}$, $\w^*\in\poly^1_{\pdw}$ and
\begin{description}
 \item[(1)] $\hat{\td}(v)=\mu_1 \rho(v)$, $\hat{\w}(v)=0$, for $\mu_1\in\RR$, if the function $\beta$ does not divide ${\alpha}_1 {\beta}_2 - {\alpha}_2 {\beta}_1$, 
 else
 \item[(2)] $\hat{\td}(v) =  \mu_1 \rho(v) + a_L \mu_2$, $\hat{\w}(v) = -b_L \mu_2 \rho(v)$, for $\mu_1,\mu_2\in\RR$. Here
\[
{a}_\ell = {\rm rem}(\alpha_\ell,\rho), \quad {b}_\ell = {\rm rem}\left(\beta_\ell/\rho,\rho\right), \quad \ell \in \{1,2\},
\]
where $\beta_\ell/\rho$ is a polynomial as can be seen in~\eqref{eq:gluing-data-case-b} and $L\in\{1,2\}$ is chosen so that ${a}_L \not =0$.
\end{description}
\end{proposition}
\begin{proof}
We follow the structure of Theorem~\ref{main-thm-2}. Lemma~\ref{lemma-case-b-high-low} gives the degrees of $\td^*$ and $\w^*$. What is left to analyze is which functions $\hat{\td},\hat{\w} \in \poly^1_1$ satisfy~\eqref{eq:beta-div-hat-functions}, i.e., 
\[
 \beta \;\big|\; {\alpha}_\ell \hat{\w} + \beta_\ell \hat{\td}, \quad \ell \in \{1,2\}.
\]
By computing
\[
{\alpha}_\ell = {\alpha}_\ell ^* \, \beta+\hat{\alpha}_\ell, \quad \beta_\ell = \beta_\ell ^* \, \beta+\hat{\beta}_\ell,
\]
this reduces to
\begin{equation} \label{eq-divisionWithbeta2}
 \beta \;\big|\; \hat{\alpha}_\ell \hat{\w} + \hat{\beta}_\ell \hat{\td}, \quad \ell \in \{1,2\}.
\end{equation}
Note that $\hat{\alpha}_1$ and $\hat{\alpha}_2$ cannot both be zero, since this would imply that $\beta$ is a non-constant common factor of ${\alpha}_1$ and ${\alpha}_2$. We have 
\[\beta_{\ell} = \rho \sprod{\orth{(\parDer{u}{}{\bfm{F}^{(\ell)}} (0,v))}}{\bfm{n}_0} = \rho \left(\rho \; {\rm quot}\left(\beta_\ell/\rho,\rho\right) + {\rm rem}\left(\beta_\ell/\rho,\rho\right) \right) = \beta\frac{{\rm quot}\left(\beta_\ell/\rho,\rho\right)}{\norm{\bfm{n}_0}^2}  + b_\ell\, \rho, 
\]
so $\hat{\beta}_\ell(v) =  {b}_\ell  \rho(v)$.
Let us further denote 
\[
\hat{\td}(v)= \td_1  \rho(v)+ \td_0, \quad 
\hat{\w}(v) = \w_1 \rho(v) + \w_0,
\]
for $\td_0, \td_1, \w_0, \w_1 \in \RR$, and let ${{c}_\ell} = {\rm quot}(\hat{\alpha}_\ell,\rho)$, i.e., $\hat{\alpha}_\ell(v) = {{c}_\ell} \rho (v) + a_ \ell$, $\ell \in \{1,2\}$.
Then
\[
 \hat{\alpha}_\ell(v) \hat{\w}(v) + \hat{\beta}_\ell(v) \hat{\td}(v) =  \left({{c}_\ell} \w_1 + {{b}_\ell} \td_1\right)
\rho^2(v) +\left({{a}_\ell} \w_1+ {{c}_\ell} \w_0 + {{b}_\ell} \td_0 \right) \rho(v) + {{a}_\ell} \w_0.
\]
Since $\gcd\left(\wt{\alpha}_1, \wt{\alpha}_2\right) = 1$, the polynomial $\rho$ can not divide both of the two polynomials $\hat{\alpha}_\ell$ (take into account that zero is divisible by any non-zero polynomial), and so at least one of ${{a}_1}, {{a}_2}$ is nonzero. Therefore, $\rho^2$ divides 
$\hat{\beta}_\ell \hat{\td} + \hat{\alpha}_\ell \hat{\w}$ iff
$\w_0 = 0$ and 
\[
\begin{bmatrix}
{{a}_1} & {{b}_1}\\ {{a}_2} & {{b}_2}
\end{bmatrix}
\begin{bmatrix}
\w_1\\ \td_0
\end{bmatrix} = 
\begin{bmatrix}
0 \\ 0
\end{bmatrix}. 
\]
If the matrix is invertible, this implies the first option, while the second option follows from computing the matrix kernel, which gives 
$\td_0 = {{a}_L} \mu_2$, $\w_1 = -{{b}_L} \mu_2$  for any $\mu_2\in \RR$, where $L\in\{1,2\}$ is chosen so that ${{a}_L} \not = 0$.
One can check easily that the condition ${{a}_2} {{b}_1} = {{a}_1} {{b}_2}$ is equivalent to $\beta\; | \;\hat{\alpha}_1 \hat{\beta}_2 - \hat{\alpha}_2 \hat{\beta}_1$, which in turn is equivalent to $\beta\; | \;{\alpha}_1 {\beta}_2 - {\alpha}_2 {\beta}_1$. This completes the proof.
\end{proof}

\subsubsection{Parabolic interface}\label{sec:parabolic}

It remains to analyze the case where the interface is a parabola - Case (c). In this case polynomial $\beta$ is an irreducible quadratic polynomial, so $\deg(\beta) = 2$. Moreover, we have $\deg(\alpha_\ell) \leq 2+ \sigma_\ell$ and $\deg(\beta_\ell) \leq 2+ \sigma_\ell$. In addition the gluing functions satisfy the following.
\begin{lemma}\label{lem:beta-divides-abba}
We have  $\beta \; | \; {\alpha}_1 {\beta}_2 - {\alpha}_2 {\beta}_1$.
\end{lemma}
\begin{proof}
From equality~\eqref{rel-alpha-beta} it follows that $\beta$ must divide $\q \left({\alpha}_1 {\beta}_2 - {\alpha}_2 {\beta}_1\right)$ where $\q={\rm gcd}(\wt{\alpha}_1,\wt{\alpha}_2)$. We now follow a proof by contradiction. Assume that $\beta$ does not divide ${\alpha}_1 {\beta}_2 - {\alpha}_2 {\beta}_1$. Then, since it is an irreducible polynomial, it must divide $\q$, i.e., $\q=\q^*\beta$. Consequently, 
$\wt{\alpha}_\ell = \q^* \beta \, {\alpha}_\ell$, which is equal to  
\[
\sprod{{\parDer{u}{}{\bfm{F}^{(\ell)}} (0,v)}}{\bfm{n}(v)} = \q^*(v) {\alpha}_\ell(v) \sprod{\bfm{n}(v)}{\bfm{n}(v)}.
\]
This implies that ${\parDer{u}{}{\bfm{F}^{(\ell)}} (0,v)} -  \q^*(v) {\alpha}_\ell(v) \bfm{n}(v)$ is orthogonal to $\bfm{n}(v)$ for every $v\in [0,1]$.
Since  
\[
 \beta_\ell(v) =  \sprod{\orth{(\parDer{u}{}{\bfm{F}^{(\ell)}} (0,v))}}{\bfm{n}(v)}  = 
 \sprod{\orth{(\parDer{u}{}{\bfm{F}^{(\ell)}} (0,v))} -  \q^*(v) {\alpha}_\ell(v) \orth{\bfm{n}(v)}}{\bfm{n}(v)}
\]
we see that $\beta$ also divides $\beta_\ell$, for $\ell \in \{1,2\}$, which contradicts our assumption. This completes the proof.
\end{proof}
As in Case (b) we can directly characterize the functions $\td^*$ and $\w^*$.
\begin{lemma}\label{lemma-case-c-high-low}
Suppose that we are in Case (c). Then the high degree contributions $\td^*_{high}$ and $\w^*_{high}$ of any $\C{1}$-smooth isogeometric function $\varphi$ must vanish. Furthermore, the low degree contributions $\td^*_{low}$ and $\w^*_{low}$ must satisfy $\td^*_{low} \in \poly^1_{\pd-3}$ and $\w^*_{low} \in \poly^1_{\pdw}$.
\end{lemma}
\begin{proof}
The proof is the same as for Lemma~\ref{lemma-case-b-high-low}.
\end{proof}
Similar to Case (b) we have to analyze the remainders $\hat{\td},\hat{\w} \in \poly^1_1$, which depend on the gluing functions. 
Due to Lemma~\ref{lem:beta-divides-abba} we have $\hat{\alpha}_1 \hat{\beta}_2 - \hat{\alpha}_2 \hat{\beta}_1 = {\rm c}\, \beta$ for some constant $c \in \RR$. We distinguish between two cases.
\begin{lemma} \label{lemma-1-div}
Consider Case (c) and suppose that $\hat{\alpha}_1 \hat{\beta}_2 - \hat{\alpha}_2 \hat{\beta}_1 = {\rm c}\, \beta$ for some nonzero constant $c$. Then~\eqref{eq:beta-div-hat-functions} holds true if and only if
\begin{equation} \label{eq-divisionWithbetaSol1}
 \hat{\td}(v) = \mu_1 \hat{\alpha}_1(v) + \mu_2 \hat{\alpha}_2(v), \quad 
 \hat{\w}(v) = -\mu_1 \hat{\beta}_1(v) - \mu_2 \hat{\beta}_2(v), 
\end{equation}
for any two free parameters $\mu_1, \mu_2 \in \RR$.
\end{lemma}
\begin{proof}
We follow a similar strategy as in Case~(b) and obtain from~\eqref{eq:beta-div-hat-functions} (and~\eqref{eq-divisionWithbeta2}) that
\begin{equation} \label{eq-divisionWithbeta3}
\hat{\alpha}_1(v) \hat{\w}(v) + \hat{\beta}_1(v) \hat{\td}(v) = \beta(v) \, q_1, \quad 
\hat{\alpha}_2(v) \hat{\w}(v) + \hat{\beta}_2(v) \hat{\td}(v) = \beta(v) \, q_2,
\end{equation}
for any $q_1, q_2 \in \RR$, which can be written in a matrix form as
\[ 
M(v)
\begin{bmatrix}
\hat{\w}(v)\\
\hat{\td}(v)
\end{bmatrix} = 
\beta(v)
\begin{bmatrix}
q_1\\
q_2
\end{bmatrix}, \quad M(v)=\begin{bmatrix}
\hat{\alpha}_1(v) &  \hat{\beta}_1(v)\\
\hat{\alpha}_2(v) &  \hat{\beta}_2(v)
\end{bmatrix},
\]
where the elements of the matrix $M(v)$ are linear polynomials and its determinant equals $\det M(v) = \hat{\alpha}_1(v) \hat{\beta}_2(v) - \hat{\alpha}_2(v) \hat{\beta}_1(v)$. 
By the assumption $\det M(v) = c\, \beta(v)$ for a nonzero contant $c$. Since $\beta$ can not have real roots, 
the solution of \eqref{eq-divisionWithbeta3} is unique. Using Cramer's rule, we obtain
\[
 \hat{\td}(v) = \frac{1}{c\, \beta(v)} \beta(v) \left(q_2 \hat{\alpha}_1(v) - q_1 \hat{\alpha}_2(v)\right), \quad
\hat{\w}(v) = \frac{1}{c\, \beta(v)} \beta(v) \left(-q_2 \hat{\beta}_1(v) + q_1 \hat{\beta}_2(v)\right),
\]
which are linear polynomials of the form \eqref{eq-divisionWithbetaSol1} where $\mu_1=\frac{q_2}{c}$, $\mu_2=-\frac{q_1}{c}$ are the two free constants. 
This completes the proof.
\end{proof}
\begin{lemma} \label{lemma-2-div}
Consider Case (c) and suppose that $\hat{\alpha}_2 \hat{\beta}_1= \hat{\alpha}_1 \hat{\beta}_2$.
Then one of $\hat{\td}$ or $\hat{\w}$ can be chosen completely free, while the other one is uniquely determined from~\eqref{eq:beta-div-hat-functions}. 
\end{lemma}
\begin{proof}
Again, we reduce~\eqref{eq:beta-div-hat-functions} to
\begin{equation} \label{eq-divisionWithbeta2-2}
 \beta \;\big|\; \hat{\alpha}_\ell \hat{\w} + \hat{\beta}_\ell \hat{\td}, \quad \ell \in \{1,2\}.
\end{equation}
Let us first consider the case when $\hat{\alpha}_1, \hat{\alpha}_2, \hat{\beta}_1, \hat{\beta}_2$ all have a common linear factor or all of them are constant, i.e.,
\[
 \hat{\alpha}_\ell(v) = a_\ell \zeta(v), \quad
\hat{\beta}_\ell(v) = b_\ell \zeta(v), \quad \ell \in \{1,2\},
\]
for $\zeta \in \poly_1$, $\zeta \not=0$, and $a_\ell, b_\ell\in \RR$, such that $a_2 b_1=a_1 b_2$.  
Conditions \eqref{eq-divisionWithbeta2-2} are then equivalent to $ \zeta(v)\left(b_\ell \hat{\td}(v) + a_\ell \hat{\w}(v)\right) = q_\ell \beta(v)$, for some $q_\ell\in\RR$, $\ell \in \{1,2\}$. Since $\beta$ is irreducible, these two equalities can hold true iff $q_1=q_2=0$.
Since $a_2 b_1=a_1 b_2$, both conditions are equivalent, and are satisfied iff 
\[
\hat{\td}(v)= a_L (\mu_1 v + \mu_2), \quad \hat{\w}(v)=- b_L (\mu_1 v + \mu_2), \quad \mu_1, \mu_2 \in \RR,
\]
where $L$ is chosen so that $a_L$ or $b_L$ is nonzero. 

Suppose now that $\hat{\alpha}_1, \hat{\alpha}_2, \hat{\beta}_1, \hat{\beta}_2$ do not have a common linear factor and not all of them are constant. 
Since $\hat{\alpha}_1, \hat{\alpha}_2, \hat{\beta}_1, \hat{\beta}_2$ are in $\poly_1$, the equality $\hat{\alpha}_2 \hat{\beta}_1= \hat{\alpha}_1 \hat{\beta}_2$ is possible only if
\begin{equation} \label{cond-alpha-beta-1}
\hat{\alpha}_\ell(v) = c \hat{\beta}_\ell(v), 
\end{equation}
or 
\begin{equation}  \label{cond-alpha-beta-2}
\hat{\alpha}_2(v) = c \hat{\alpha}_1(v), \quad \hat{\beta}_2(v) = c \hat{\beta}_1(v),  
\end{equation}
for some $c\in \RR$. 

Under the assumption \eqref{cond-alpha-beta-1}, conditions~\eqref{eq-divisionWithbeta2-2} become 
$\beta \;\big|\; \hat{\beta}_\ell \left(\hat{\td} + c \hat{\w}\right), \, \ell \in \{1,2\}.$
Since $\beta$ is irreducible, this is possible iff $\hat{\td} + c \hat{\w}=0$ or equivalently if
\[
\hat{\td}(v)= -c(\mu_1 v + \mu_2), \quad \hat{\w}(v)= \mu_1 v + \mu_2, \quad \mu_1, \mu_2 \in \RR. 
\]
If $\hat{\beta}_1 = \hat{\beta}_2 = 0$, \eqref{eq-divisionWithbeta2-2} is satisfied iff 
$ \hat{\w} = 0$, $\hat{\td}(v) = \mu_1 v + \mu_2$ for any $\mu_1, \mu_2 \in \RR$. 

Assuming \eqref{cond-alpha-beta-2},  both conditions~\eqref{eq-divisionWithbeta2-2} are the same, so it is enough to consider only one of them. Note that the same is true if $\hat{\alpha}_1 = \hat{\beta}_1 = 0$ or if
$\hat{\alpha}_2 = \hat{\beta}_2 = 0$
Thus, let us fix $\ell = L$, $L\in \{1,2\}$, so that $\hat{\alpha}_L$, $\hat{\beta}_L$  are not both identically zero. Suppose that $\hat{\alpha}_L \in \poly_1$  is not constant. Then $\{1, \hat{\alpha}_L, \hat{\alpha}_L^2\}$ is a basis of $\poly_2$. Writing $\hat{\beta}_L$, $\hat{\td}$, $\hat{\w}$ and $\beta$ in this basis, i.e., as
\[
\hat{\beta}_L(v) = b_1 \hat{\alpha}_L(v)+b_0, \quad 
\hat{\td}(v) = \td_1 \hat{\alpha}_L(v)+\td_0, \quad
\hat{\w}(v) = \w_1 \hat{\alpha}_L(v)+\w_0, \quad 
\hat{\beta}(v) = c_2 \hat{\alpha}_L^2(v)+ c_1 \hat{\alpha}_L(v) +c_0, 
\]
yields 
\[
 \hat{\beta}_L(v) \hat{\td}(v) + \hat{\alpha}_L(v) \hat{\w}(v) = 
 \hat{\alpha}_L^2(v) \left(b_1 \td_1 + \w_1\right) + 
 \hat{\alpha}_L(v) \left(b_1  \td_0 + b_0 \td_1 + \w_0\right)  + b_0 \td_0.
\] 
This expression is divisible by $\beta$ iff it is equal to $q_L \beta$ for any $q_L \in \RR$, which implies
\[
b_0 \td_0 = q_L c_0, \quad b_1 \td_0 + b_0 \td_1 + \w_0 = q_L c_1, \quad
b_1 \td_1 + \w_1 = q_L c_2.
\]
Note that $c_0\not=0$ (because $\beta$ is irreducible) and $q_L$ is a free constant. Thus, the solution of these three equations can be expressed with two free parameters $\mu_1, \mu_2 \in \RR$ as
\begin{equation}\label{cond-t-w-2}
\hat{\td}(v) = \mu_1 \hat{\alpha}_L(v) + \mu_2, \quad 
\hat{\w}(v) = \frac{1}{c_0}\left(\mu_2 b_0 c_2 - \mu_1 b_1 c_0\right) \hat{\alpha}_L(v) + 
\frac{\mu_2}{c_0} \left(b_0 c_1 - b_1 c_0\right) - \mu_1 b_0.
\end{equation}
If $\hat{\alpha}_L$  is a constant, then $\hat{\beta}_L$ must be of degree one, and $\{1, \hat{\beta}_L, \hat{\beta}_L^2\}$ can be taken as a basis of $\poly_2$. Writing $\hat{\alpha}_L$, $\hat{\td}$, $\hat{\w}$ and $\beta$ in this basis, i.e., as
\[
\hat{\alpha}_L(v) = a_1 \hat{\beta}_L(v)+a_0, \quad 
\hat{\td}(v) = \td_1 \hat{\beta}_L(v)+\td_0, \quad
\hat{\w}(v) = \w_1 \hat{\beta}_L(v)+\w_0, \quad 
\hat{\beta}(v) = c_2 \hat{\beta}_L^2(v)+ c_1 \hat{\beta}_L(v) +c_0, 
\] 
we compute in the same way that~\eqref{eq-divisionWithbeta2-2} holds true iff 
\begin{equation}\label{cond-t-w-3}
\hat{\td}(v) = \frac{1}{c_0}\left(\mu_2 a_0 c_2 - \mu_1 a_1 c_0\right) \hat{\beta}_L(v) + 
\frac{\mu_2}{c_0} \left(a_0 c_1 - a_1 c_0\right) - \mu_1 a_0, \quad 
\hat{\w}(v) = \mu_1 \hat{\beta}_L(v) + \mu_2.
\end{equation}
The proof is completed.
\end{proof}

\begin{proposition} \label{proposition-parabola}
Suppose that we are in Case (c), that is, the interface $\bE$ is a parabola, and let $\varphi$ be an isogeometric function. Then $\varphi\in \mathcal{V}^1_\pd(\Omega)$, if and only if $\td^*\in \poly^1_{\pd-3}$, $\w^*\in\poly^1_{\pdw}$ and $\hat{\td}$ and $\hat{\w}$ are given as in Lemma~\ref{lemma-1-div} or~\ref{lemma-2-div}.
\end{proposition}
\begin{proof}
 This statement follows directly from Theorem~\ref{main-thm-2} together with Lemmas~\ref{lemma-case-c-high-low},~\ref{lemma-1-div} and~\ref{lemma-2-div}.
\end{proof}

\subsubsection{Summary}\label{sec:summary}

Since the analysis given in the previous subsections is quite technical, we summarize here the main steps in the construction of the trace function $\tr$ and normal derivative $\w$ using Algorithm~\ref{alg:TrAndW}. 
Moreover, we explicitly denote and store the parameters of freedom that correspond to $\tr$ and $\w$ in the sets $\mbox{FreePar}_\tr$ and $\mbox{FreePar}_\w$, respectively. In case of a parabolic interface, the two parameters of freedom $\mu_1$ and $\mu_2$, given 
by Lemmas~\ref{lemma-1-div} and \ref{lemma-2-div}, can be assigned for some special cases only to the trace $\tr$ or only to the normal derivative~$\w$. Thus we store them separately in the set $\mbox{FreePar}_\mu$. The algorithm also returns values
$n_\tr = |\mbox{FreePar}_\tr|$ and $n_\w = |\mbox{FreePar}_\w|$ which give the number of degrees of freedom corresponding to $\tr$ and $\w$, respectively, as well as $n_\mu = |\mbox{FreePar}_\mu|$.

\begin{algorithm}
\setstretch{1.2}
\SetKwInput{KwData}{Input}
\SetKwInput{KwResult}{Output}
\SetAlgoLined
\uIf{interface is linear and uniformly parameterized}
{
{\bf Go to } Proposition~\ref{proposition-linear-edge}\;
\uIf{\eqref{eq:Proposition3-case1} is satisfied}
{
$\tr(v) \leftarrow \sum_{i=0}^{\pd-1} \tr_i \bern{\pd-1}{i}(v)$, $\mbox{FreePar}_\tr = \set{\tr_0, \tr_1, \dots, \tr_{\pd-1}}$, $n_\tr = \pd$\;
$\w(v) \leftarrow \sum_{i=0}^{\pdw} \w_i \bern{\pdw}{i}(v)$, $\mbox{FreePar}_\w = \set{\w_0, \w_1, \dots, \w_{\pdw}}$, $n_\w = \pdw+1$\;
}
\Else
{
$\tr(v) \leftarrow \sum_{i=0}^{\pd} \tr_i \bern{\pd}{i}(v)$, $\mbox{FreePar}_\tr = \set{\tr_0, \tr_1, \dots, \tr_{\pd}}$, $n_\tr = \pd+1$\;
compute $\w_{\pdw+1}$ as stated in~\eqref{eq:linear-constraint-w-tr}\;
$\w(v) \leftarrow \sum_{i=0}^{\pdw+1} \w_i \bern{\pdw+1}{i}(v)$, $\mbox{FreePar}_\w = \set{\w_0, \w_1, \dots, \w_{\pdw}}$, $n_\w = \pdw+1$\;
}
$\mbox{FreePar}_{\mu} = \set{}$, $n_\mu=0$\;
}
\uElseIf{interface is linear and non-uniformly parameterized}
{
{\bf Go to } Proposition~\ref{proposition-nonuniform-edge}\;
$\td^*(v) \leftarrow \sum_{i=0}^{\pd-3} \td_i \bern{\pd-3}{i}(v)$, $ \w^*(v) \leftarrow \sum_{i=0}^{\pdw} \w_i \bern{\pdw}{i}(v)$\;
\uIf{$\beta \nmid {\alpha}_1 {\beta}_2 - {\alpha}_2 {\beta}_1$}
{
Compute $\hat{\td}$ and $\hat{\w}$ as stated in Proposition~\ref{proposition-nonuniform-edge}, case~(1)\;
$\mbox{FreePar}_\tr = \set{\tr_0, \td_0,  \td_1, \dots, \td_{\pd-3}, \mu_1}$, $n_\tr = \pd$\;
$\mbox{FreePar}_\w = \set{\w_0, \w_1, \dots, \w_{\pdw}}$, $n_\w = \pdw+1$\;
}
\Else
{
Compute $\hat{\td}$ and $\hat{\w}$ as stated in Proposition~\ref{proposition-nonuniform-edge}, case~(2)\;
$\mbox{FreePar}_\tr = \set{\tr_0, \td_0,  \td_1, \dots, \td_{\pd-3}, \mu_1, \mu_2}$, $n_\tr = \pd+1$\;
$\mbox{FreePar}_\w = \set{\w_0, \w_1, \dots, \w_{\pdw}}$, $n_\w = \pdw+1$\;
}
$\mbox{FreePar}_{\mu} = \set{}$, $n_\mu=0$\;
Compute $\tr$ and $\w$ by \eqref{def-trAndw-Theorem1}\;
}
\ElseIf{interface is parabolic}
{

{\bf Go to } Proposition~\ref{proposition-parabola}\;
$\td^*(v) \leftarrow \sum_{i=0}^{\pd-3} \td_i \bern{\pd-3}{i}(v)$, $ \w^*(v) \leftarrow \sum_{i=0}^{\pdw} \w_i \bern{\pdw}{i}(v)$\;

\uIf{$\hat{\alpha}_2 \hat{\beta}_1 \neq \hat{\alpha}_1 \hat{\beta}_2$}
{
Compute $\hat{\td}$ and $\hat{\w}$ as stated in Lemma~\ref{lemma-1-div}\;
}
\Else
{
Compute $\hat{\td}$ and $\hat{\w}$ as stated in Lemma~\ref{lemma-2-div}\;
}
$\mbox{FreePar}_\tr = \set{\tr_0, \td_0,  \td_1, \dots, \td_{\pd-3}}$, $n_\tr = \pd-1$\;
$\mbox{FreePar}_\w = \set{\w_0, \w_1, \dots, \w_{\pdw}}$, $n_\w = \pdw+1$\;
$\mbox{FreePar}_{\mu} = \set{\mu_1, \mu_2}$, $n_\mu=2$\;
Compute $\tr$ and $\w$ by \eqref{def-trAndw-Theorem1}\;

}
\KwResult{
$\set{\tr, \w, \mbox{FreePar}_\tr, \mbox{FreePar}_\w,\mbox{FreePar}_{\mu}, n_\tr, n_\w,n_\mu}$
}
\caption{Compute $\tr$-$\w$-basis}
\label{alg:TrAndW}
\end{algorithm}

\section{Dimension and basis for the $\C{1}$-smooth isogeometric space over (bi-)quadratic elements}\label{sec:basis}

In this section we show how a basis for the $\C{1}$-smooth isogeometric space $\mathcal{V}^1_\pd(\Omega)$ can be constructed in a geometrically intuitive way that could be extended to construct splines over more than two elements. Before that, we give bounds on the dimension of the space $\mathcal{V}^1_\pd(\Omega)$, cf. \cite[Proposition~4.6]{MoViVi16}, which presents a similar dimension count. Note that \cite[Proposition~4.6]{MoViVi16} covers a more general setting than the following corollary. The difference lies in the description of the dimension in terms of geometric properties and simple conditions on the gluing data, while the approach in \cite{MoViVi16} requires the computation of a basis for a syzygy module, which makes a geometric interpretation more difficult.
\begin{corollary}
Let $\mathcal{V}^1_\pd(\Omega)$ be the $\C{1}$-smooth isogeometric space over two elements. Its dimension and the upper bound for the dimension equal
\[
 \mathrm{dim}(\mathcal{V}^1_\pd(\Omega)) = D_0 + 2\pd +  \min_{\ell\in\{1,2\}}\{ \sigma_\ell - \deg({\alpha}_\ell) \} + \kappa \leq D_0 + 2\pd + 1 + \min_{\ell\in\{1,2\}}\{ \sigma_\ell \},
\]
where $D_0$ is defined as in~\eqref{eq-D0}, $\kappa\in\{0,1\}$, with 
\begin{itemize}
 \item $\kappa=0$, in case of Proposition~\ref{proposition-linear-edge}~(1), or Proposition~\ref{proposition-nonuniform-edge}~(1), 
 \item $\kappa=1$, in case of Proposition~\ref{proposition-linear-edge}~(2), Proposition~\ref{proposition-nonuniform-edge}~(2), or Proposition~\ref{proposition-parabola}. 
\end{itemize}
In Case (a) and (b) we have $0 \leq \deg({\alpha}_\ell) \leq 1+\sigma_\ell$ whereas in Case (c) we have $0 \leq \deg({\alpha}_\ell) \leq 2+\sigma_\ell$. Thus, we obtain for all cases the lower bound
\[
 \mathrm{dim}(\mathcal{V}^1_\pd(\Omega)) \geq D_0 + 2\pd -1
\]
for the dimension of the $\C{1}$-smooth isogeometric space.
\end{corollary}

We can also derive the B\'ezier representation of $\C{1}$-smooth isogeometric functions.
\begin{corollary}
Let $\varphi \in \mathcal{V}_\pd^1(\Omega)$ be a $\C{1}$-smooth isogeometric function. The functions $f^{(\ell)} = \varphi \circ \bfm{F}^{(\ell)}$ are expressed in the Bernstein basis as
\[
f^{(\ell)}(u,v) = \begin{cases}
\displaystyle{\sum_{0\leq i+j\leq \pd} {b}^{(\ell)}_{i,j} \bernT{\pd}{i,j}(u,v)}, & {\rm if} \quad
\mathcal{D}^{(\ell)}= \parDomT\\[2mm]
\displaystyle{\sum_{i,j=0}^\pd {b}^{(\ell)}_{i,j} \bernB{ \scalebox{0.6}{$\Box$}, \pd}{i,j}(u,v)}, &   {\rm if} \quad \mathcal{D}^{(\ell)}= \parDomQ
\end{cases}.
\]
We then have that $\binom{\pd}{2}$ B\'ezier coefficients ${b}^{(\ell)}_{i,j}$, for $i\geq 2$ and $i+j\leq\pd$, in case of a triangular element, and $(\pd-1)(\pd+1)$ B\'ezier coefficients ${b}^{(\ell)}_{i,j}$, for $2\leq i\leq d$ and $0\leq j\leq d$, in case of a quadrilateral element can be chosen completely free. This results in $D_0$ degrees of freedom, that have no influence on the $\C{1}$ conditions at the interface.

Then the remaining B\'ezier coefficients of $f^{(\ell)}$ are given such that $b^{(1)}_{0,j}=b^{(2)}_{0,j}$, for $j=0,1,\dots, \pd$, are the B\'ezier coefficients of the function $\tr$ and 
\[
 b^{(\ell)}_{1,j} = b^{(\ell)}_{0,j} + \frac{1}{\pd} \eta_{\ell,j}, \quad j=0,1,\dots, \pd-1+\sigma_\ell,
\]
where $\eta_{\ell,j}$ are the B\'ezier coefficients of the function $\eta_\ell$, i.e.,
\[
\eta_\ell(v) = \sum_ {j=0}^{d-1+\sigma_\ell} \eta_{\ell,j} \bern{\pd-1+\sigma_\ell}{j}(v),
\] 
as defined in Theorem~\ref{main-thm-1}.
\end{corollary}
\begin{proof}
Clearly, the B\'ezier coefficients $b^{(\ell)}_{i,j}$ that are assumed to be freely chosen have no influence on $\C{1}$-smoothness over the common interface. The B\'ezier coefficients $b^{(\ell)}_{i,j}$ with the first index $i\in\{0,1\}$ are uniquely computed from polynomial identities \eqref{cond-Tr} and
$\parDer{u}{}{f^{(\ell)}} (0, v)= \eta_\ell(v)$. 
\end{proof}

Using Propositions~\ref{proposition-linear-edge},~\ref{proposition-nonuniform-edge} and~\ref{proposition-parabola} together with Algorithm~\ref{alg:TrAndW} one way to derive linearly independent $\C{1}$-smooth isogeometric functions that correspond to the interface is to connect them to independent free parameters given in $\mbox{FreePar}_\tr$, $\mbox{FreePar}_\w$ and $\mbox{FreePar}_\mu$. Namely, a basis function associated to a free parameter is defined by assigning a nonzero value to that parameter and zero values to all other free parameters. It is straightforward to see from the representation of $\tr$ and $\w$ that the obtained functions are indeed linearly independent. Let us call this basis the $\tr$-$\w$-basis. From the $\tr$-$\w$-basis we can construct another basis by defining the interpolation conditions on the trace~$\tr$ and on the normal derivative~$\w$ in such a way that the problem is uniquely solvable, i.e., the free coefficients of $\tr$ and $\w$, given in $\mbox{FreePar}_\tr$, $\mbox{FreePar}_\w$ and $\mbox{FreePar}_\mu$, are uniquely determined.  
The corresponding collocation matrix consists of rows obtained by applying the chosen interpolation functionals on the elements of the $\tr$-$\w$-basis.  The set of new basis functions can then be determined by setting the value of one interpolation data to a nonzero value and all other values to zero, and by repeating this through all interpolation conditions. 

Below, we present a method for the construction of a set of basis functions, which uses just the interpolation approach for a linear interface, and combines both approaches for a parabolic interface. To start, let $\lambda_t^{(\ell)}$ denote the linear functional defined on a set of differentiable univariate functions as 
\[ 
\lambda_t^{(\ell)} f:= f^{(\ell)}(t),  \mbox{ }\ell \in \NN_0.
\]
Furthermore, recall that $n_{\tr}$ and $n_{\w}$ denote the degrees of freedom for the construction of the trace~$\tr$ and of the normal derivative~$\omega$, respectively, and let $K \in \NN_0$ with $K \leq \min (\lfloor\frac{n_{\tr}-2}{2} \rfloor,\lfloor\frac{n_{\omega}}{2} \rfloor)$. We first include the following $4K+2$ interpolation functionals
\begin{subequations} \label{eq:int_pol}
\begin{equation} \label{eq:int_pol_1}
\lambda_0^{(\ell)} \tr,\; \lambda_1^{(\ell)} \tr, \quad \ell=0,1,\dots, K, \quad {\rm and} \quad \lambda_0^{(\ell)} \w, \; \lambda_1^{(\ell)} \w, \quad \ell=0,1, \dots, K-1,
\end{equation}
in the interpolation problem. Note that their interpolation values can be uniquely determined by prescribing $\C{K}$ interpolation conditions at the boundary vertices of the interface (see e.g. \cite[Proof of Theorem~1]{GrKaKnTaVi2020}).  Afterwards, we choose parameters~$\widetilde{t}_{i},\widehat{t}_{j} \in (0,1)$, for $i=1,\ldots,n_{\tr}-2K-2$, and $j=1,\ldots,n_{\omega}-2K$, e.g., uniformly given by 
\[
 \widetilde{t}_{i} = \frac{i}{n_{\tr}-2K-1}, \quad \mbox{and} \quad 
 \widehat{t}_{j} = \frac{j}{n_{\w}-2K+1},
\]
respectively, and include the interpolation functionals 
\begin{equation} \label{eq:int_pol_2}
 \lambda^{(0)}_{\widetilde{t}_i} \tr, \quad i=1,\ldots,n_{\tr}-2K-2, \quad \mbox{and} \quad
 \lambda^{(0)}_{\widehat{t}_j} \w, \quad j=1,\ldots,n_{\omega}-2K,
\end{equation}
\end{subequations} 
in the interpolation problem. Then, a set of basis functions is given via the interpolation problem~\eqref{eq:int_pol} by constructing the basis functions in such a way that in each case one interpolation data is set to nonzero and all others are set to zero. 
A good and practical choice for $K$, which we also follow in this work, is to select $K$ as large as possible, i.e., $K=\min (\lfloor\frac{n_{\tr}-2}{2} \rfloor,\lfloor\frac{n_{\omega}}{2} \rfloor)$.

In the case of a linear interface the described construction uniquely determines all the basis functions that correspond to the interface. However, if the interface is a parabola, we get two additional parameters of freedom, $\mu_1$ and $\mu_2$ 
(stored in a set $\mbox{FreePar}_\mu$). As we already commented, these two parameters can be for some configurations of meshes involved only in the trace $\tr$ or only in the normal derivative $\w$. Thus, we define the two basis functions that correspond to these 
two degrees of freedom by first assigning a nonzero value to the parameter $\mu_1$ and a zero value to $\mu_2$, and vice-versa. Moreover, for these two basis functions the remaining parameters $\mbox{FreePar}_\tr$ and $\mbox{FreePar}_\w$ are computed
from the interpolation problem~\eqref{eq:int_pol} where all the interpolation values are set to zero. Note that this is not equivalent to setting all of these parameters to zero.  
Since in the case of a parabolic interface $n_\tr = d-1$ and $n_\w = \pdw + 1 \geq \pd-2$, we 
get that $K= \lfloor\frac{\pd - 3}{2} \rfloor$. If also $\pdw = \pd-3$ (the generic case), then 
the interpolation functionals \eqref{eq:int_pol} are equal to
\begin{subequations}\label{eq-IntFun-PI-d}
\begin{equation} \label{eq-IntFun-PI-even-d}
\lambda_0^{(j)} \tr, \quad \lambda_1^{(j)} \tr, \quad j=0,1,\dots, k-2, \quad 
\lambda_0^{(j)} \w, \quad \lambda_1^{(j)} \w, \quad j=0,1,\dots, k-3, \quad
\lambda_{1/2}^{(0)} \tr, \quad \lambda_{1/3}^{(0)} \w,  \quad \lambda_{2/3}^{(0)} \w,
\end{equation}
if $\pd=2k$ is even, or to
\begin{equation} \label{eq-IntFun-PI-odd-d}
\lambda_0^{(j)} \tr, \quad \lambda_1^{(j)} \tr, \quad j=0,1,\dots, k-1, \quad 
\lambda_0^{(j)} \w, \quad \lambda_1^{(j)} \w, \quad j=0,1,\dots, k-2, \quad
\lambda_{1/2}^{(0)} \w,
\end{equation}
if $\pd=2k+1$ is odd.  
\end{subequations}

Observing Lemmas~\ref{lemma-1-div} and~\ref{lemma-2-div} we note that the basis functions that correspond to $\mu_1$ and $\mu_2$ can be almost linearly dependent if the mesh is such that we are close to the case where $\hat{\alpha}_1 \hat{\beta}_2 = \hat{\alpha}_2 \hat{\beta}_1$ in which the dimension stays the same but the way the remainders $\hat{\td}$ and $\hat{\w}$ are computed is different. To avoid the instability that can follow from these two basis functions we propose to apply the Gram--Schmidt algorithm (with respect to a scalar product defined as the integral over the domain)  on these two functions to make them mutually orthogonal. Since this procedure demands just a computation of two norms and one scalar product, we propose to always apply this orthogonalization. 

In the next section some numerical examples of mesh elements for the cases described previously are provided. We also show, for the parabolic interface, examples of basis functions, constructed as described using~\eqref{eq-IntFun-PI-d}.

\section{Configurations of two mesh elements and corresponding $\C{1}$-smooth isogeometric bases} \label{sec:examples}

In the following we consider several examples of pairs of elements and construct, for some of them, $\C{1}$-smooth isogeometric basis functions over them.

\subsection{Different configurations of two mesh elements} 

We always assume to have one triangular element $\Omega^{(1)}$ and one quadrilateral element $\Omega^{(2)}$, with control points $\cF^{(1)}$ and $\cF^{(2)}$, respectively, as in~\eqref{eq:control-pts-elements}.
\begin{example} \label{example-1}
As the first example let us take
\begin{equation} \label{example-1-C1And2}
\cF^{(1)} = \begin{bmatrix}
(0,0) & \left(\frac{1}{4}, \frac{1}{2}\right) & \left(0,1\right)\\[1mm]
\left(x_1,y_1\right) & \left(\frac{3}{4}, 1\right) &\\[1mm]
\left(\frac{6}{5}, \frac{3}{4}\right)
\end{bmatrix}, \quad
\cF^{(2)} = \begin{bmatrix}
\left(0,0\right) & \left(\frac{1}{4},\frac{1}{2}\right) & (0,1) \\[1mm]
 \left(-\frac{2}{3},-\frac{1}{5}\right) & \left(x_2,y_2\right) & \left(-\frac{7}{10},\frac{6}{5}\right) \\[1mm]
 (-1,0) & \left(-\frac{5}{4},\frac{1}{2}\right) & (-1,1)
\end{bmatrix}, 
\end{equation}
for which $\beta(v) = \left(v-\frac{1}{2}\right)^2+1$.
For $\left(x_1,y_1\right) = \left(\frac{1}{2}, -\frac{1}{5}\right)$ and $\left(x_2,y_2\right) = \left(-\frac{1}{2},\frac{2}{3}\right)$ (see Figure~\ref{fig:example1}, left)
we compute that 
\begin{linenomath}\begin{align*}
& \wt{\alpha}_1(v) =\frac{1}{10} \left(14 v^2-11 v+12\right), \quad 
\wt{\alpha}_2(v) =\frac{1}{15} \left(-10 v^3+31 v^2-22 v-17\right),\\
& \beta_1(v) =\frac{1}{10} (4 v+1),\quad \beta_2(v) = \frac{1}{30} \left(-8 v^3-6 v^2+79 v-32\right), 
\end{align*}\end{linenomath}
and $\q=\gcd(\wt{\alpha}_1,\wt{\alpha}_2)=1$, $\pdw = d-1$. Moreover,
\[
\hat{\alpha}_1(v) = \frac{1}{20} (6 v-11),\quad \hat{\alpha}_2(v) = \frac{1}{60} (46 v-173), \quad \hat{\beta}_1(v) = \frac{1}{10} (4 v+1), \quad \hat{\beta}_2(v) = \frac{1}{60} (150 v-29)
\]
and $\hat{\alpha}_2(v) \hat{\beta}_1(v)- \hat{\alpha}_1(v) \hat{\beta}_2(v) = -\frac{133}{300} \beta(v)$, so the remainders $\hat{\td}$ and $\hat{\w}$ are given by Lemma~\ref{lemma-1-div}. The number of degrees of freedom corresponding to the interface is thus equal to $2d-1$.

If we change $y_2$ to $y_2 = \frac{1}{300} (154-225 x_2)$, then $\hat{\alpha}_2(v) = -\frac{25}{2} x_2 \hat{\alpha}_1(v)$, $\hat{\beta}_2(v) = -\frac{25}{2} x_2 \hat{\beta}_1(v)$, and thus $\hat{\td}$ and $\hat{\w}$ are expressed by \eqref{cond-t-w-2} or equivalently by \eqref{cond-t-w-3}.
The number of degrees of freedom corresponding to the interface remains equal to $2d-1$. Figure~\ref{fig:example1} (right) shows the line on which we can choose the point 
$ (x_2,y_2)$ to get to this special case, together with the quadrilateral mesh for $(x_2, y_2) = \left(-\frac{2}{5},\frac{61}{75}\right)$.

It is easy to compute that $\wt{\alpha}_1$ and $\wt{\alpha}_2$ would have a common linear factor $v-\xi$ for some $\xi\in \RR$ iff 
\begin{equation} \label{example1-lines-c}
y_1 = \frac{\xi  (2 \xi +1)}{2(2 \xi ^2-3 \xi +1)} + \frac{2}{1-2 \xi } x_1, \quad  y_2 = \frac{30 \xi ^3-59 \xi ^2+28 \xi -17}{30 \xi  \left(2 \xi ^2-3 \xi +1\right)} + \frac{2}{1-2 \xi } x_2.
\end{equation}
Further, for 
\begin{equation} \label{example1-lines-c1}
x_2 = \frac{ 2 \xi ^3+39 \xi ^2+84 \xi -17 + (4 \xi ^3+54 \xi ^2-228 \xi +170)x_1}{-150 \xi  \left(4 \xi ^2-5 \xi +1\right) + 600\xi (\xi -1)^2 x_1}
\end{equation}
we get that $\hat{\alpha}_2 \hat{\beta}_1 = \hat{\alpha}_1 \hat{\beta}_2$. 
Figure~\ref{fig:example2} (left) shows meshes for 
\begin{equation} \label{example1-values-c}
\xi=-2,\quad x_1=\frac{9}{10}, \quad y_1 = \frac{14}{25}, \quad x_2=-\frac{1}{2},\quad y_2 =\frac{41}{100}, 
\end{equation}
together with lines given by \eqref{example1-lines-c}. The dashed quadrilateral mesh, obtained by
\begin{equation} \label{example1-values-c1}
x_1=\frac{9}{10}, \quad y_1 = \frac{14}{25}, \quad x_2=-\frac{19}{45},\quad y_2 =\frac{397}{900}, 
\end{equation}
 corresponds to the case where \eqref{example1-lines-c1} holds true.  
 For both cases, \eqref{example1-values-c} and \eqref{example1-values-c1}, we have that $\q(v)=v+2$ and $\pdw = d$,
 so the number of degrees of freedom corresponding to the interface is raised by one, i.e., to $2d$.
 For \eqref{example1-values-c}, we compute  
 \[
 \quad \hat{\alpha}_1(v) =\frac{1}{50} (31-6 v),
 \quad \hat{\alpha}_2(v) =\frac{1}{60} (-61+10 v),
 \quad \hat{\beta}_1(v) =\frac{1}{50} (51-76 v),
 \quad \hat{\beta}_2(v) =\frac{1}{30} (-53+75 v),
 \]
 and $\hat{\td}$ and $\hat{\w}$ are expressed by \eqref{eq-divisionWithbetaSol1} with two free parameters $\mu_1$ and $\mu_2$.  
 For \eqref{example1-values-c1},   
 \[
 \quad \hat{\alpha}_1(v) =\frac{1}{50} (31-6 v),
 \quad \hat{\alpha}_2(v) =-\frac{1}{36} (31- 6 v),
 \quad \hat{\beta}_1(v) = \frac{1}{50} (51-76 v),
 \quad \hat{\beta}_2(v) =-\frac{1}{36} (51- 76 v),
 \]
 and $\hat{\td}$, $\hat{\w}$ are expressed by \eqref{cond-t-w-2}. 
 
As a final special case we compute that for
\begin{equation} \label{example1-d1}
x_2= \frac{24 x_1 y_1-48 x_1^2+54 x_1-44 y_1-17}{60 \left(2 x_1-y_1\right)}, \quad
y_2=\frac{-48 x_1 y_1+84 x_1+24 y_1^2-8 y_1-17}{60 \left(2 x_1-y_1\right)},
\end{equation}
$\wt{\alpha}_1$ and $\wt{\alpha}_2$ have a common quadratic factor. Choosing $x_1=\frac{1}{2}$, $y_1=-\frac{1}{5}$ 
(see Figure~\ref{fig:example2} (right)), we compute that
\begin{linenomath}\begin{align*}
& \q(v)=\wt{\alpha}_1(v), \quad {\alpha}_1(v) =\hat{\alpha}_1(v)=1,  \quad 
{\alpha}_2(v) = \hat{\alpha}_2(v) = \frac{1}{90}(13 v-85), \\
& \hat{\beta}_1(v) = \frac{1}{10}(4 v+1),
\quad \hat{\beta}_2(v) = -\frac{1}{36}(11 v+6), 
\end{align*}\end{linenomath}
and $\hat{\td}$, $\hat{\w}$ are given by  \eqref{eq-divisionWithbetaSol1}. Since $\pdw = d+1$, the trace $\tr$ and normal derivative $\w$ are expressed by $2d+1$ free parameters. If in addition to \eqref{example1-d1}
$y_1 = \frac{1}{4}(1-2x_1)$ or $y_1 = -\frac{17}{12} + 2 x_1$, then 
$\hat{\alpha}_2\hat{\beta}_1 = \hat{\alpha}_1 \hat{\beta}_2$. 
The red lines in Figure~\ref{fig:example2} (right) demonstrate these positions of a point $(x_1,y_1)$ - clearly, the position of a point $(x_2,y_2)$ changes according to \eqref{example1-d1}. 
For $x_1=\frac{7}{10}$, $y_1 = -\frac{1}{10}$ (dashed mesh on Figure~\ref{fig:example2}, right) it holds that $\q=\beta$.  In this case
\[
 {\alpha}_1(v) =\hat{\alpha}_1(v)=\frac{6}{5},  \quad 
{\alpha}_2(v) = \hat{\alpha}_2(v) = -\frac{4}{75}(17+v), \quad \hat{\beta}_1(v) = \hat{\beta}_2(v) = 0, 
\]
and so $\hat{\td}(v) = \mu_1 + \mu_2 v$, $\hat{\w}(v)=0$. Since $\pdw = d+1$, the trace and normal derivative are expressed with  $2d+1$ free parameters, where $d$ parameters correspond only to the trace and $d+1$ to the normal derivative. 
\begin{figure}[htb]
\centering\footnotesize
 \begin{minipage}{0.45\textwidth}
\includegraphics[width=1\textwidth]{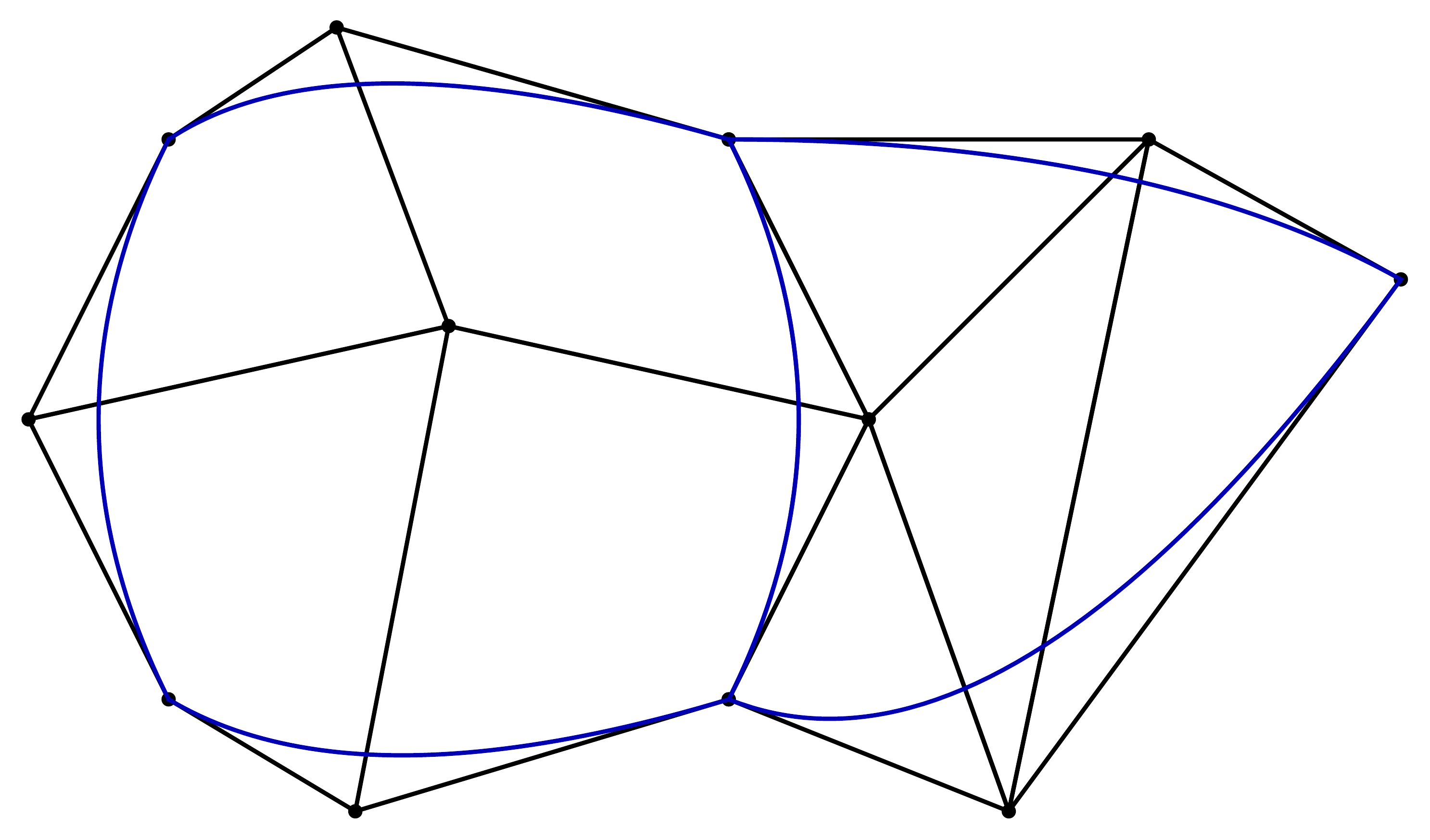}
\end{minipage}
 \begin{minipage}{0.45\textwidth}
\includegraphics[width=1\textwidth]{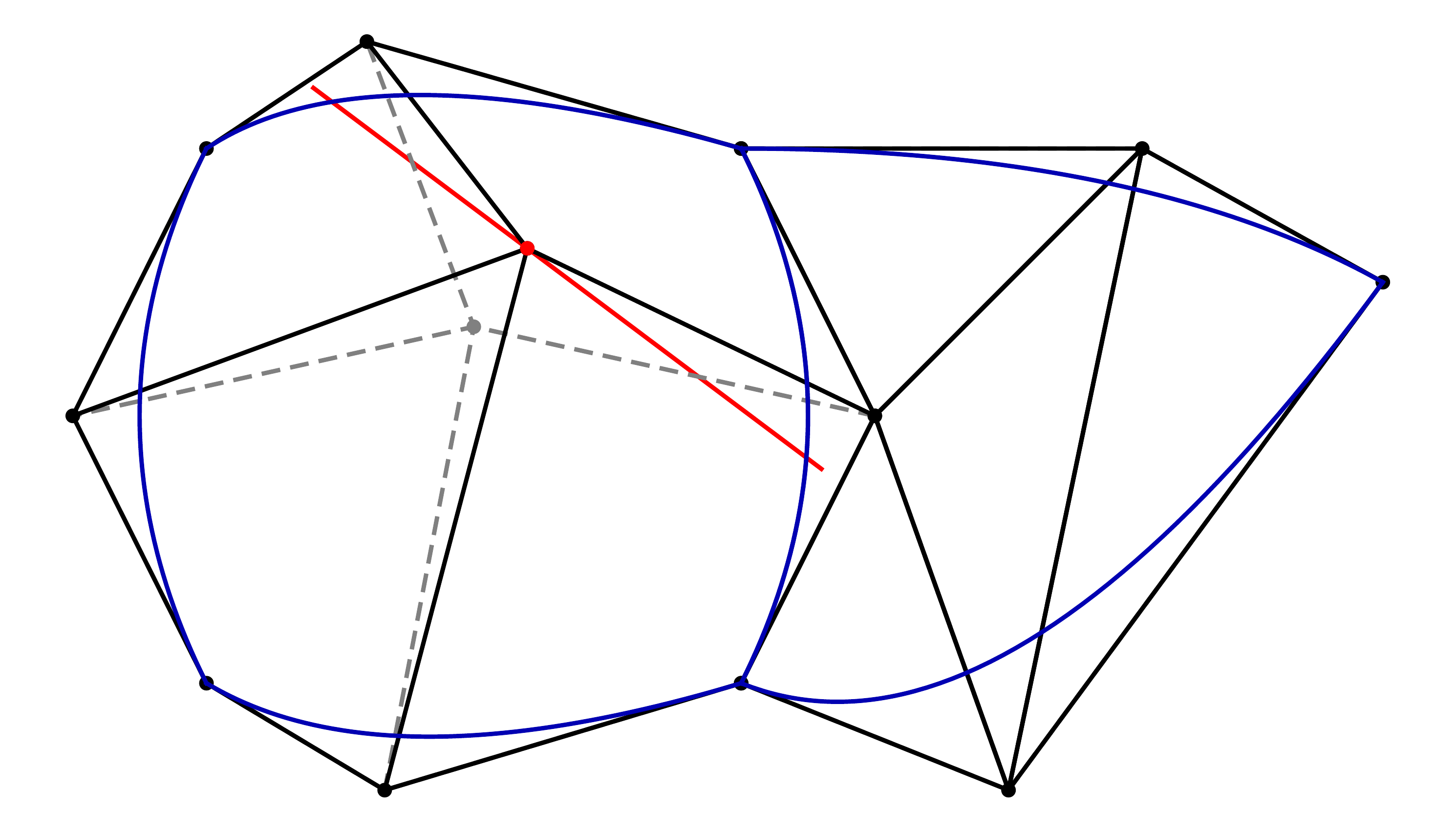}
\end{minipage}
\caption{Examples of mixed (bi-)quadratic triangular and quadrilateral mesh elements from Example~\ref{example-1} with different special cases.}
\label{fig:example1}
\end{figure}
\begin{figure}[htb]
\centering\footnotesize
 \begin{minipage}{0.45\textwidth}
\includegraphics[width=0.9\textwidth]{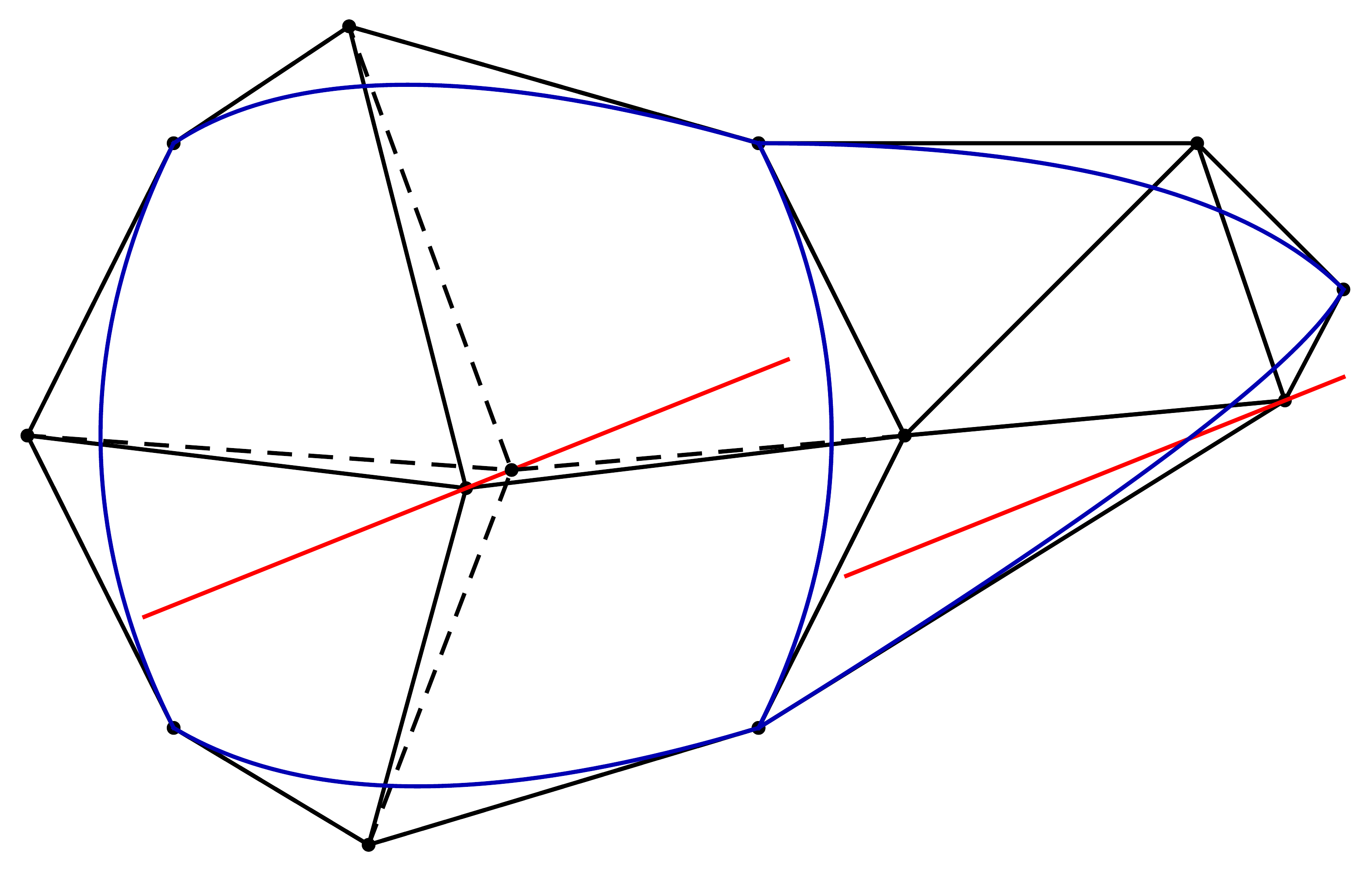}
\end{minipage}
 \begin{minipage}{0.45\textwidth}
\includegraphics[width=0.9\textwidth]{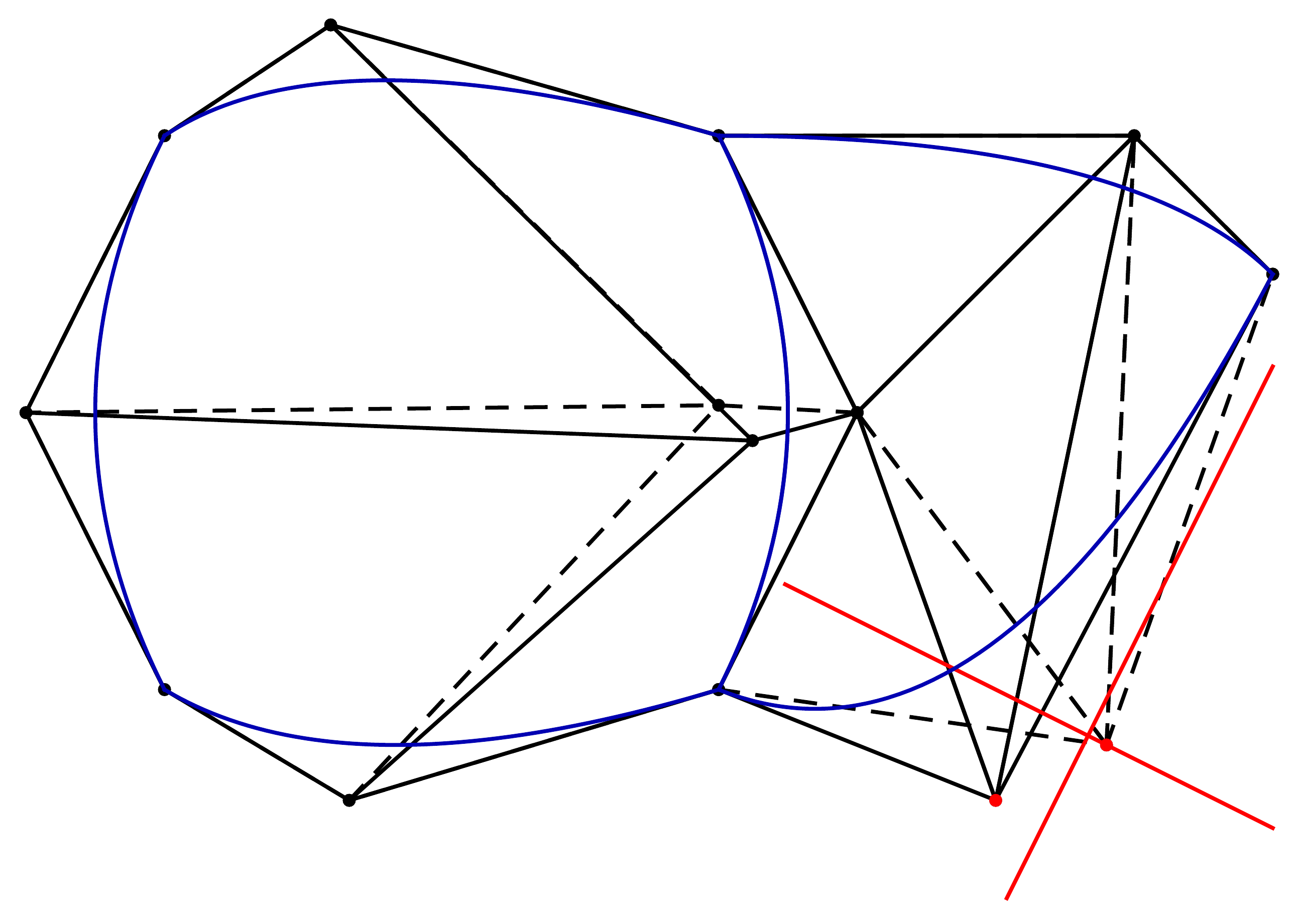}
\end{minipage}
\caption{Examples of mixed (bi-)quadratic triangular and quadrilateral mesh elements from Example~\ref{example-1} with different special cases.}
\label{fig:example2}
\end{figure}
\end{example}
\begin{example} \label{example-2}
As the next example, let us consider a linear interface, parameterized non-uniformly: 
\begin{equation} \label{example-2-C1And2}
\cF^{(1)} = \begin{bmatrix}
(0,0) & \left(0, \frac{1}{3}\right) & \left(0,1\right)\\[1mm]
\left(x_1,y_1\right) & \left(\frac{3}{4}, 1\right) &\\[1mm]
\left(\frac{6}{5}, \frac{3}{4}\right)
\end{bmatrix}, \quad
\cF^{(2)} = \begin{bmatrix}
\left(0,0\right) & \left(0,\frac{1}{3}\right) & (0,1) \\[1mm]
 \left(-\frac{2}{3},-\frac{1}{5}\right) & \left(x_2,y_2\right) & \left(-\frac{7}{10},\frac{6}{5}\right) \\[1mm]
 (-1,0) & \left(-\frac{5}{4},\frac{1}{2}\right) & (-1,1)
\end{bmatrix}.
\end{equation}
Here $\beta(v) = \frac{4}{9}(1+v)^2$ and $\rho(v)= \frac{2}{3}(1+v)$. 
Choosing $\left(x_1,y_1\right) = \left(\frac{1}{2}, -\frac{1}{5}\right)$ and $\left(x_2,y_2\right) = \left(-\frac{1}{2},\frac{2}{3}\right)$ 
we compute that $\q(v) = \rho(v)$,
\begin{linenomath}\begin{align*}
& {\alpha}_1(v) =\frac{1}{2}(2+v), \quad {\alpha}_2(v) =\frac{1}{15} \left(-20+10v-11 v^2\right), \\
&  \beta_1(v) =\frac{2}{15}(-3+13v)\rho(v),\quad \beta_2(v) = -\frac{2}{15} (3-16v+10v^2)\rho(v), 
\end{align*}\end{linenomath}
and
\[
\hat{\alpha}_1(v) = \frac{3}{4} \rho(v) + \frac{1}{2},\quad \hat{\alpha}_2(v) = \frac{16}{5} \rho(v)- \frac{41}{15}, \quad \hat{\beta}_1(v) = -\frac{32}{15} \rho(v), \quad \hat{\beta}_2(v) = -\frac{58}{15}\rho(v).
\]
The remainders $\hat{\td}$ and $\hat{\w}$ are given by case (1) in Proposition~\ref{proposition-nonuniform-edge}: $\hat{\td}(v) = \mu_1 \rho(v)$, $\hat{\w}(v) = 0$, so there are $2d-1$ degrees of freedom corresponding to the interface ($d$ for $\tr$ and $d-1$ for $\w$). 
The mesh is shown in Figure~\ref{fig:example3}, left, together with the (red) line $y_2 = -\frac{1}{900}(3067 + 3840 x_2)$ that shows the positions for a point 
$(x_2,y_2)$ for which the solution is given by case (2) in Proposition~\ref{proposition-nonuniform-edge}. The dashed mesh is a mesh obtained for $x_2= -\frac{85}{100}$. In this case the number of degrees of freedom for the interface is $2 d$, because $\pdw = d-2$ and 
\[
\hat{\td}(v) =\frac{2}{3}(1+v)\mu_1 + \frac{1}{3}\mu_2, \quad \hat{\w}(v) = \frac{64}{45}(1+v)\mu_2.
\]
Furthermore, it is straightforward to compute that $\beta$ divides $\wt{\alpha}_1$ and $\wt{\alpha}_2$ (so  $\pdw = d-1$) iff 
$x_1 = \frac{3}{8}$, $x_2 = -\frac{101}{120}$. These two lines are shown in Figure~\ref{fig:example3}, right, together with a mesh for 
\begin{equation}\label{example2-choice-3}
x_1 = \frac{3}{8}, \quad x_2 = -\frac{101}{120}, \quad y_1 = -\frac{1}{5}, \quad y_2= \frac{2}{3}.
\end{equation}
The dashed control mesh is obtained for
\begin{equation}\label{example2-choice-4}
x_1 = \frac{3}{8}, \quad x_2 = -\frac{101}{120}, \quad y_1 = \frac{1}{3}, \quad y_2= \frac{11}{60}.
\end{equation}
For \eqref{example2-choice-3}, $\q(v)=\beta(v)$, $\pdw = d-1$ and 
$\hat{\td}(v) =\mu_1 \rho(v) $, $\hat{\w}(v)=0$, so there are $d$ degrees of freedom for the trace $\tr$ and $d$ degrees of freedom for the normal derivative $\w$. 
For \eqref{example2-choice-4}, the polynomial $\beta$ divides also $\beta_1$ and $\beta_2$, so $\hat{\beta}_1(v)=\hat{\beta}_2(v) = 0$ and the remainders are given as $\hat{\td}(v) = \mu_1 \rho(v) + \frac{9}{8}\mu_2$, $\hat{\w}(v) = 0$. The trace $\tr$ and normal derivative $\w$ are expressed with  $2d+1$ free parameters, where $d+1$ parameters correspond only to $\tr$ and $d$ to $\w$.  

\begin{figure}[htb]
\centering\footnotesize
 \begin{minipage}{0.45\textwidth}
\includegraphics[width=0.9\textwidth]{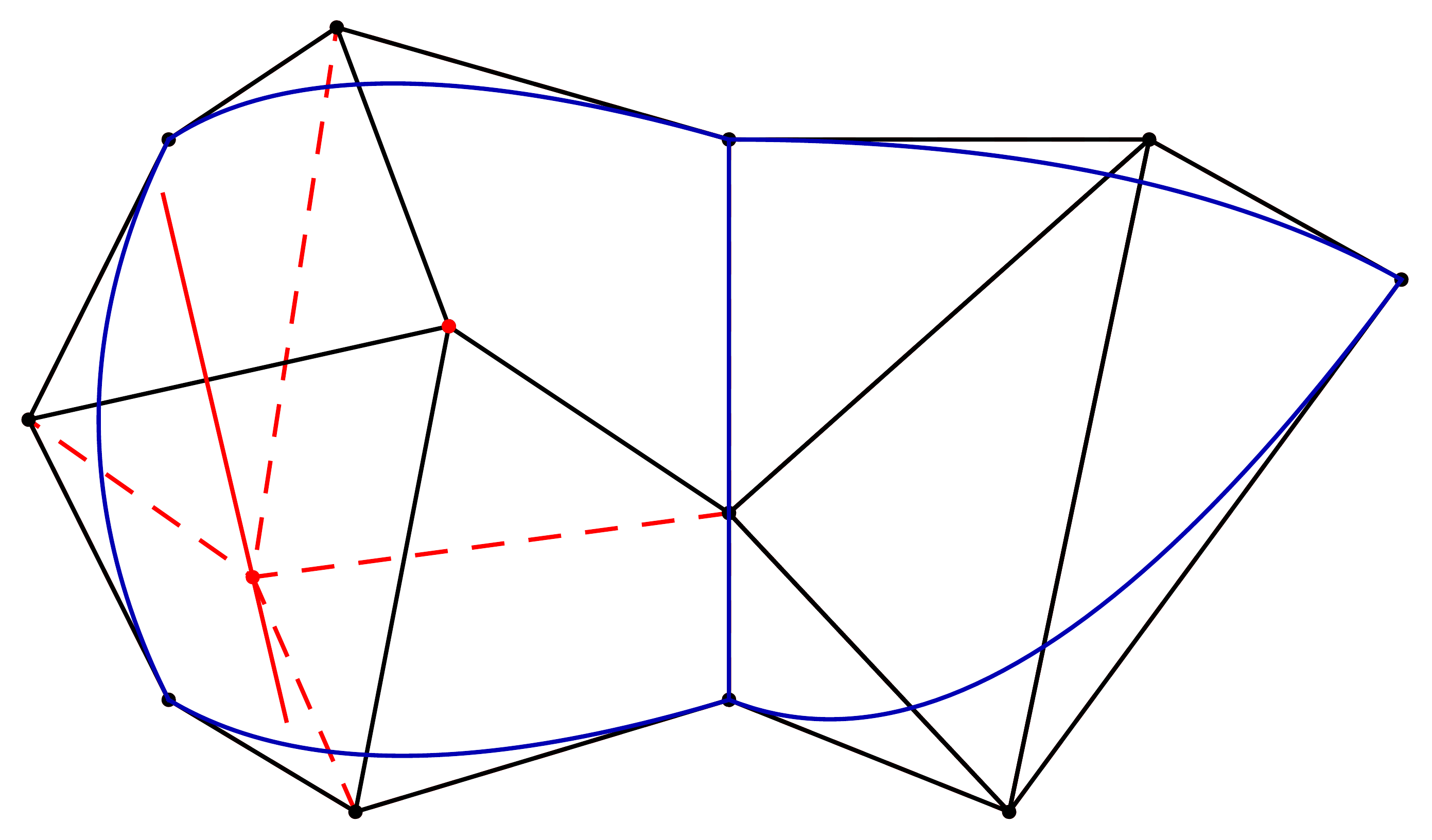}
\end{minipage}
 \begin{minipage}{0.45\textwidth}
\includegraphics[width=0.9\textwidth]{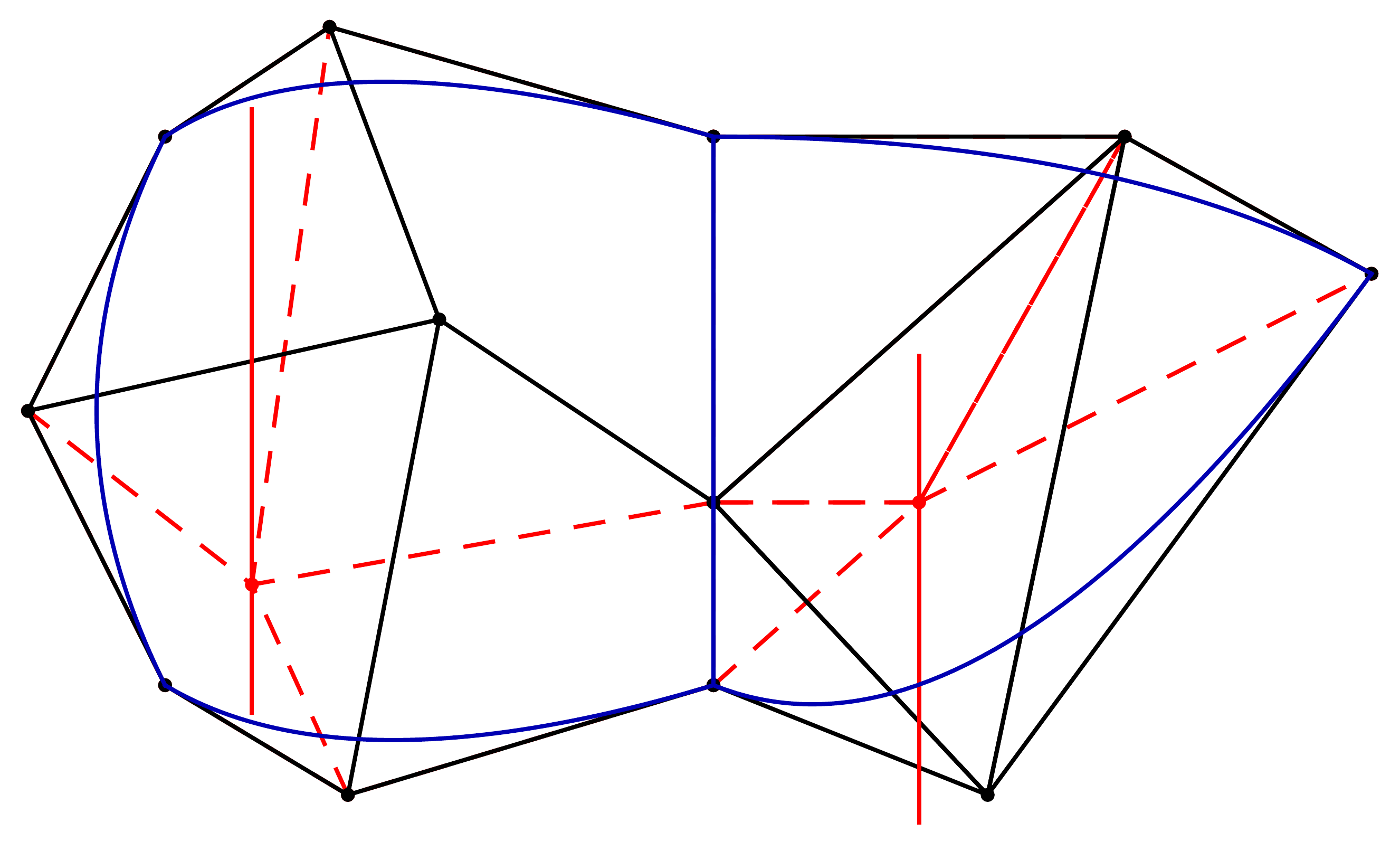}
\end{minipage}
\caption{Examples of mixed (bi-)quadratic triangular and quadrilateral mesh elements with non-uniformly parameterized linear interface from Example~\ref{example-2} with different special cases.}
\label{fig:example3}
\end{figure}
\end{example}

\begin{example} \label{example-3}
As the final example, we choose two mesh elements with uniformly parameterized linear interface, given by \eqref{example-2-C1And2}
with the control point $\cF_1 = \cF_{0,1}^{(1)} = \cF_{0,1}^{2)}$ replaced by 
$ \left(0, \frac{1}{2}\right)$. 
Then $\beta(v) = 1$, and the trace and the normal derivative are given by Proposition~\ref{proposition-linear-edge}. 
It is straightforward to compute that case (1) occurs iff 
$\frac{86}{15} - 4 x_1 + 4 x_2 - \frac{82}{15}y_1 - 8 x_2 y_1 - 6 y_2 + 8 x_1 y_2 \not= 0$. Else we are in case (2) which yields one additional degree of freedom. 
\end{example}

\subsection{Construction of isogeometric basis functions over two mesh elements} 

Let us demonstrate in the following the construction of a basis for $\mathcal{V}_\pd^1(\Omega)$, as presented in Section~\ref{sec:basis}, on a few examples. 
\begin{example} \label{example-basisFunctions}
Consider first the mesh from Figure~\ref{fig:example1}, left (first case of Example~\ref{example-1}). In this case we have $2\pd-1$ free parameters, i.e., $n_\tr = \pd-1$ parameters corresponding to the trace, $n_\w = \pd-2$ parameters corresponding to the normal derivative and two parameters $\mu_1$, $\mu_2$ for $\hat{\td}$, $\hat{\w}$. Let us first consider the case $d=6$. Then the interpolation functionals are chosen as \eqref{eq-IntFun-PI-even-d} and the interpolation problem is defined by 
\begin{linenomath}\begin{align*}
& \lambda_0^{(0)} \tr = a_1, \quad \lambda_0^{(1)} \tr = 6 a_2, \quad \lambda_1^{(0)} \tr = a_3, \quad \lambda_1^{(1)} \tr = 6 a_4, \quad  \lambda_0^{(0)} \w = 5 a_5, \quad  \lambda_1^{(0)} \w = 5 a_6,  \\
& \lambda_{1/2}^{(0)} \tr = a_7, \quad \lambda_{1/3}^{(0)} \w = 5 a_8, \quad \lambda_{2/3}^{(0)} \w = 5 a_9, \quad 
\mu_1 = a_{10}, \quad \mu_2 = a_{11}  
\end{align*}\end{linenomath}
with some vector $\bfm{a} = (a_i)_{i=1}^{2d-1} \in \RR^{2d-1}$. Taking $\bfm{a} = \bfm{e}_i$, $i=1,2,\dots, 2d-1$, where $\bfm{e}_i$ is the $i$-th basis vector in $\RR^{2d-1}$, we obtain $11$ basis functions shown in Figure~\ref{fig:example3-b} 
where for the last two functions we have additionally applied the Gram--Schmidt orthogonalization. The collocation matrix corresponding to the interpolation equations has the condition number (in the Euclidean norm) equal to $40.35$. This confirms that we can numerically compute these basis functions in a stable way. 
For $\pd=7$, the basis functions follow from \eqref{eq-IntFun-PI-odd-d}, 
\begin{linenomath}\begin{align*}
& \lambda_0^{(0)} \tr = a_1, \quad \lambda_0^{(1)} \tr = 7 a_2, \quad \lambda_0^{(2)} \tr = 7\cdot 6\, a_3, \quad \lambda_1^{(0)} \tr = a_4, \quad \lambda_1^{(1)} \tr = 7 a_5, \quad \lambda_1^{(2)} \tr = 7\cdot 6\, a_6,\\
&  \lambda_0^{(0)} \w = 6 a_7, \quad  \lambda_0^{(1)} \w = 6\cdot 5\, a_8,
 \quad  \lambda_1^{(0)} \w = 6 a_9, \quad  \lambda_1^{(1)} \w = 6\cdot5\, a_{10}, \\
& \lambda_{1/2}^{(0)} \w = a_{11}, \quad 
\mu_1 = a_{12}, \quad \mu_2 = a_{13},
\end{align*}\end{linenomath}
and are shown in Figure~\ref{fig:example4}. Again, the numerical computations are stable since the condition number of the collocation matrix equals $37.16$.
\begin{figure}[htb]
\centering\footnotesize
\includegraphics[width=.24\textwidth]{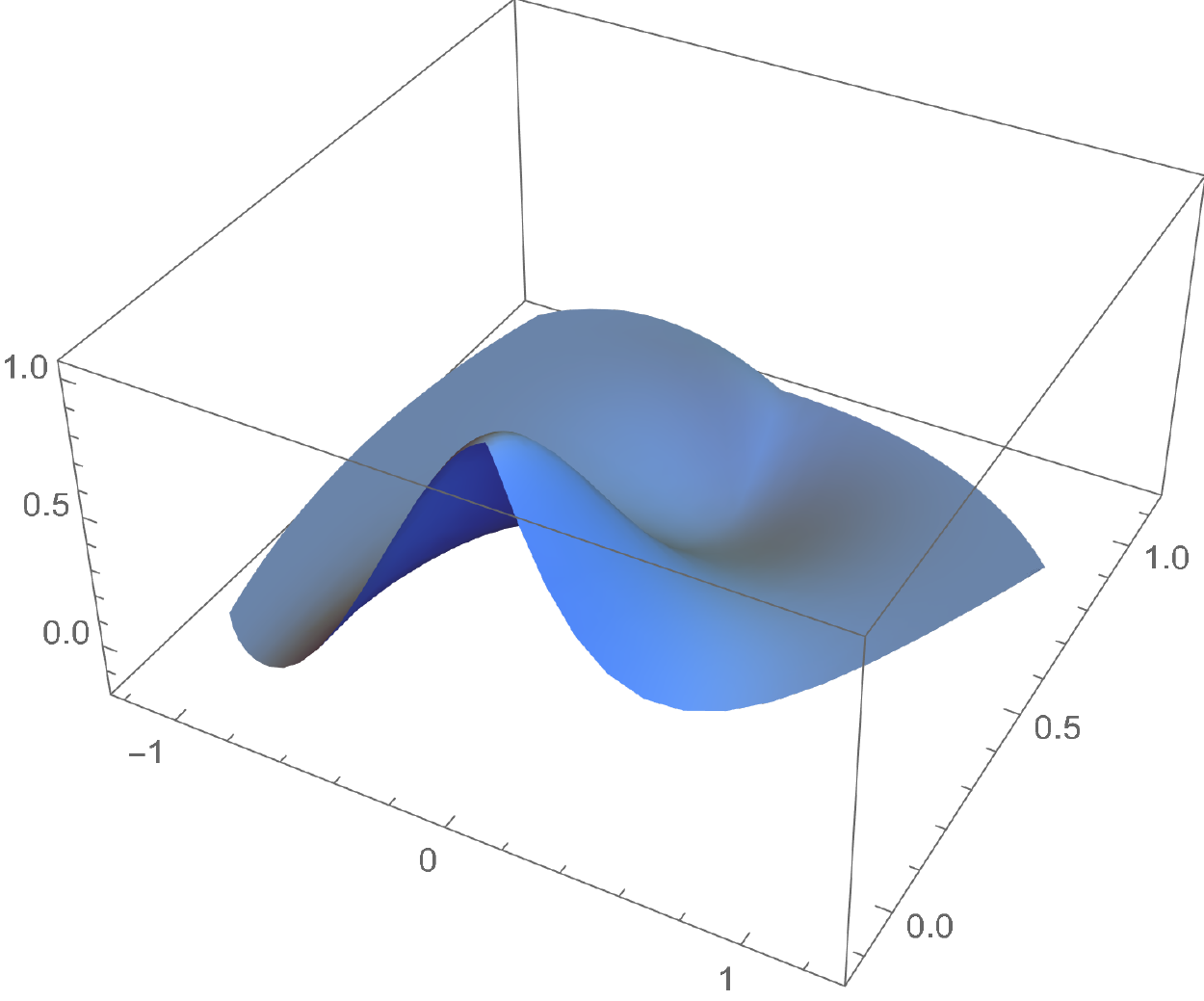}
\includegraphics[width=.24\textwidth]{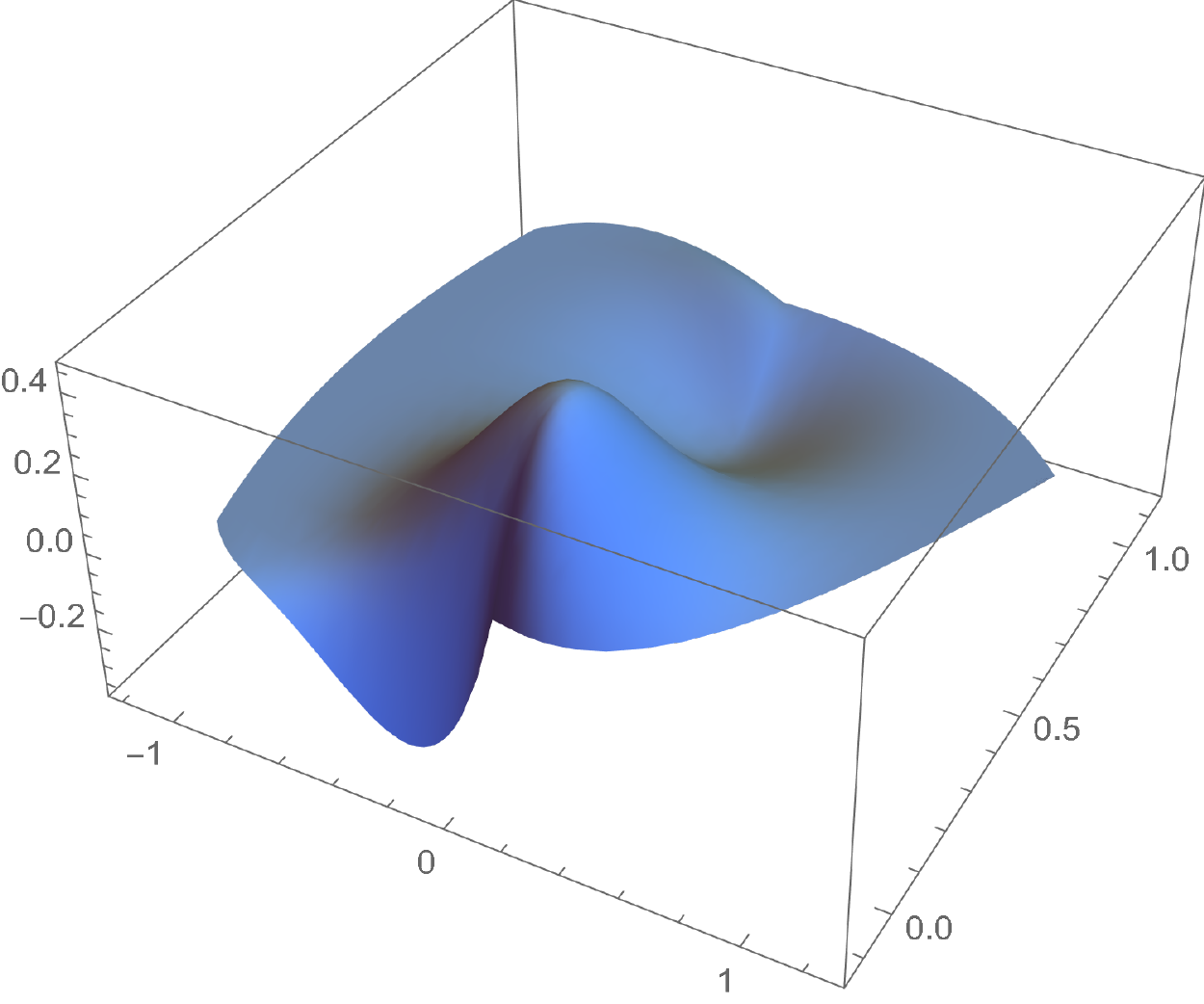}
\includegraphics[width=.24\textwidth]{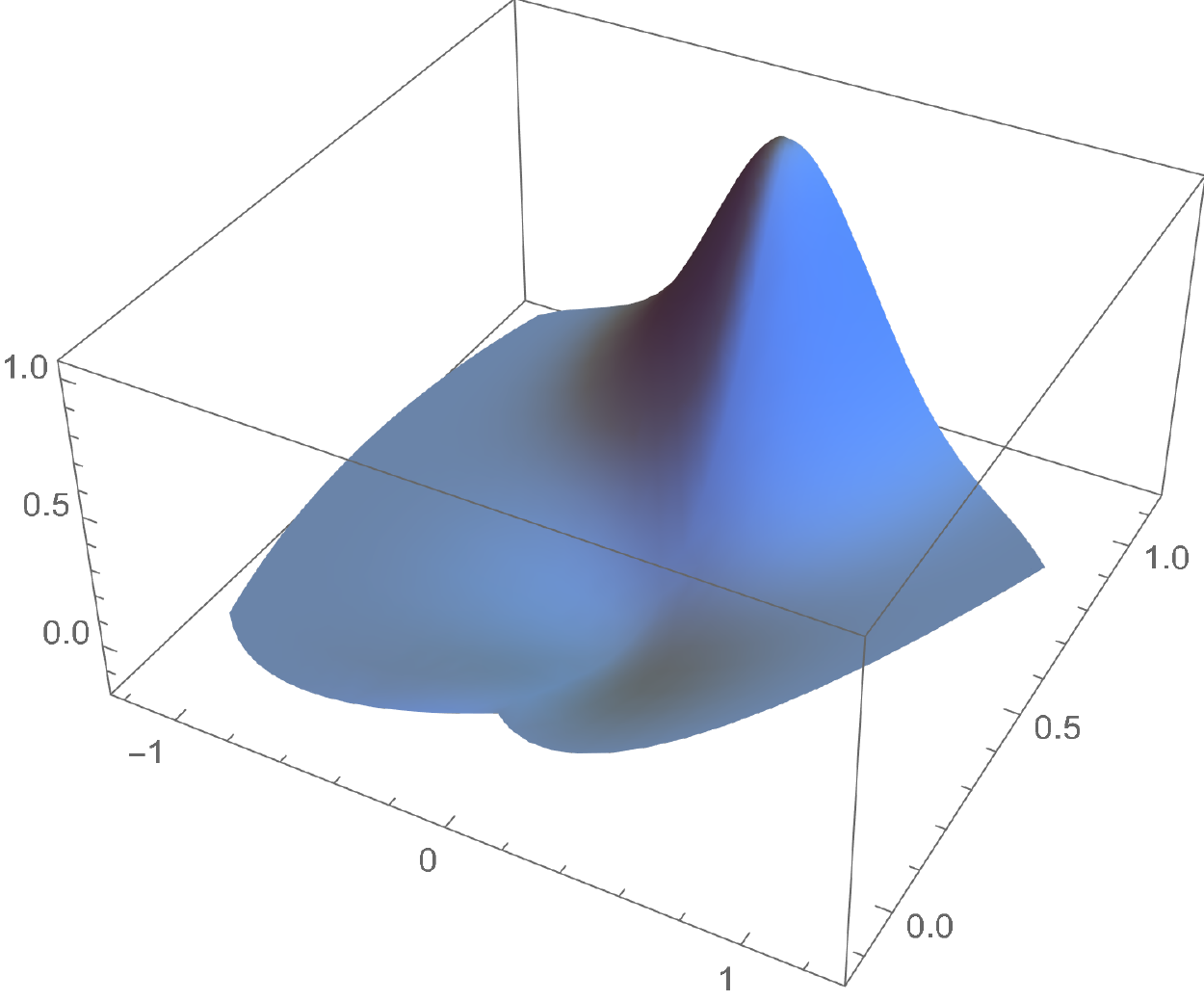}
\includegraphics[width=.24\textwidth]{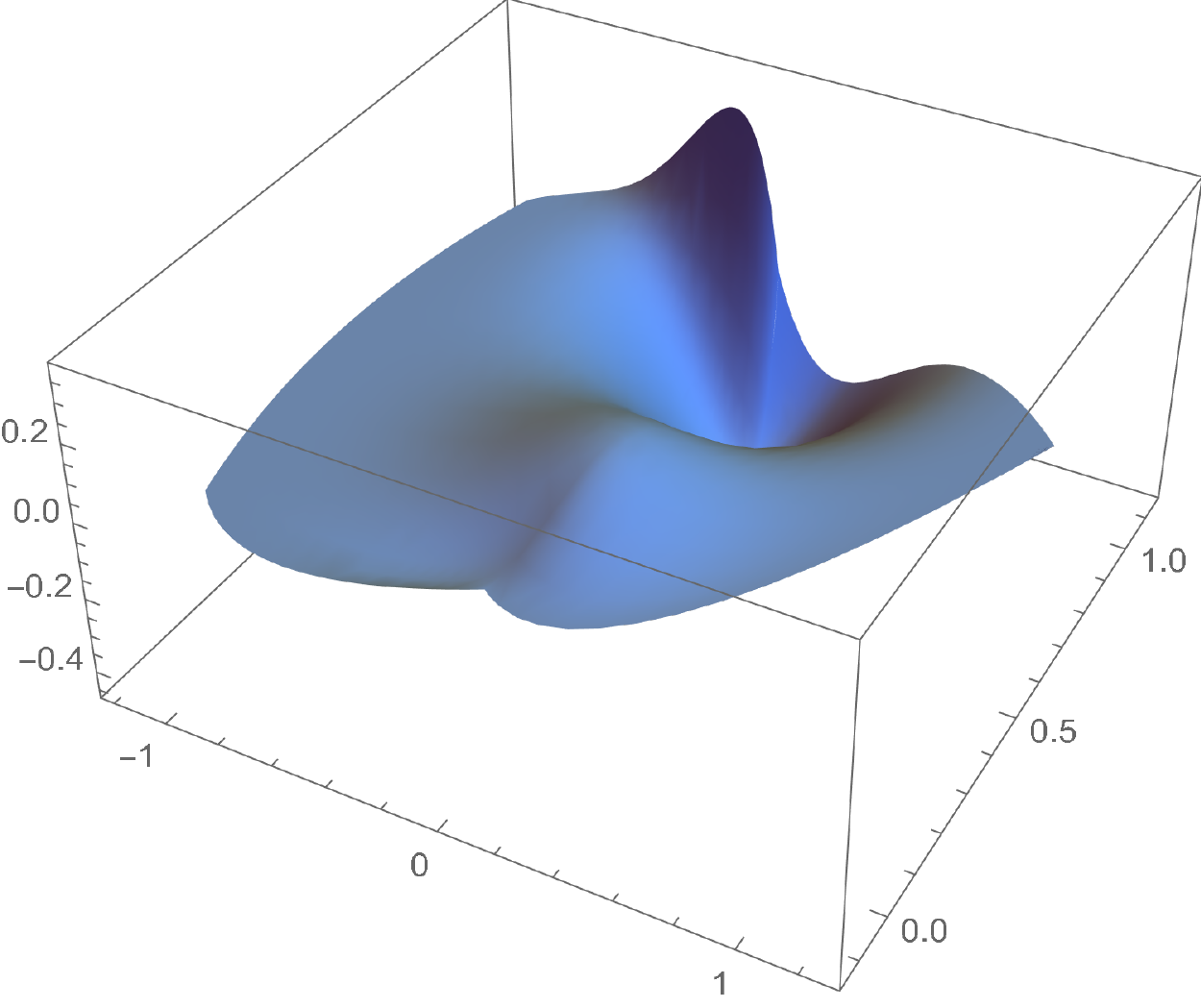}
\includegraphics[width=.24\textwidth]{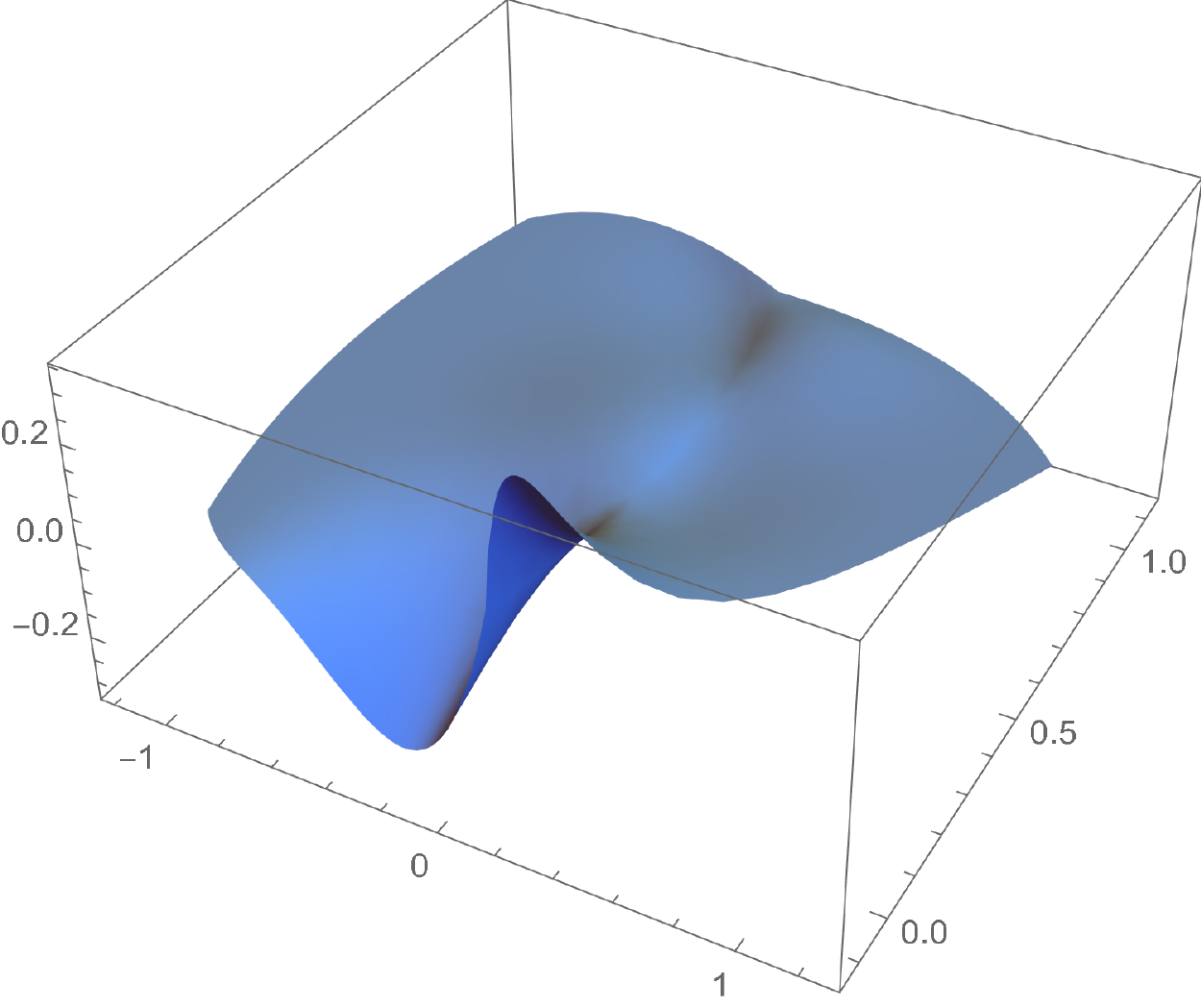}
\includegraphics[width=.24\textwidth]{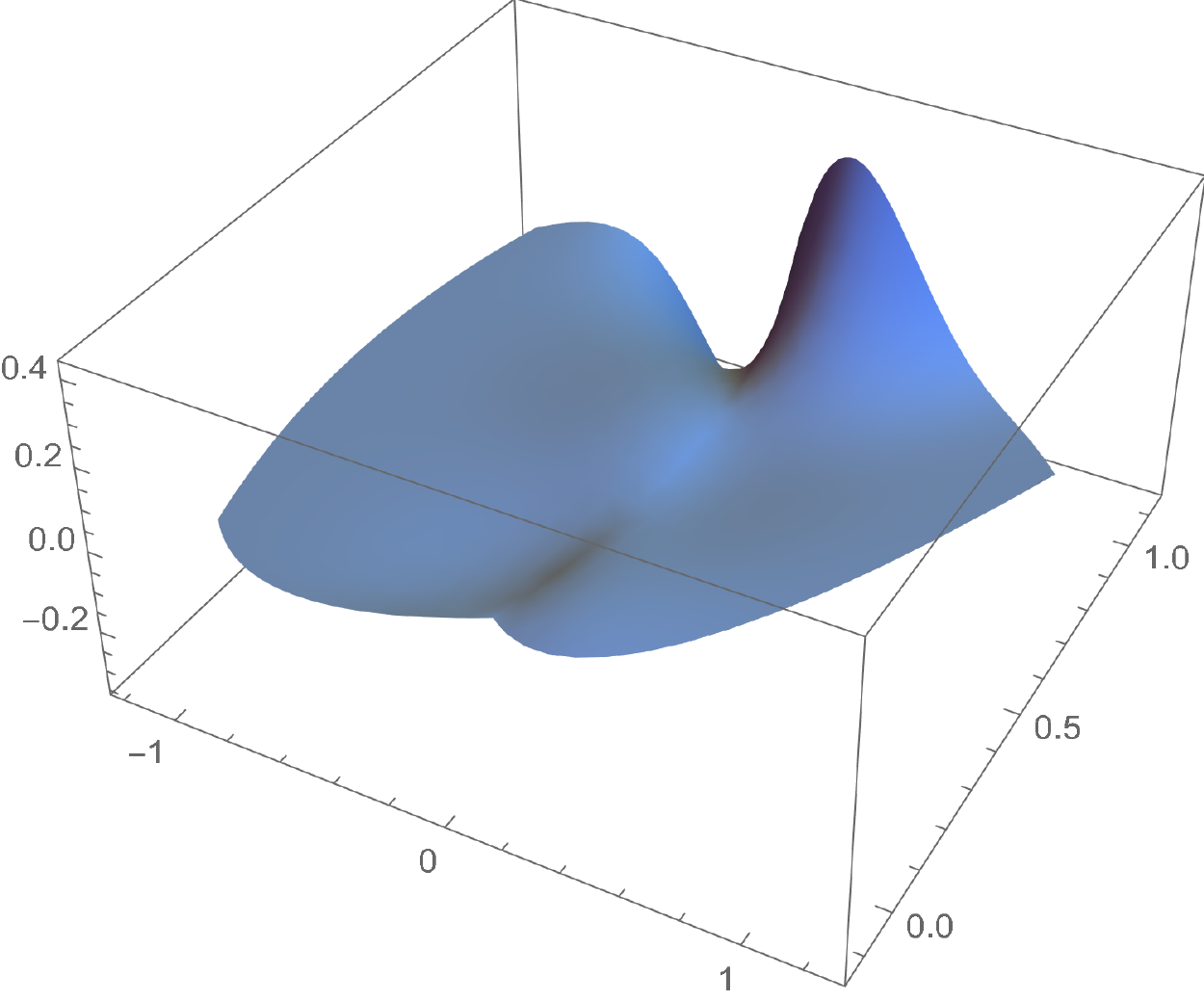}
\includegraphics[width=.24\textwidth]{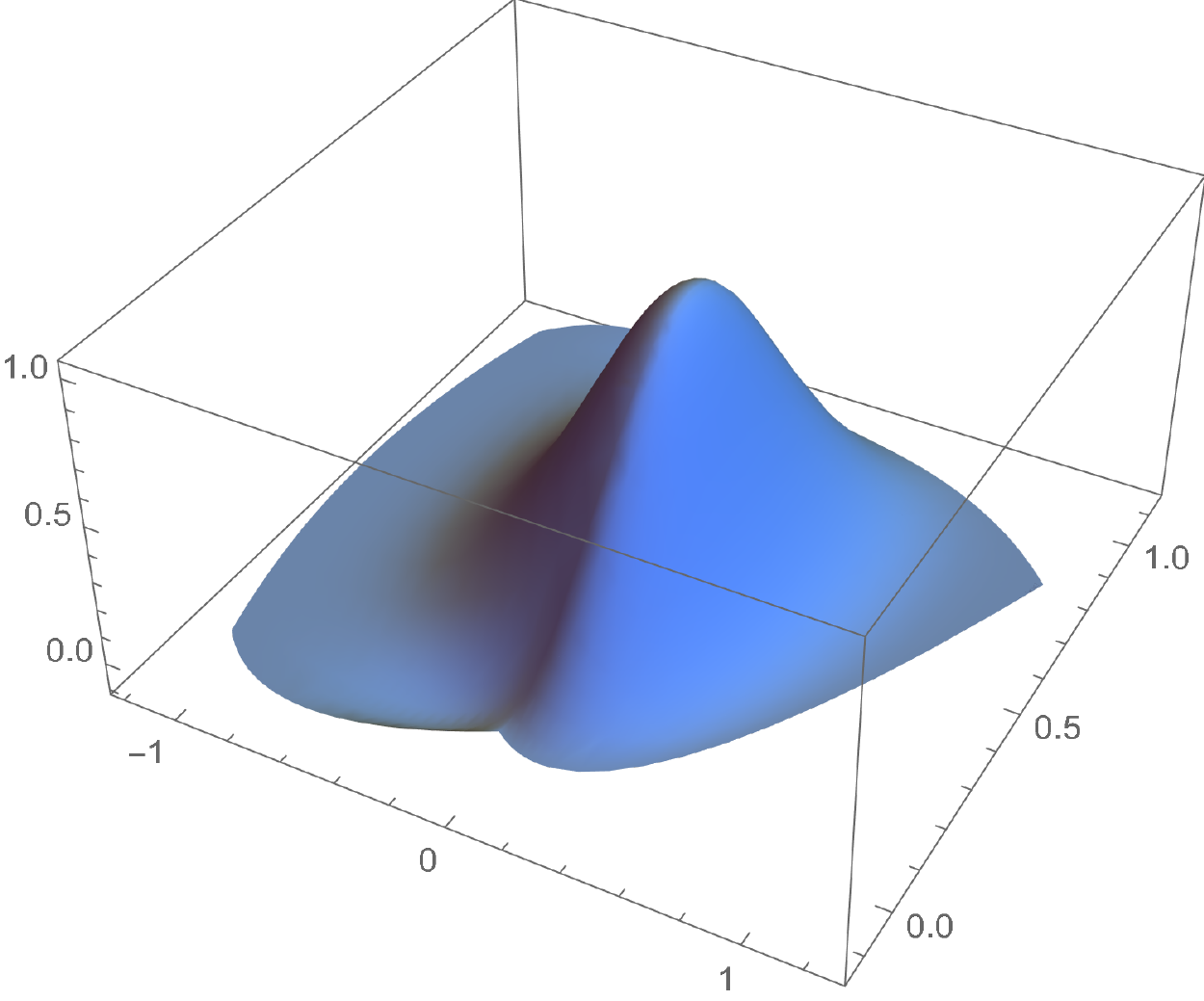}
\includegraphics[width=.24\textwidth]{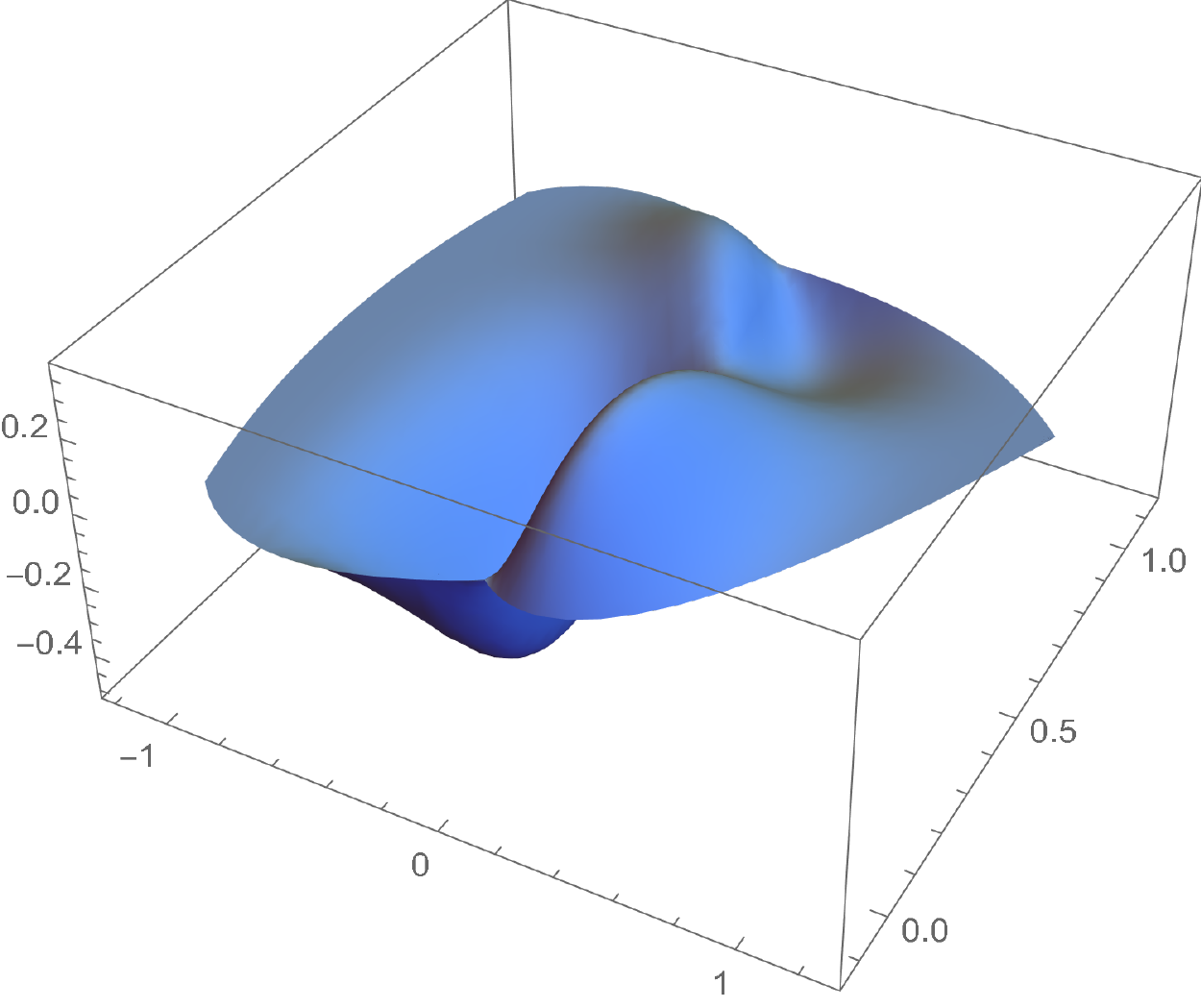}
\includegraphics[width=.24\textwidth]{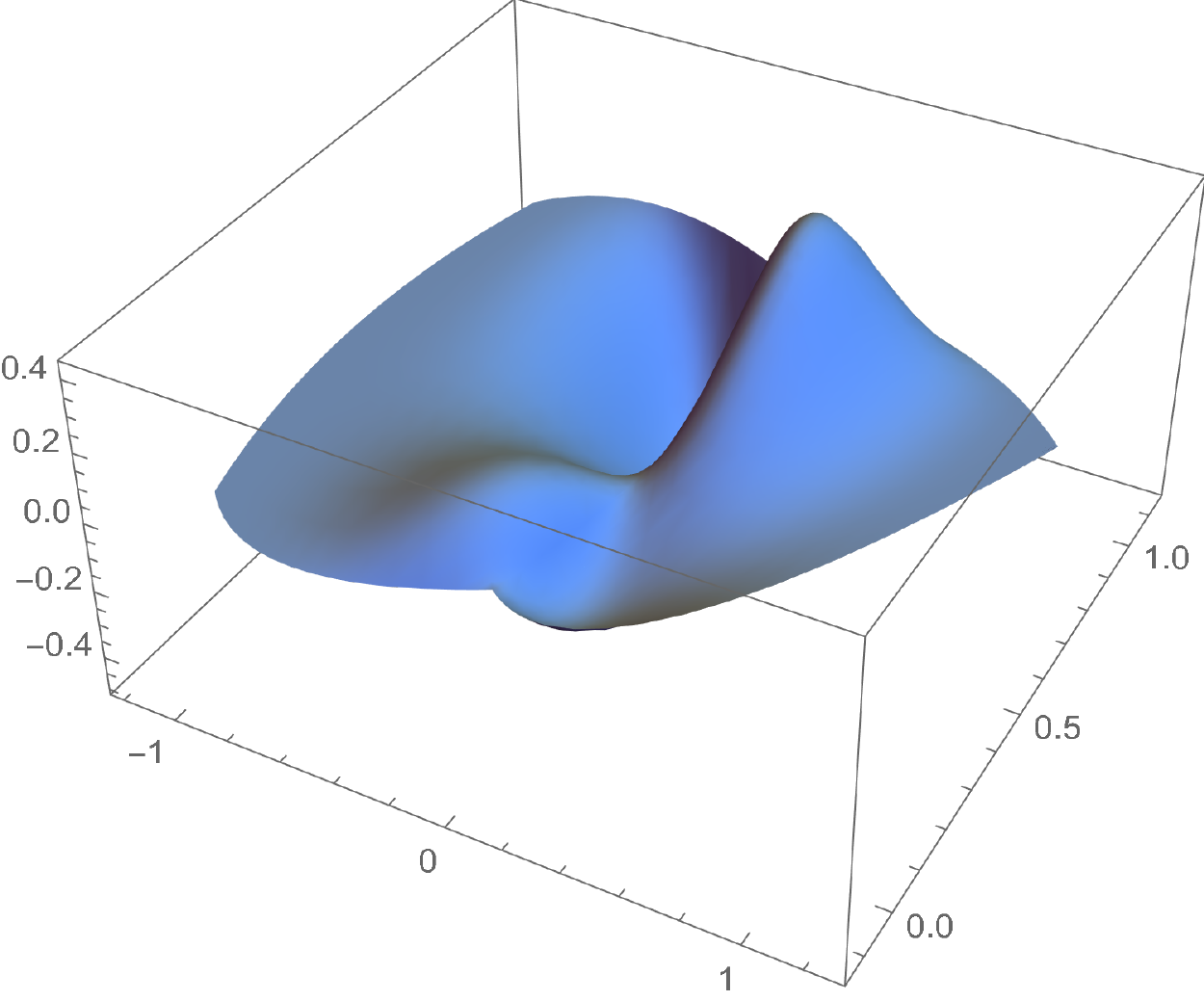}
\includegraphics[width=.24\textwidth]{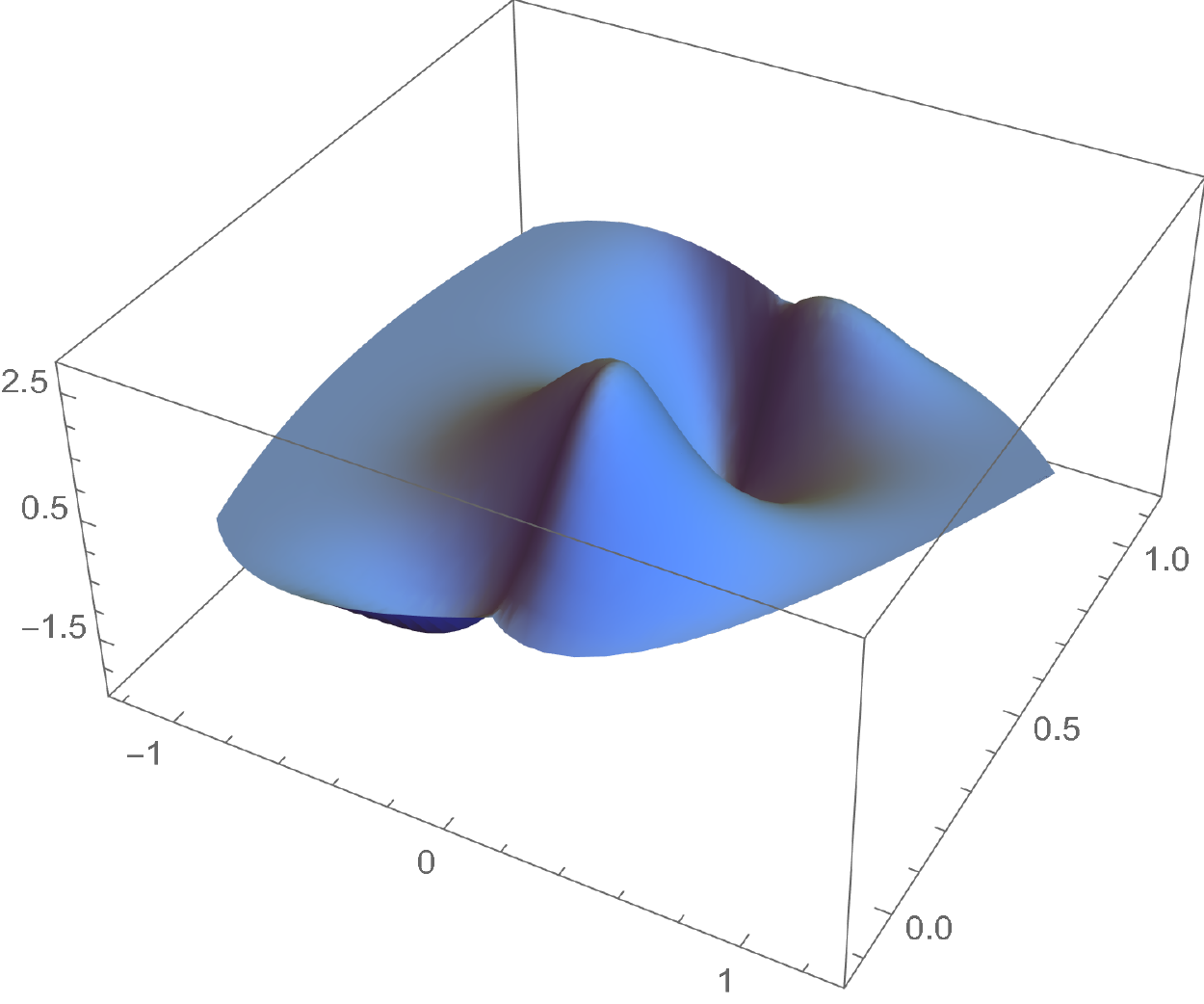}
\includegraphics[width=.24\textwidth]{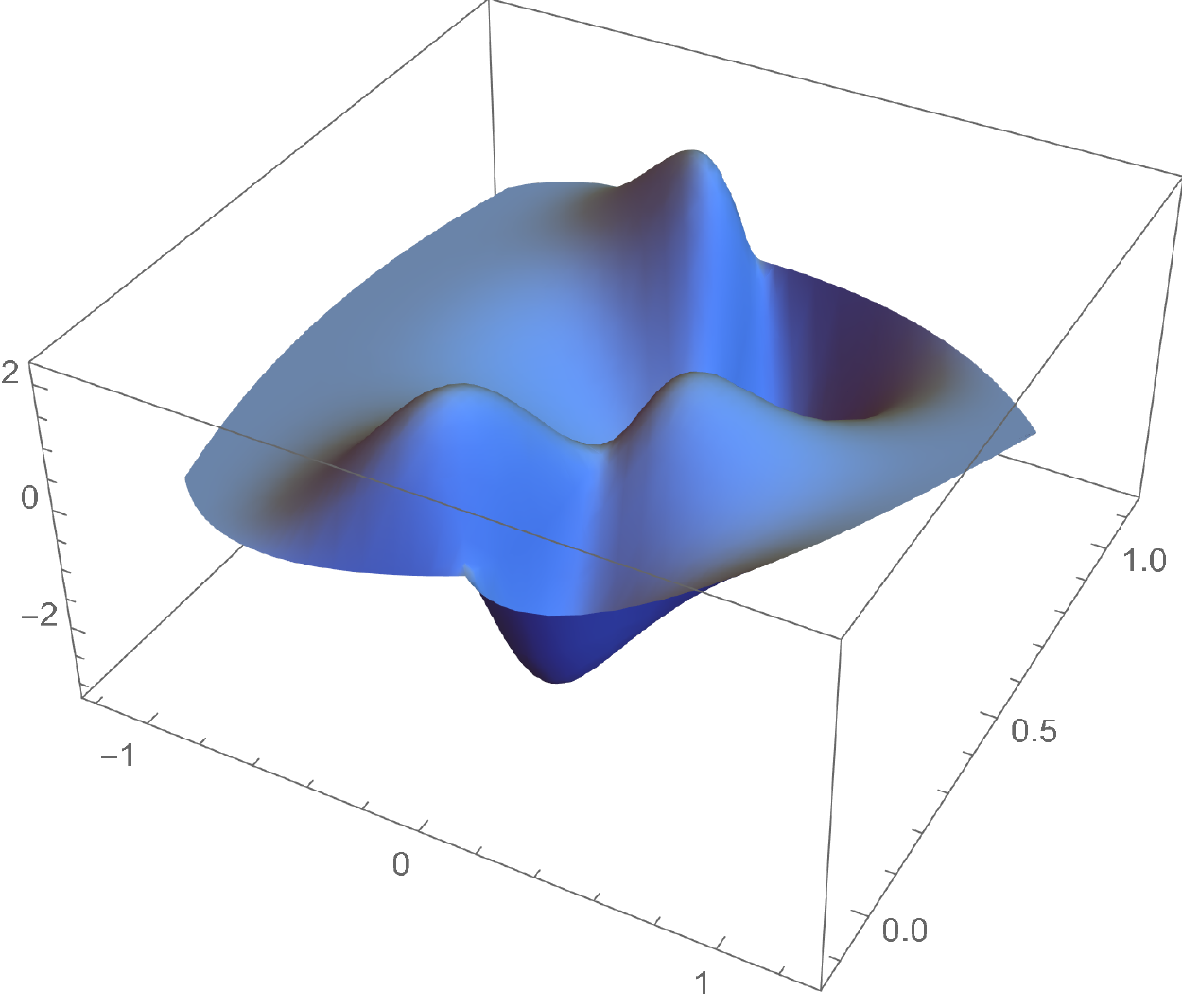}
\caption{Basis functions from Example~\ref{example-basisFunctions} for $\pd=6$.}
\label{fig:example3-b}
\end{figure}
\begin{figure}[htb]
\centering\footnotesize
\includegraphics[width=.24\textwidth]{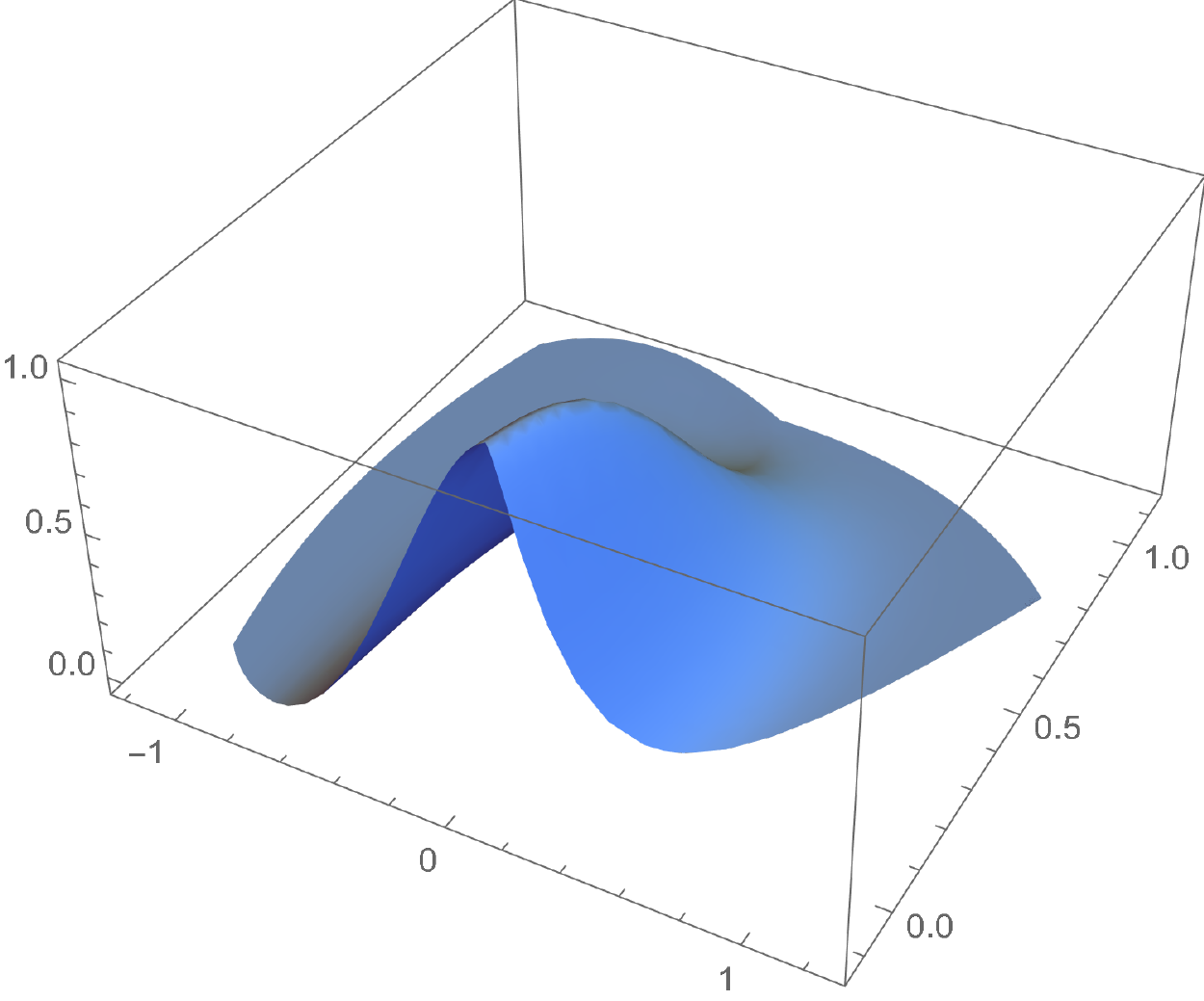}
\includegraphics[width=.24\textwidth]{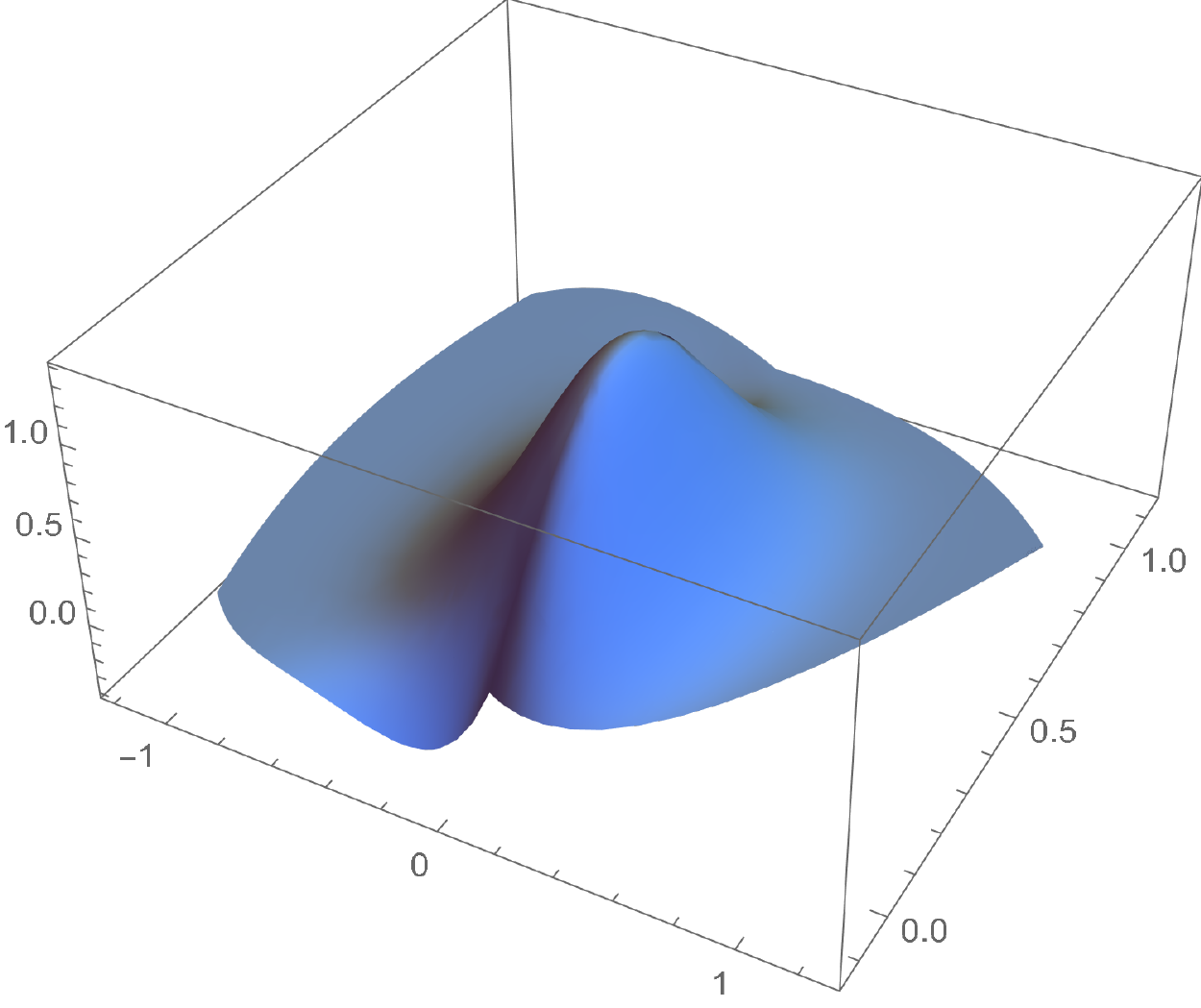}
\includegraphics[width=.24\textwidth]{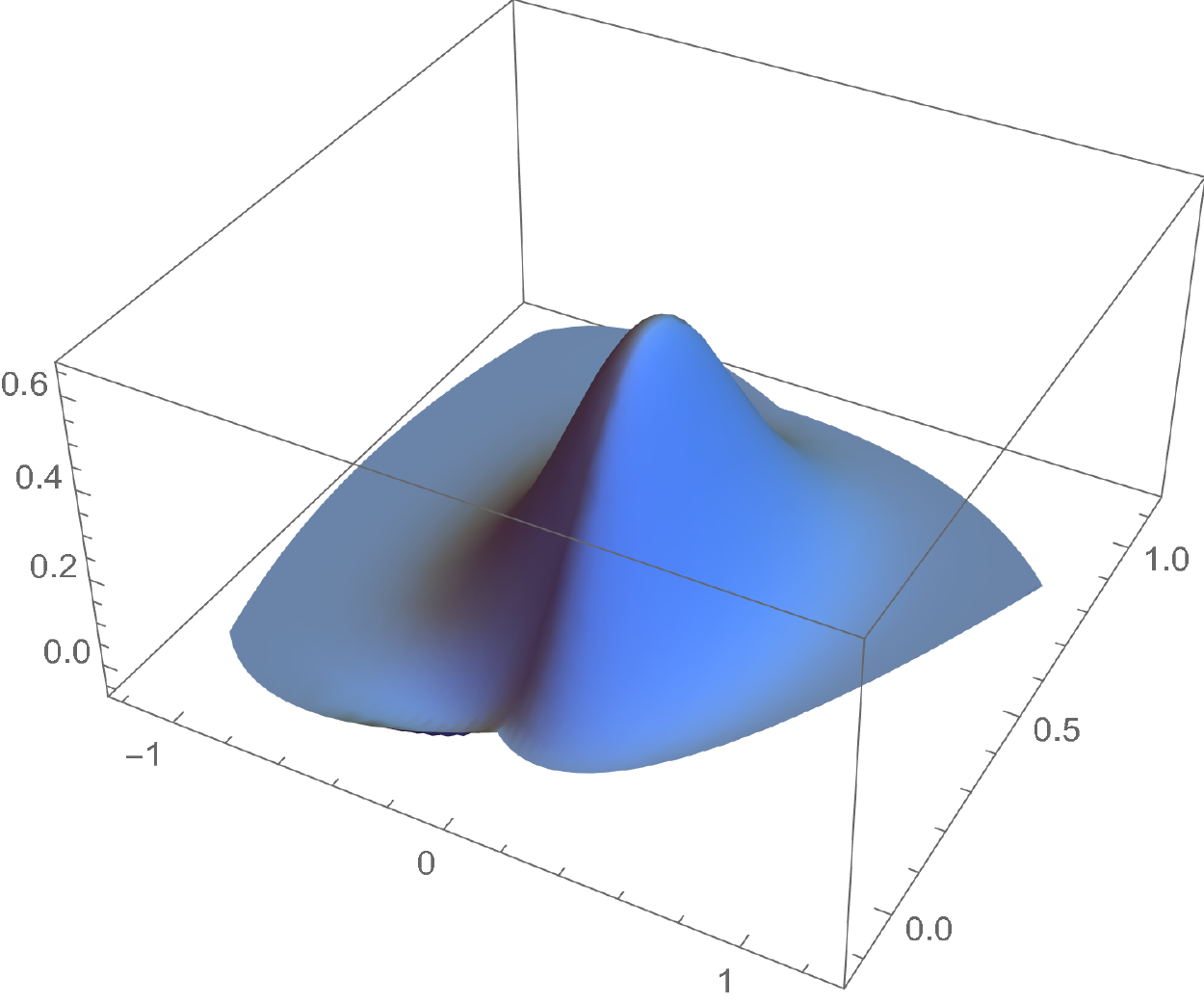}\\
\includegraphics[width=.24\textwidth]{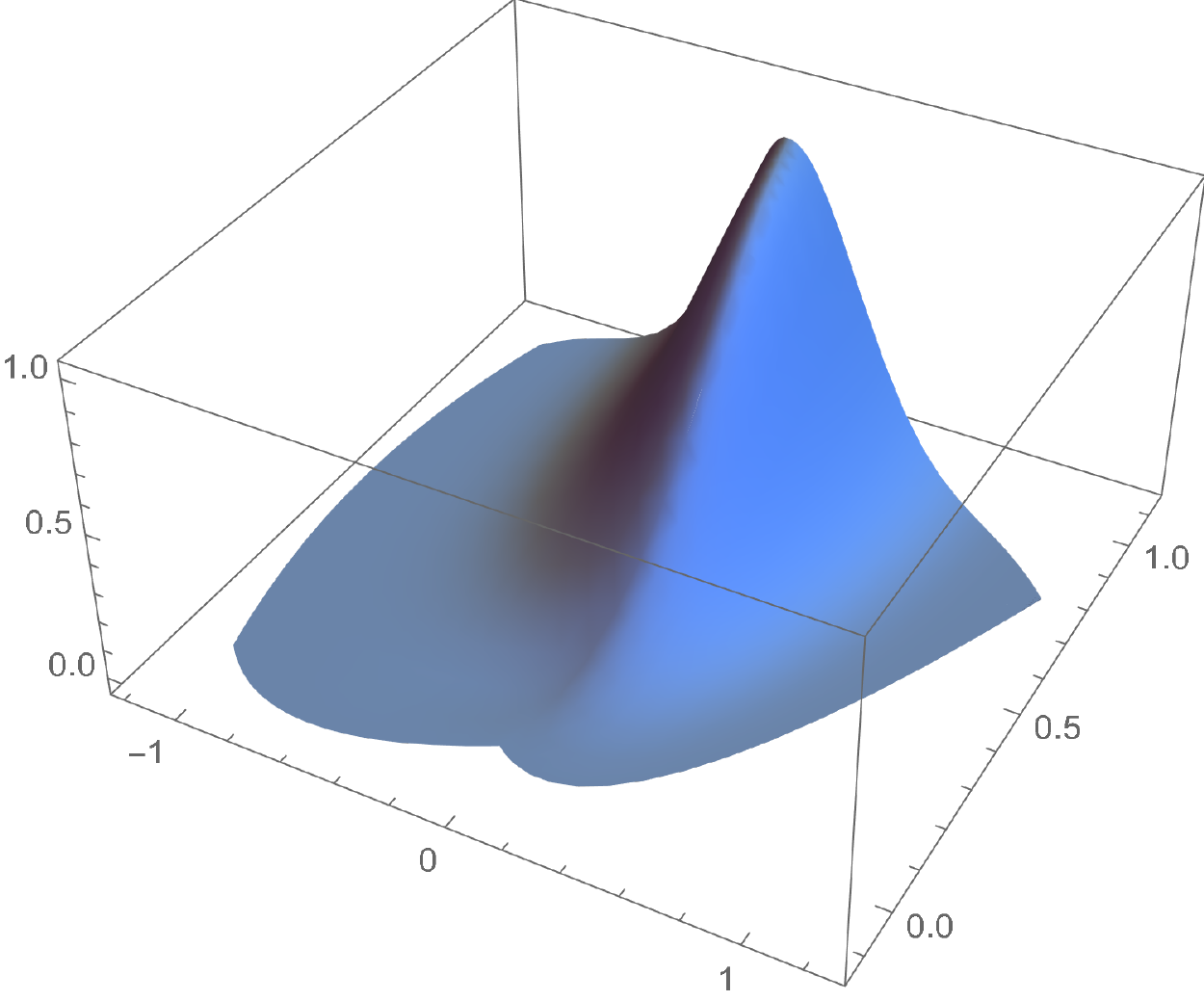}
\includegraphics[width=.24\textwidth]{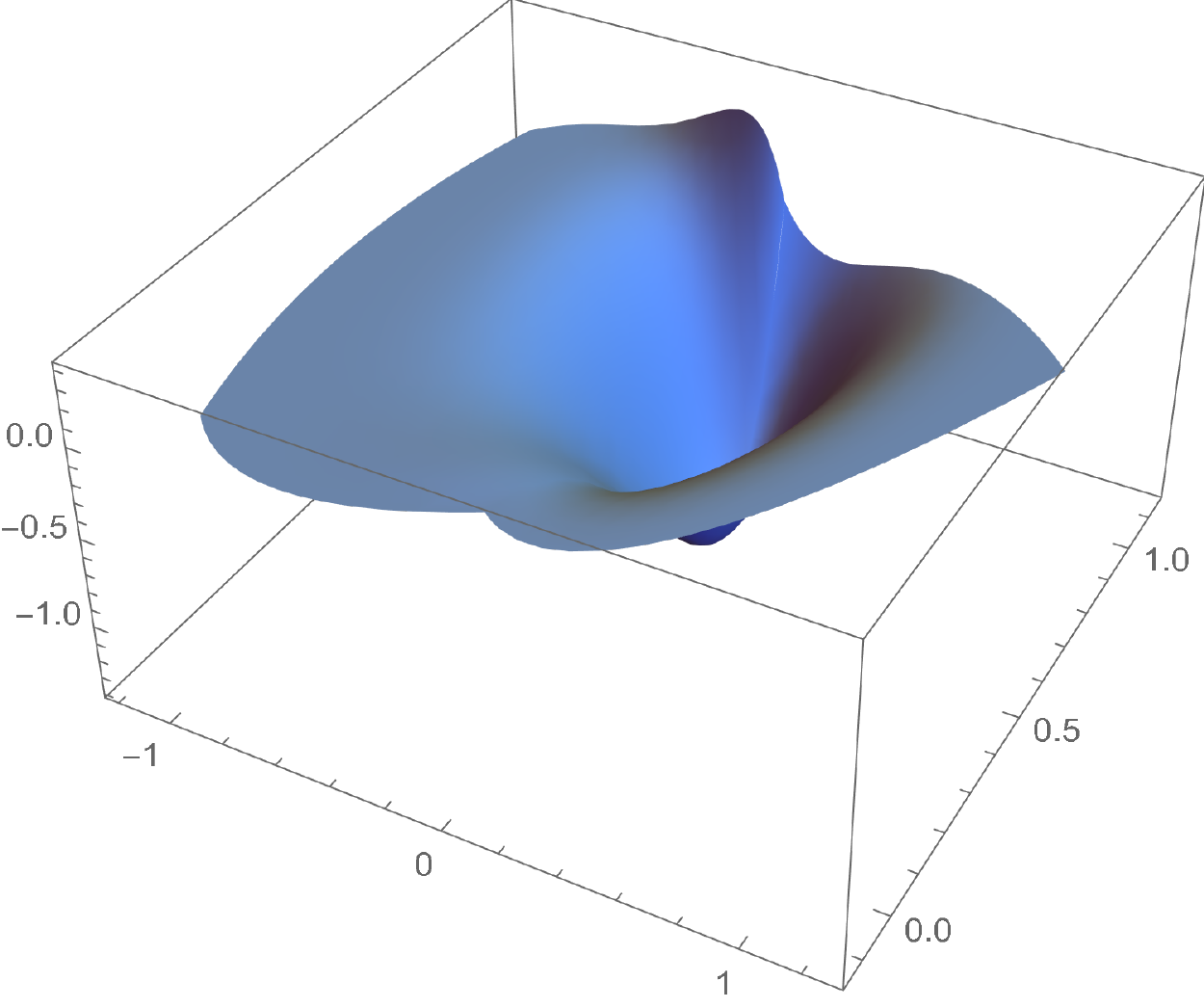}
\includegraphics[width=.24\textwidth]{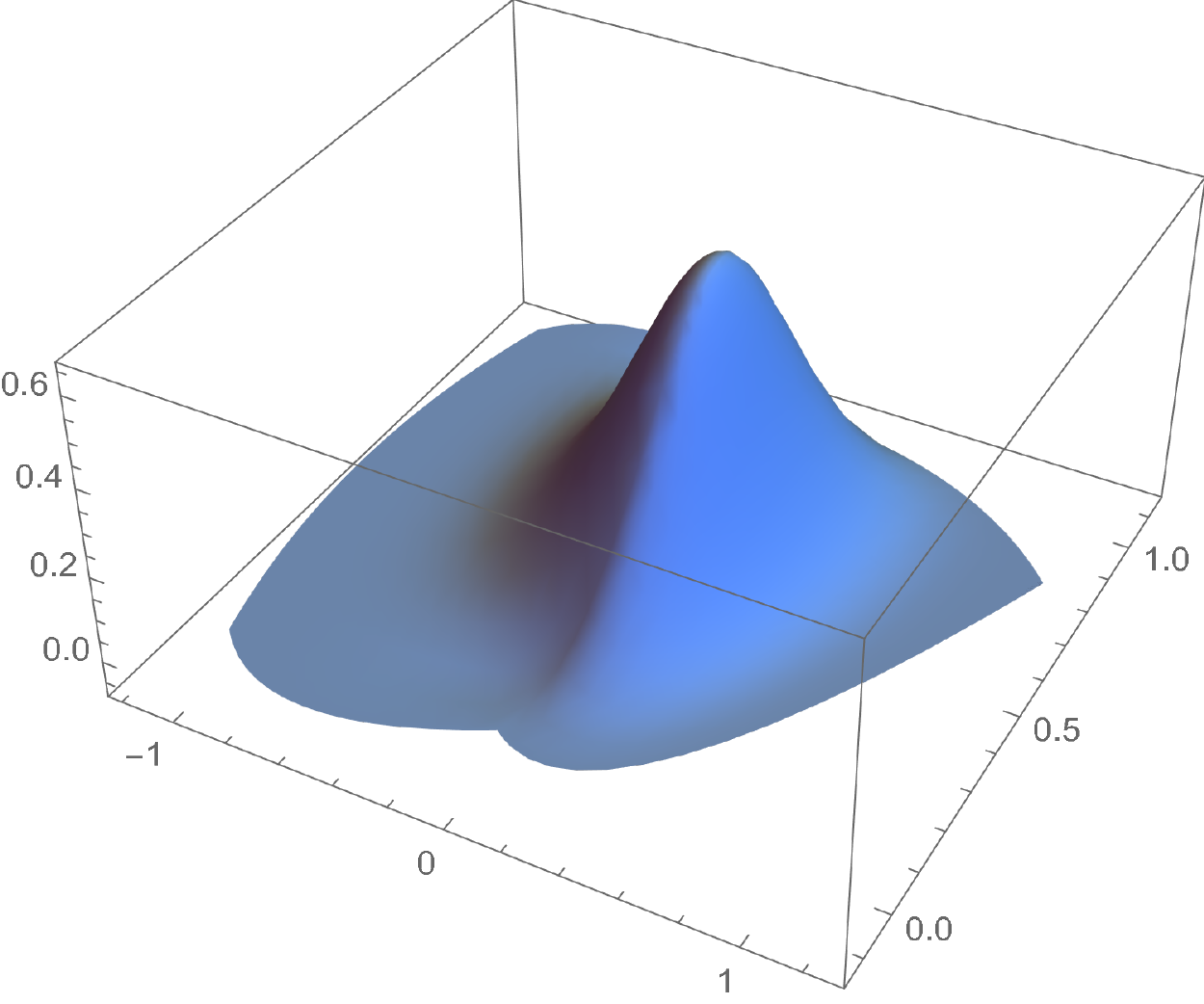}\\
\includegraphics[width=.24\textwidth]{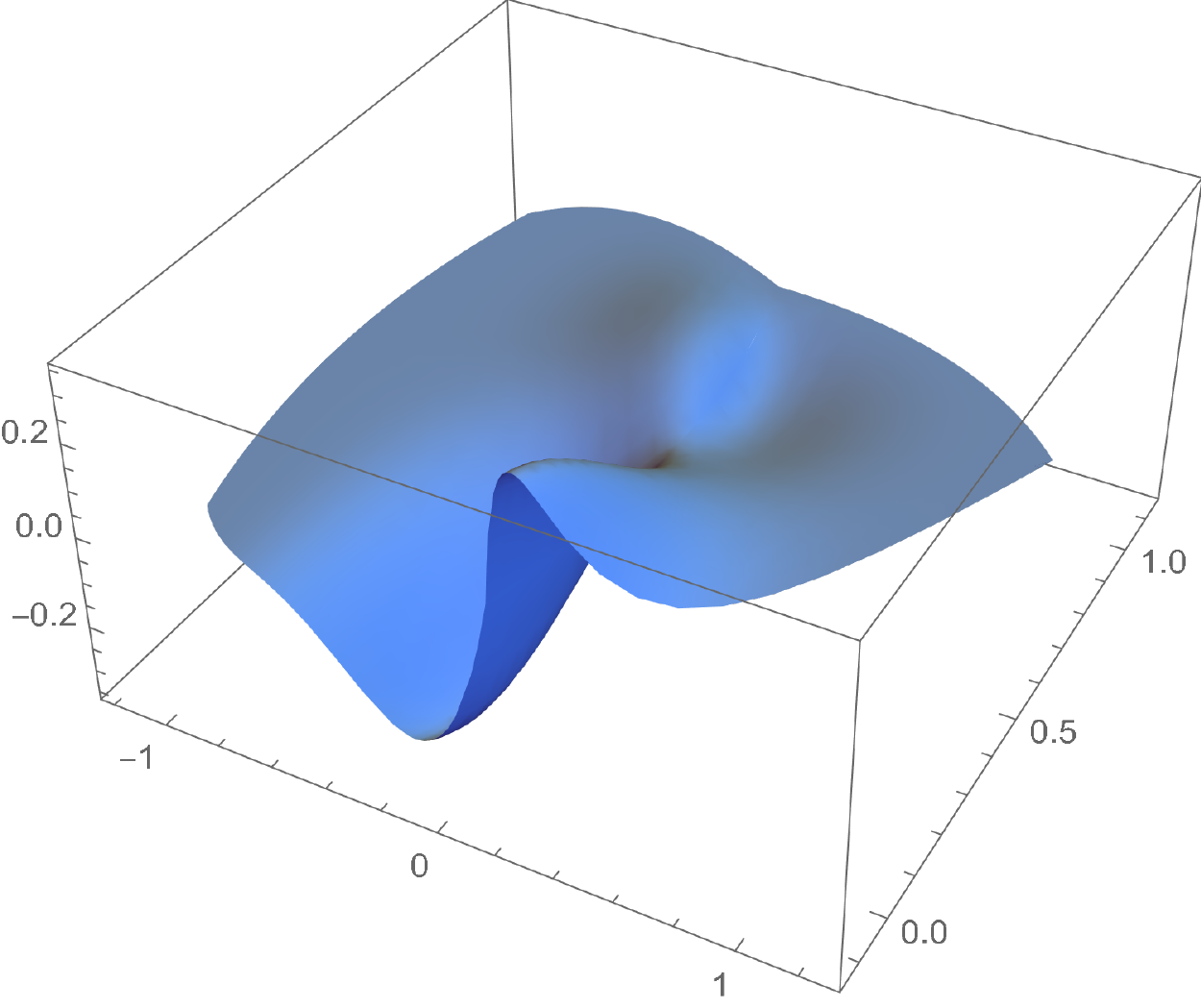}
\includegraphics[width=.24\textwidth]{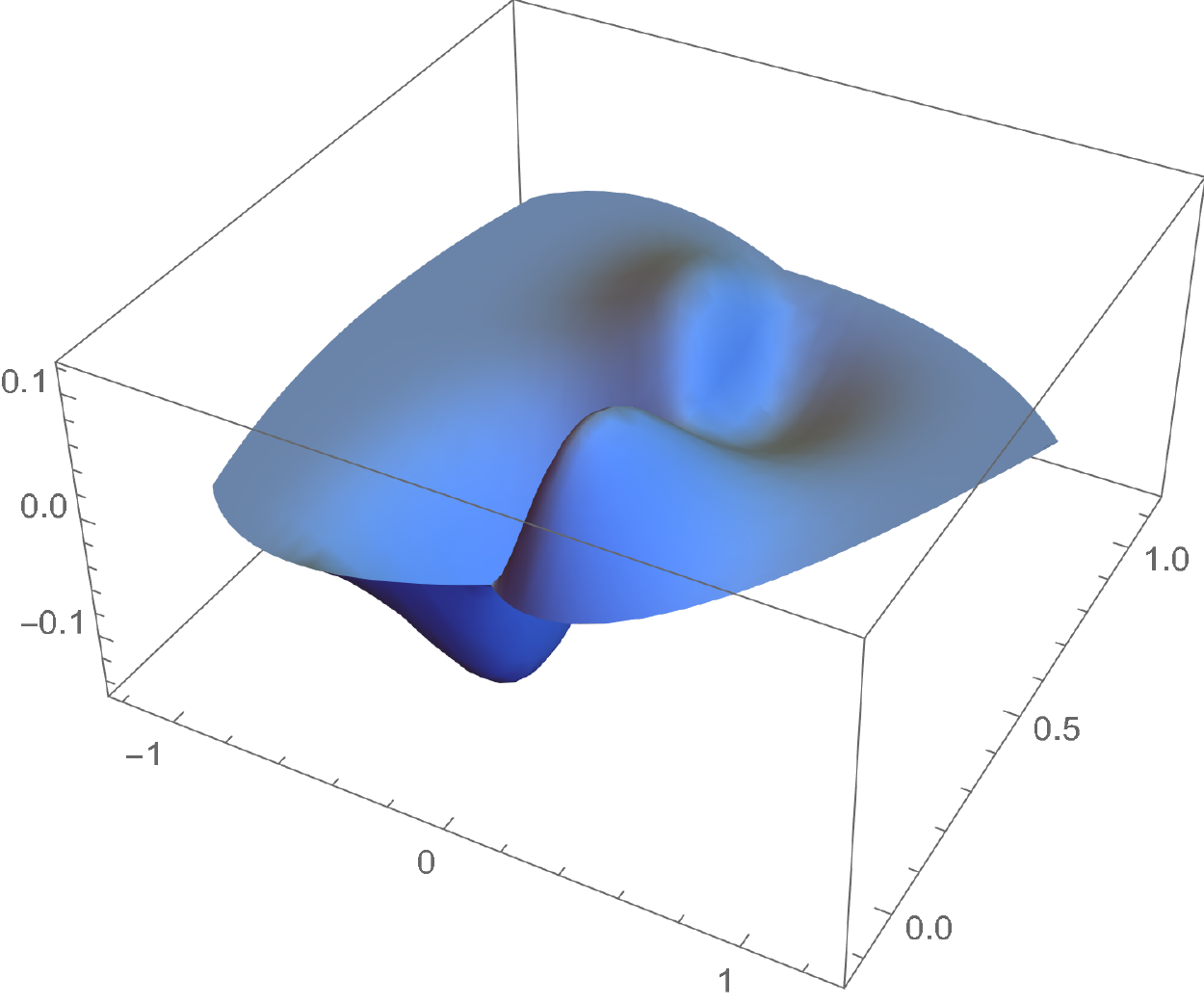}
\includegraphics[width=.24\textwidth]{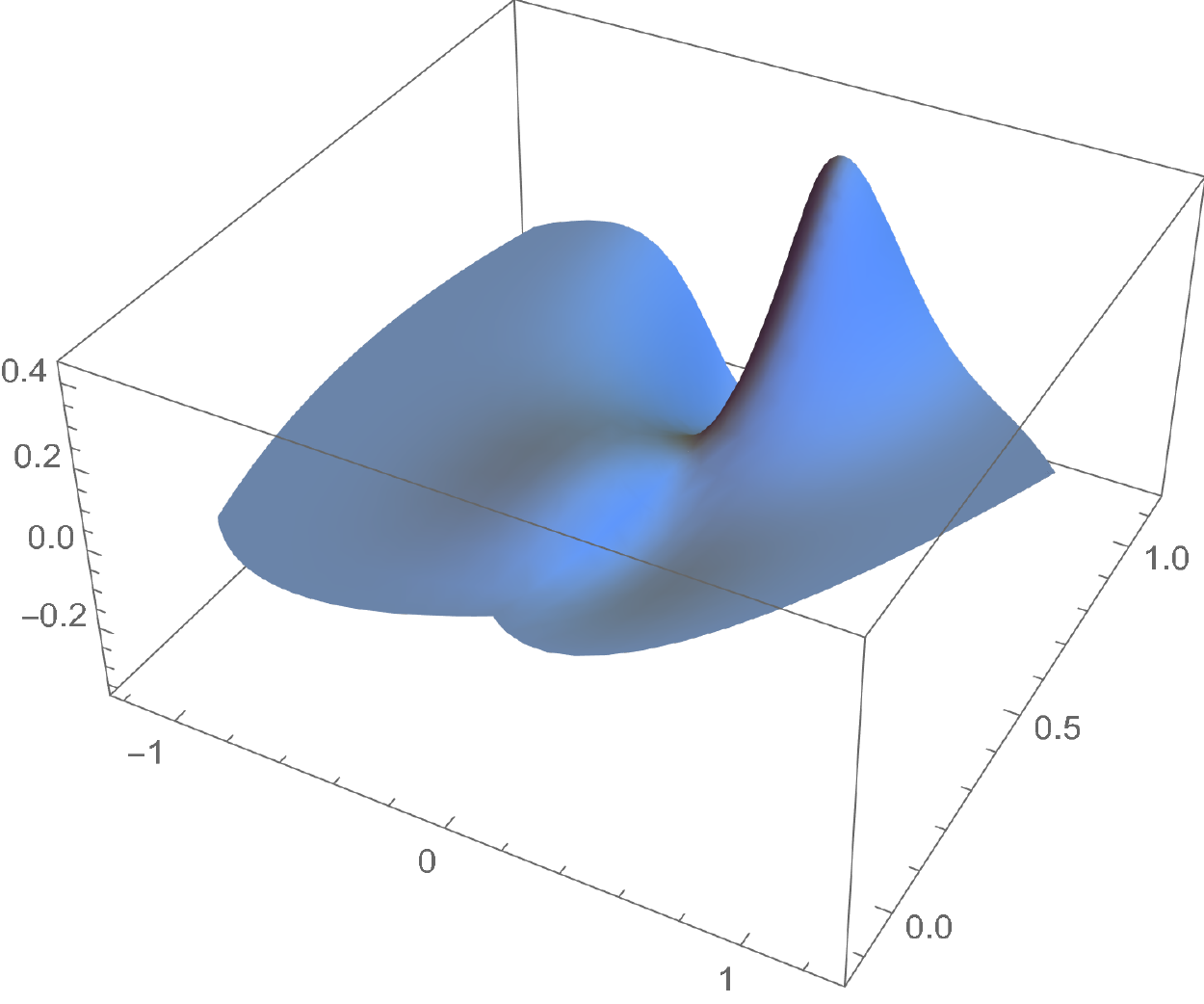}
\includegraphics[width=.24\textwidth]{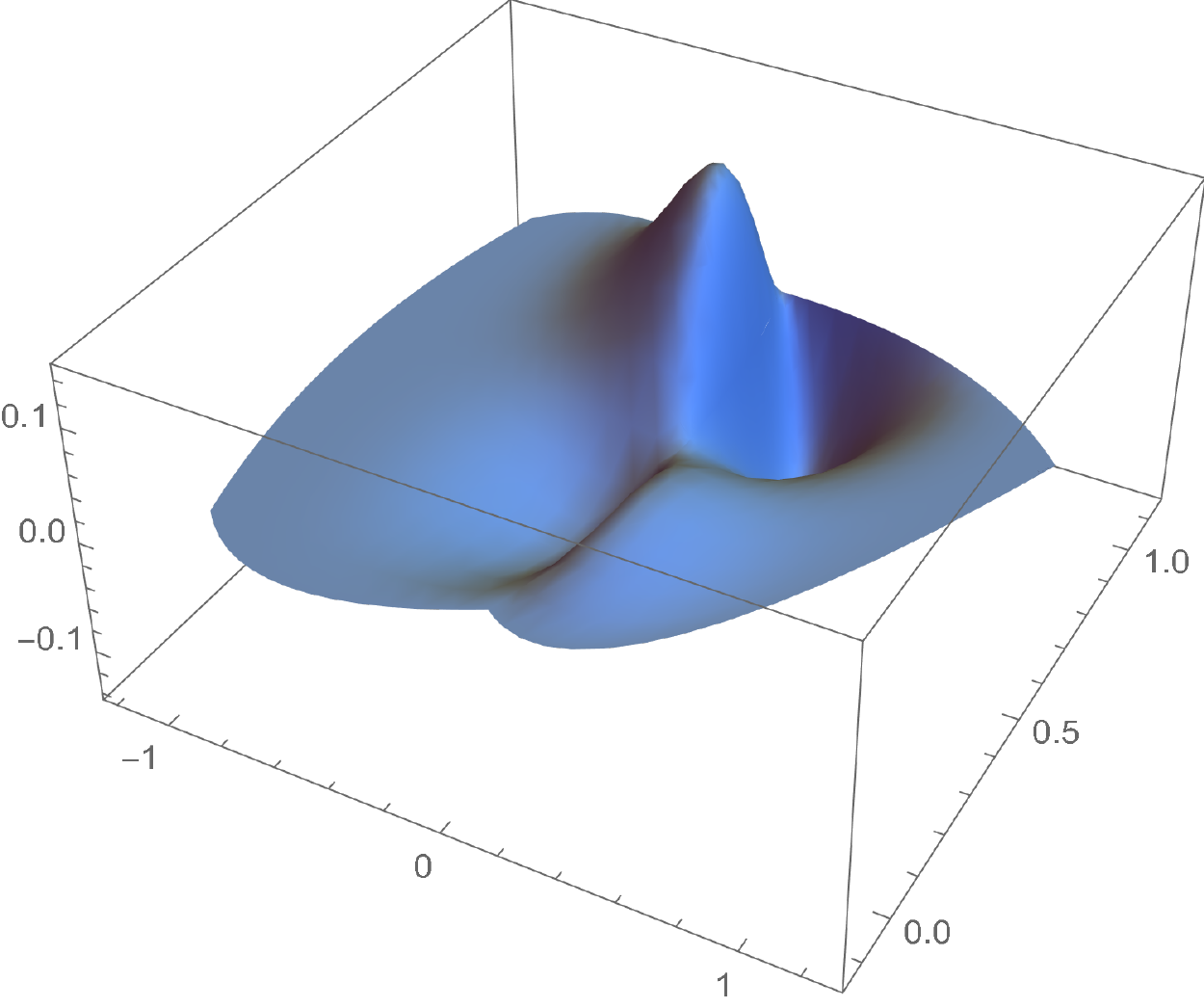}
\includegraphics[width=.24\textwidth]{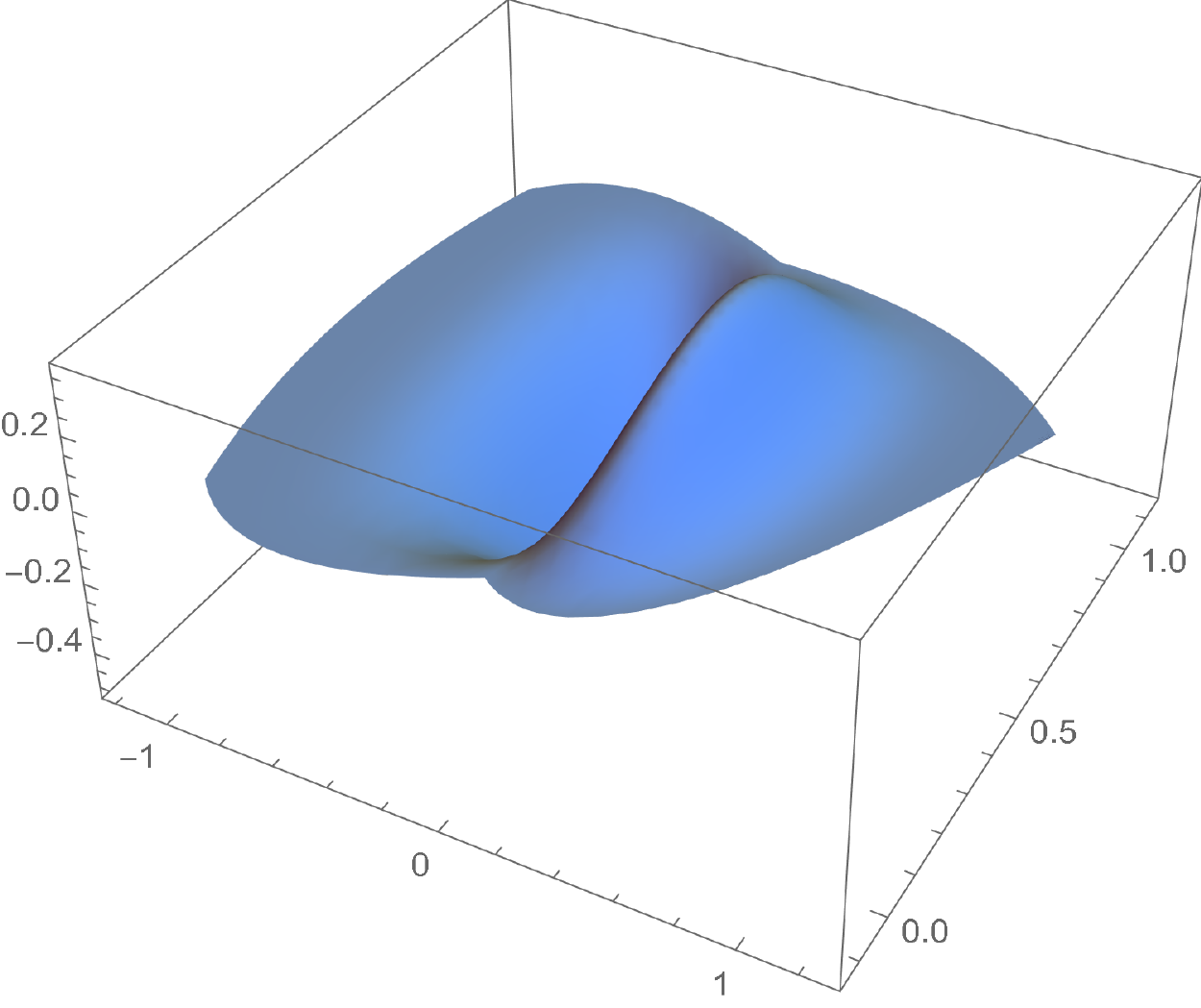}
\includegraphics[width=.24\textwidth]{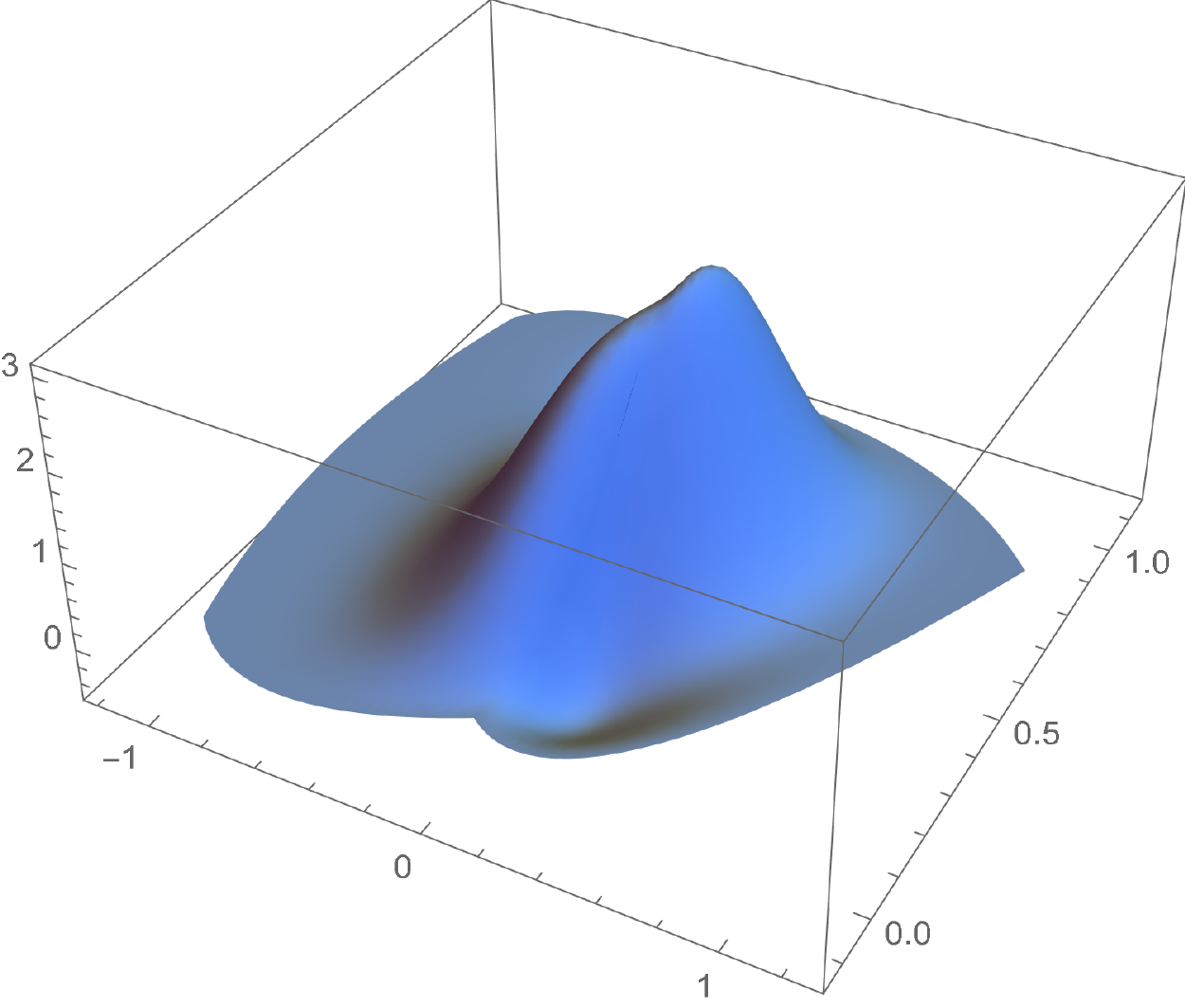}
\includegraphics[width=.24\textwidth]{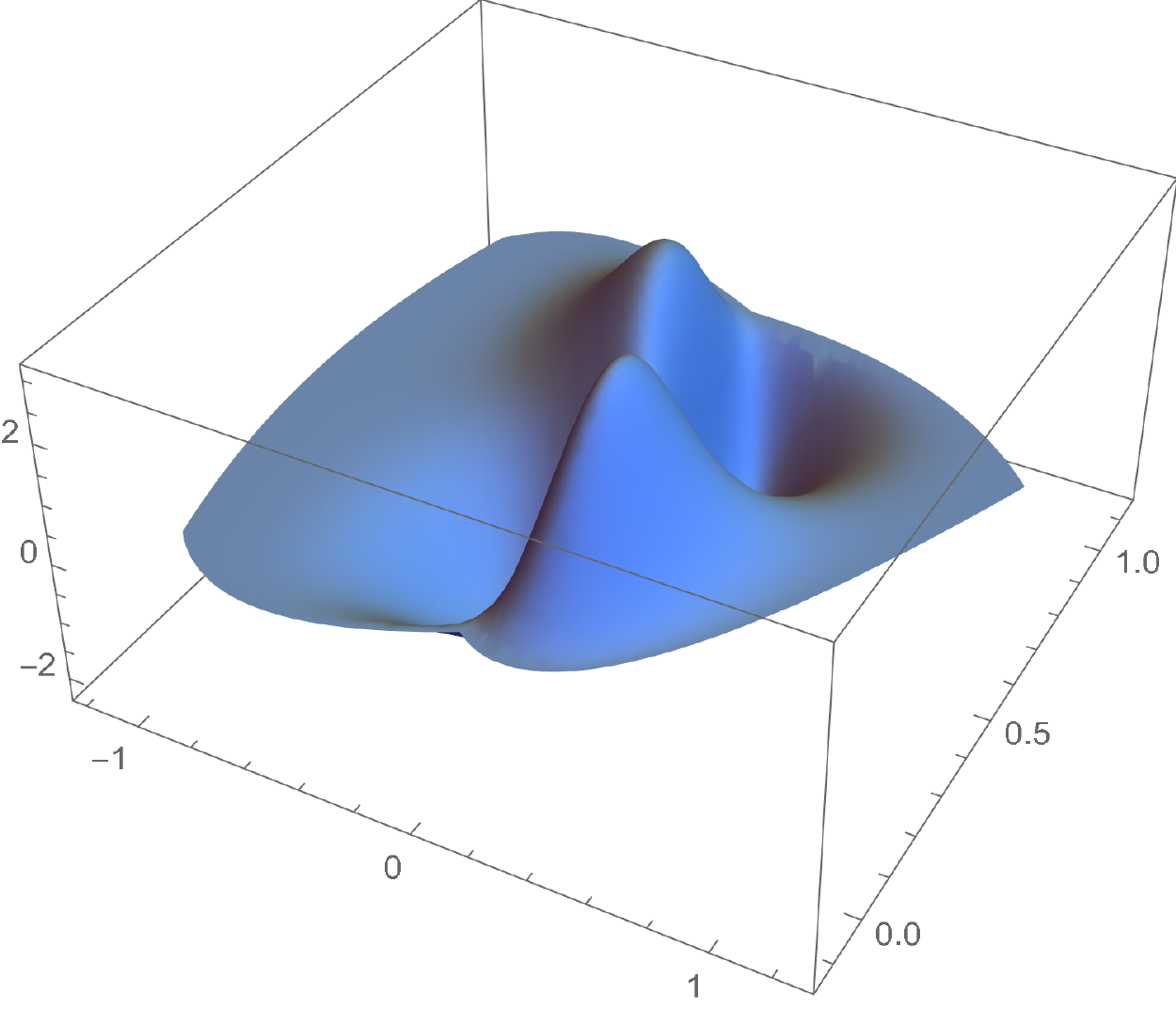}
\caption{Basis functions from Example~\ref{example-basisFunctions} for $\pd=7$.}
\label{fig:example4}
\end{figure}
\end{example}

\begin{example} \label{example-basisFunctions-2}
Choosing a mesh defined by \eqref{example-1-C1And2} and \eqref{example1-values-c} (see Figure~\ref{fig:example2}, left) we have $2\pd$ free parameters. In particular, we get one additional parameter for the construction of $\w$, i.e., $n_\w = d-1$. For $d=6$ we choose the interpolations functionals by \eqref{eq-IntFun-PI-even-d} where the last two functionals are replaced by  $\lambda_{1/4}^{(0)} \w$, $\lambda_{1/2}^{(0)} \w$, $\lambda_{3/4}^{(0)} \w$. The condition number of the collocation matrix is $238.96$, so numerical computations needed to solve the interpolation problem are stable.
\end{example}
 
One may expect that the basis computation is not stable in a configuration that is close to the special case where the dimension changes. However, this is not the case when we consider the parabolic interface. Namely, additional degrees of freedom are obtained if $ {\rm gcd}(\wt{\alpha}_1,\wt{\alpha}_2)$ is a nonconstant polynomial or if $\wt{\alpha}_\ell < 2+\sigma_\ell$, $\ell=1,2$.
This implies that $\pdw > \pd-3$ or equivalently that one can take  $\w^*_{low}$ of degree greater than $\pd-3$. This does not cause any instability when the parameterization is close to the special case. To illustrate this numerically let us construct configurations that are close to the setting from Example~\ref{example-basisFunctions-2} (but which remain generic configurations as in Example~\ref{example-basisFunctions}). Namely, we 
choose $x_1=\frac{9}{10}$, $y_1 = \frac{14}{25}$, $x_2=-\frac{1}{2}+\varepsilon_x, y_2 =\frac{41}{100}+\varepsilon_y$, where $\varepsilon_x,\varepsilon_y$ are random, non-zero rational numbers, where the numerator is selected from the interval $[-100,100]$ and the denominator from $[10^{12},2\cdot 10^{12}]$. We fix the degree to $d=6$ and observe that the condition numbers of the collocation matrices corresponding to interpolation functionals \eqref{eq-IntFun-PI-even-d} are always $40.35$, as in the generic case from Example~\ref{example-basisFunctions}. This shows that the computation of the basis through interpolation is stable in this configuration. Clearly, if we compute in a floating point arithmetic, the greatest common divisor and the degrees of $\wt{\alpha}_\ell$ cannot be determined exactly but only up to
some prescribed precision. Thus also the $C^1$ continuity conditions are satisfied only up to some precision.

\section{Conclusions}
\label{sec:Conclusion}
 
We investigated the $\C{1}$-smooth isogeometric spline space of general polynomial degree~$\pd \geq \delta$ over planar domains partioned into two elements, where each element can be a  B\'{e}zier triangle or a B\'{e}zier quadrilateral of (bi-)degree~$\delta \geq 1$. To fully explore the $\C{1}$-smooth isogeometric spline space, a theoretical framework was developed. It was used to analyze the $\C{1}$-smoothness conditions of the functions across the interface of the two elements and to study the representation of the functions in the neighborhood of the interface, more precisely, to study traces and normal derivatives along the interface. In case of $\delta=2$, i.e., in case of quadratic triangles and biquadratic quadrilaterals, we further provide for all possible configurations of the two mesh elements the exact dimension count as well as a basis construction of the $\C{1}$-smooth isogeometric spline space. The obtained results were demonstrated in detail for several examples of interesting configurations of the two mesh elements. We moreover provide a first study on the stability of the basis computation near a special case. A more exhaustive stability analysis is planned in the future. We also want to derive a algorithm to compute a stable subspace (a subspace of the complete $\C{1}$-smooth space, whose dimension and degree-of-freedom structure is independent of the geometry), for which certain approximation properties can be shown.

This paper is an important preliminary step to analyze the space of $\C{1}$-smooth isogeometric spline functions over a planar mixed (bi-)quadratic mesh composed of multiple triangles and quadrilaterals and to study the local polynomial reproduction properties of such a space. Moreover, the presented work is the basis for the surface case by using at least quadratic triangular and biquadratic quadrilateral surface patches. Beside these two topics for future research, we also plan to use the $\C{1}$-smooth isogeometric spline space to solve fourth order PDEs such as the biharmonic equation, the Kirchhoff--Love shell problem, problems of strain gradient elasticity or the Cahn--Hilliard equation over mixed \mbox{(bi-)quadratic} triangle and quadrilateral meshes.

\section*{Acknowledgments}

This paper was developed within the Scientific and Technological Cooperation ``Smooth splines over mixed triangular and quadrilateral meshes for numerical simulation'' between Austria and Slovenia 2023-24, funded by the OeAD under grant nr. SI 17/2023 and by ARRS bilateral project nr. BI-AT/23-24-018. 

The research of M. Kapl is partially supported by the Austrian Science Fund (FWF) through the project P~33023-N. The research of M.~Knez is partially supported by the research program P1-0288 and the research projects J1-3005 and N1-0137 from ARRS, Republic of Slovenia. The research of J.~Gro\v{s}elj is partially supported by the research program P1-0294 from ARRS, Republic of Slovenia. The research of V.~Vitrih is partially supported by the research program P1-0404 and research projects N1-0296, J1-1715, N1-0210 and J1-4414 from ARRS, Republic of Slovenia. This support is gratefully acknowledged.

\end{document}